\documentclass[10pt]{amsart}

\usepackage{amssymb,graphicx}
\usepackage{amsfonts}
\usepackage{amsmath}
\usepackage{amsthm}
\usepackage{fancyhdr}
\usepackage{indentfirst}
\usepackage{hyperref}
\usepackage{comment}
\usepackage{color}
\usepackage{subcaption}
\usepackage{epsfig}
\usepackage{pst-grad} 
\usepackage{pst-plot} 
\usepackage[space]{grffile} 
\usepackage{etoolbox} 
\usepackage{mathrsfs} 

\newtheorem{theorem}{Theorem}[section]
\newtheorem{lemma}[theorem]{Lemma}

\theoremstyle{definition}
\newtheorem{definition}[theorem]{Definition}

\newtheorem{corollary}[theorem]{Corollary}

\newtheorem{conjecture}[theorem]{Conjecture}
\newtheorem{proposition}[theorem]{Proposition}
\newtheorem{question}{Question}

\theoremstyle{remark}
\newtheorem{remark}[theorem]{Remark}

\numberwithin{equation}{section}

\newcommand{\al}{\alpha}
\newcommand{\be}{\beta}
\newcommand{\V}{\mathcal{V}}
\newcommand{\RV}{\mathcal{RV}}
\newcommand{\IV}{\mathcal{IV}}
\newcommand{\R}{\mathbb{R}}
\newcommand{\N}{\mathbb{N}}
\newcommand{\mH}{\mathcal{H}}
\newcommand{\F}{\mathcal{F}}

\newcommand{\A}{\mathcal{A}}

\newcommand{\C}{\mathcal{C}}

\newcommand{\lan}{\langle}
\newcommand{\ran}{\rangle}
\newcommand{\weakto}{\rightharpoonup}
\newcommand{\la}{\lambda}
\newcommand{\lc}{\llcorner}
\newcommand{\La}{\Lambda}
\newcommand{\si}{\sigma}
\newcommand{\Si}{\Sigma}
\newcommand{\de}{\delta}

\newcommand{\ep}{\epsilon}
\newcommand{\pr}{\prime}
\newcommand{\Om}{\Omega}
\newcommand{\om}{\omega}
\newcommand{\Ga}{\Gamma}
\newcommand{\ga}{\gamma}

\newcommand{\ti}{\tilde}

\newcommand{\Z}{\mathcal{Z}}

\newcommand{\f}{\mathbf{f}}
\newcommand{\m}{\mathbf{m}}
\newcommand{\M}{\mathbf{M}}
\newcommand{\Gr}{\mathbf{G}}
\newcommand{\bL}{\mathbf{L}}
\newcommand{\bV}{\mathbf{V}}
\newcommand{\bRV}{\mathbf{RV}}
\newcommand{\bIV}{\mathbf{IV}}

\newcommand{\mf}{\mathbf{f}}
\newcommand{\D}{\mathcal{D}}

\newcommand{\mF}{\mathbf{F}}

\newcommand{\mR}{\mathcal{R}}
\newcommand{\bmu}{\boldsymbol \mu}
\newcommand{\btau}{\boldsymbol \tau}
\newcommand{\bleta}{\boldsymbol \eta}
\newcommand{\n}{\mathbf{n}}

\newcommand{\md}{\mathbf{d}}
\newcommand{\sA}{\mathscr{A}}

\newcommand{\ttimes}{\scalebox{0.7}{$\mathbb{X}$}}

\newcommand{\tM}{\widetilde{M}}

\newcommand{\tE}{\widetilde{\exp}}

\newcommand{\tBcal}{\widetilde{\mathcal{B}}}
\newcommand{\tAcal}{\widetilde{\mathcal{A}}}
\newcommand{\tScal}{\widetilde{\mathcal{S}}}

\newcommand{\tS}{\widetilde{S}}
\newcommand{\tir}{\widetilde{r}}

\newcommand{\Area}{\textrm{Area}}

\newcommand{\VarTan}{\operatorname{VarTan}}

\newcommand{\spt}{\operatorname{spt}}
\newcommand{\dist}{\operatorname{dist}}
\newcommand{\Div}{\operatorname{div}}

\newcommand{\Hess}{\operatorname{Hess}}
\newcommand{\dom}{\operatorname{dom}}

\newcommand{\interior}{\operatorname{int}}
\newcommand{\Ric}{\operatorname{Ric}}
\newcommand{\Clos}{\operatorname{Clos}}

\newcommand{\rom}[1]{\expandafter\romannumeral #1}

\begin{document}


\title[Min-max Theory for Free Boundary Minimal Hypersurfaces I]{Min-max Theory for Free Boundary Minimal Hypersurfaces I - Regularity Theory}

\author[Martin Li]{Martin Man-chun Li}
\address{Department of Mathematics, The Chinese University of Hong Kong, Shatin, N.T., Hong Kong}
\email{martinli@math.cuhk.edu.hk}

\author[Xin Zhou]{Xin Zhou}
\address{Department of Mathematics, University of California Santa Barbara, Santa Barbara, CA 93106, USA}
\email{zhou@math.ucsb.edu}

\begin{abstract}
In 1960s, Almgren \cite{Almgren62,Almgren65} initiated a program to find minimal hypersurfaces in compact manifolds using min-max method. This program was largely advanced by Pitts \cite{Pitts} and Schoen-Simon \cite{Schoen-Simon81} in 1980s when the manifold has no boundary. In this paper, we finish this program for general compact manifold with nonempty boundary. As a result, we prove the existence of a smooth embedded minimal hypersurface with free boundary in any compact smooth Euclidean domain. An application of our general existence result combined with the work of Marques and Neves \cite{Marques-Neves17} shows that for any compact Riemannian manifolds with nonnegative Ricci curvature and convex boundary, there exist infinitely many embedded minimal hypersurfaces with free boundary which are properly embedded.
\end{abstract}

\maketitle
\tableofcontents


\section{Introduction}
\label{S:intro}

\subsection{The problem and main results}

Minimal surfaces have been a central object of study in mathematics for centuries because of its natural appearance in a wide range of fields including geometry, analysis and partial differential equations. A major triumph in minimal surface theory in the twentieth century was the solution of the longstanding Plateau problem by Douglas and Rado. 
Since then, there have been immense research activities on the existence and regularity of solutions to the Plateau problem and its variants. Other than earlier works of Gergonne in 1816 and H.A. Schwarz in 1890, Courant was the first mathematician who studied systematically the (partially and totally) free boundary problems for minimal surfaces in a series of seminal papers (see \cite[Chapter VI]{Courant}), 
(see also Davids \cite{Courant-Davids40}). A typical problem they considered is the following:

\begin{question}
\label{Q1}
Given a closed surface $S$ in $\R^3$ and a Jordan curve $\Gamma$ in $\R^3 \setminus S$ which is not contractible in $\R^3 \setminus S$, does there exist a surface $\Sigma$ of minimal (or stationary) area in the class of all disk type surfaces whose boundary curve lies on $S$ and is ``linked'' with $\Gamma$ (see \cite[p.213-218]{Courant})?
\end{question}

There are a few important points to note for \textbf{Question \ref{Q1}} above:
\begin{itemize}
\item[(1)] The hypotheses imply that genus$(S)>0$ by the generalized Schoenflies theorem (assuming $S$ is \emph{smoothly} embedded). For the genus zero case, the direct method would produce a point (of zero area) as a degenerate minimizer. 
\item[(2)] The minimal surface $\Sigma$ may penetrate the support surface $S$, which is unrealistic from the physical point of view. Along the boundary $\partial \Sigma$, the minimal surface $\Sigma$ has to meet $S$ \emph{orthogonally} (the free boundary condition).
\item[(3)] In general, the minimal surface $\Sigma$ is only \emph{immersed} with possibly some branch points (in the interior or at the free boundary).
\end{itemize}

Several important results were obtained in order to address the issues above. Using a dual Plateau construction, Smyth \cite{Smyth84} showed that if $S$ is the boundary of a tetrahedron (which is \emph{non-smooth}), there exists exactly three disk-type minimal surfaces embedded inside the tetrahedron solving the free boundary problem. For a smooth surface $S$ of genus zero, Struwe \cite{Struwe84} \cite{Struwe88} established 
the existence of at least one \emph{unstable} disk-type solution to the free boundary problem. On the other hand, Fraser \cite{Fraser00} used Sacks-Uhlenbeck's perturbed energy method to prove the existence of free boundary minimal disks in any codimension. However, the solutions may only be immersed (possibly with branch points), and penetrate the support surface $S$. For applications in three-dimensional topology, Meeks and Yau \cite{Meeks-Yau80} proved the existence of \emph{embedded} minimizing disks in a compact $3$-dimensional Riemannian manifold with boundary which is (mean) \emph{convex} and not simply connected. Later on, Gr\"{u}ter and Jost \cite{Gruter-Jost86a} proved the existence of an unstable embedded minimal disk inside any \emph{convex} domain bounded by $S \subset \R^3$, based upon geometric measure theoretic techniques developed by Allard, Almgren, Pitts, Simon and Smith.

In 1960s, 
Almgren \cite{Almgren65} initiated an ambitious program to develop a variational calculus in the large for minimal submanifolds in Riemannian manifolds of any dimension and codimension, generalizing Morse' theory for geodesics. The setup in \cite{Almgren65} is completely general (i.e. without any curvature assumptions) and includes the case 
with (partially fixed or free) boundary. For the free boundary problem in particular, Almgren proposed the following question:

\begin{question}[Constrained Free Boundary Problem]
\label{Q2}
Given any compact Riemannian manifold $(M^{n+1},g)$ with boundary $\partial M \neq \emptyset$, does there exists a smooth, embedded $k$-dimensional minimal submanifold $\Sigma$ \emph{contained in $M$} with boundary $\partial \Sigma \subset \partial M$ solving the free boundary problem?
\end{question}

The free boundary problem is \emph{constrained} in the sense that the minimal submanifold $\Sigma$ we seek has to lie completely inside $M$. This imposes substantial difficulty to the problem as seen in (2) above. Moreover, as already remarked in \cite{Tomi86} (see also \cite{Li15}), the minimal submanifold $\Sigma$ will possibly touch $\partial M$ along interior portions of arbitrary size. We refer to this touching phenomenon as \emph{non-properness} (see Definition \ref{D:proper}). Historically, this issue was avoided at the expense of putting certain convexity assumptions on the boundary $\partial M$ (see for example \cite{Meeks-Yau80} \cite{Gruter-Jost86a} \cite{Fraser02}). However, \textbf{Question \ref{Q2}} in its complete generality remains unsolved for over 50 years.

In contrast, the analogous problem to \textbf{Question \ref{Q2}} where $\partial M=\emptyset$ (which we refer to as the ``\emph{closed}'' case) was well studied and substantial progress has been made. Using his foundational work on the homotopy groups of integral cycles \cite{Almgren62}, Almgren \cite[Theorem 15.1]{Almgren65} proved the existence of a non-trivial weak solution (as stationary varifolds) in any dimension and codimension, and he showed that such a solution is at least \emph{integer rectifiable}. Higher regularity was established in the codimension-one (i.e. $k=n$) case by the seminal work of Pitts \cite{Pitts} (for $2 \leq n \leq 5$) and later extended by Schoen-Simon \cite{Schoen-Simon81} (for $n \geq 6$). 
Very recently, Marques and Neves revisited the Almgren-Pitts min-max theory (see for example \cite{Marques-Neves-survey}) and gave surprising applications to solve a number of longstanding conjectures in geometry including the Willmore conjecture \cite{Marques-Neves14}, the Freedman-He-Wang conjecture \cite{Agol-Marques-Neves16} and Yau's conjecture \cite{Marques-Neves17} in the case of positive Ricci curvature. Due to these tremendous success, there have been a vast number of exciting developments on this subject (see for example \cite{Liokumovich-Marques-Neves} \cite{Marques-Neves16} \cite{Zhou15} \cite{Zhou17}). 

For the constrained free boundary problem, it was claimed in \cite[Theorem 4.1]{Jost86} that \textbf{Question \ref{Q2}} was solved for the case $k=n=2$ when $M$ is any compact domain in $\R^3$ diffeomorphic to a ball whose boundary is \emph{strictly mean convex}. In the same paper, it was also claimed that similar results hold without the mean convexity assumption \cite[Theorem 4.2]{Jost86}. However, it was not valid (in particular, Lemma 3.1 in \cite{Jost86} was incorrect) as the touching phenomenon has not been accounted for. An important attempt to address this issue was carried out by the first author in \cite{Li15}. Unfortunately, there was still an error in Lemma 3.5 of \cite{Li15} as \emph{properness} may not be preserved in the limit so the result in \cite{Li15} only covers the cases where properness is not lost throughout the constructions. Assuming boundary convexity, \textbf{Question \ref{Q2}} was solved for any dimension $k=n$ by De Lellis and Ramic \cite{DeLellis-Ramic} very recently.

In this paper, we give an affirmative answer to \textbf{Question \ref{Q2}} in the codimension-one case in complete generality. There is no assumptions at all on the topology of $M$ nor the curvatures of $\partial M$. Our key novel finding is that the minimal hypersurface $\Sigma$ produced by our method is contained inside $M$ but may not be \emph{proper}. Nonetheless, $\Sigma$ is a \emph{smooth} embedded hypersurface with boundary lying on $\partial M$, even near the touching points which make $\Sigma$ improper. A precise statement of our main result is as follows.

\begin{theorem}
\label{T:manifold}
Let $2 \leq n \leq 6$. For any smooth compact Riemannian manifold $M^{n+1}$ with boundary $\partial M \neq \emptyset$, there exists a hypersurface $\Sigma$ with (possibly empty) boundary $\partial \Sigma$ lying on $\partial M$ such that 
\begin{itemize}
\item[(i)] $\Sigma$ is smoothly embedded in $M$ up to the boundary $\partial \Sigma$,
\item[(ii)] $\Sigma$ is minimal in the interior of $\Sigma$,
\item[(iii)] $\Sigma$ meets $\partial M$ orthogonally along its boundary $\partial \Sigma$.
\end{itemize}
\end{theorem}

\begin{remark}
\begin{itemize}
\item[(a)] In the important special case that $M$ is a compact domain in $\R^{n+1}$, $\Sigma$ must have \emph{non-empty} boundary as there is no closed minimal hypersurface in $\R^{n+1}$. In general, it is possible that $\Sigma$ is a closed (i.e. compact without boundary) embedded minimal hypersurface in $M$.
\item[(b)] $\Sigma$ is smooth and 
embedded even at points on the touching portion $\Sigma \cap \partial M$ which does not lie on the boundary of $\Sigma$. Therefore it makes sense to talk about minimality at such points by vanishing of the mean curvature.
\item[(c)] There is no topological or geometric assumptions made for the ambient manifold $M$ and its boundary $\partial M$, in sharp contrast to all the previous results by Courant-Davids \cite{Courant-Davids40}, Meeks-Yau \cite{Meeks-Yau80}, Gr\"{u}ter-Jost \cite{Gruter-Jost86a} and De Lellis-Ramic \cite{DeLellis-Ramic}.
\item[(d)] The restriction on the dimension $2 \leq n \leq 6$ comes from the regularity theory for stable free boundary minimal hypersurfaces \cite{Li-Zhou-A}. However, we remark that the same result holds in higher dimensions by allowing a singular set of Hausdorff codimension at least seven as in \cite{Schoen-Simon81}. The details will appear in a forthcoming paper of the authors.
\end{itemize}
\end{remark}

The existence of a non-trivial weak solution to \textbf{Question \ref{Q2}} was already covered in the work of Almgren in \cite{Almgren62} and \cite{Almgren65} (which holds in fact in any dimension and codimension) so the key merit of Theorem \ref{T:manifold} is on the \emph{regularity} of $\Sigma$. Ultimately the regularity would follow from certain ``\emph{almost minimizing property}'' in the same spirit as in the closed case \cite{Pitts}. Nonetheless, the free boundary problem is much more subtle as illustrated below:

\begin{itemize}
\item[(1)] Unlike in the closed case, where Birkhoff \cite{Birkhoff17} showed that any Riemannian two-sphere admits a non-trivial simple closed geodesic, there are examples of (non-convex) simply connected planar domains in $\R^2$ which do not contain any free boundary geodesics (even allowing part of the geodesic to lie on the boundary of the domain). A well-known example is given by Bos \cite{Bos63}. This shows that there may not exist any \emph{smooth} solution to \textbf{Question \ref{Q2}} when $k=n=1$.
\item[(2)] The weak solutions obtained by Almgren \cite{Almgren65} are varifolds which are stationary with respect to ambient deformations in $M$ which preserve $\partial M$ (as a set but not necessarily the identity map on $\partial M$). These stationary varifolds can be very different from being a minimal hypersurface in $M$, especially when touching phenomenon happens. In the extreme case, $\Sigma=\partial M$ is stationary by definition but it is not minimal in $M$ (unless $\partial M$ is minimal). Therefore, the weak solutions of Almgren are very far from the free boundary minimal hypersurfacs $\Sigma$ we are looking for in \textbf{Question \ref{Q2}}.
\item[(3)] In similar problems where solutions are required to be confined in a given region, e.g. the obstacle problem \cite{Caffarelli98}, the optimal regularity of the (minimizing) solutions is only $C^{1,1}$.
\end{itemize}

Another important contribution of this paper is that we have developed the min-max theory for free boundary minimal hypersurfaces in the general Almgren-Pitts setting, which is most suitable for use in geometric applications (for example \cite{Marques-Neves14}). We would like to point out that there are also important geometric applications of the min-max theory in the smooth setting, which is particularly useful in low dimensions (see for example \cite{Smith82} \cite{Colding-deLellis03} \cite{Li15} \cite{DeLellis-Tasnady13} \cite{DeLellis-Ramic} \cite{Colding-Gabai-Ketover}). As remarked in \cite{Marques-Neves17}, we have the following corollary directly from our main result Theorem \ref{T:manifold}.




\begin{corollary}
\label{C:1}
Let $2 \leq n \leq 6$. For any compact Riemannian manifold $(M^{n+1},g)$ with boundary, either
\begin{itemize}
\item[(i)] there exists a disjoint collection $\{\Sigma_1,\cdots,\Sigma_{n+1}\}$ of $(n+1)$ compact, smoothly embedded, connected free boundary minimal hypersurfaces in $M$; or
\item[(ii)] there exists infinitely many compact, smoothly embedded, connected free boundary minimal hypersurfaces in $M$.
\end{itemize}
\end{corollary}

Moreover, as the embedded Frankel property \cite[Definition 1.1]{Marques-Neves17} is satisfied when the ambient manifold $M$ has nonnegative Ricci curvature and strictly convex boundary \cite[Lemma 2.4]{Fraser-Li14}, (i) in Corollary \ref{C:1} never happens, so we have the following result. Note that when $\partial M$ is strictly mean-convex, the maximum principle excludes the touching phenomenon and thus the free boundary minimal hypersurface $\Sigma$ we obtained in Theorem \ref{T:manifold} is in fact \emph{properly} embedded.

\begin{corollary}
\label{C:2}
Let $2 \leq n \leq 6$ and $(M^{n+1},g)$ be a compact Riemannian manifold with non-empty boundary. If the Ricci curvature of $M$ is nonnegative and the boundary $\partial M$ is strictly convex, then $M$ contains an infinite number of distinct compact, smooth, properly embedded, free boundary minimal hypersurfaces.
\end{corollary}


Finally, we end this subsection with a note about the \emph{properness} of the free boundary minimal hypersurfaces $\Sigma$ we produced in Theorem \ref{T:manifold}. Since these hypersurfaces are produced by min-max constructions, we will call them free boundary \emph{min-max} minimal hypersurfaces in our discussions below. We conjecture that our result (Theorem \ref{T:manifold}) concerning properness is optimal in the following sense:
\begin{conjecture}
\label{Con: properness1}
There exists a compact smooth domain in $\mathbb{R}^{n+1}$ for which any (one-parameter) free boundary min-max minimal hypersurface is non-proper.
\end{conjecture}

A possible example for Conjecture \ref{Con: properness1} is shown in Figure \ref{nonproper minmax}. To explain why this may be a possible candidate for Conjecture \ref{Con: properness1}, we now give a brief introduction to the general setup of the min-max method modulo precise terminology in geometric measure theory.  Denote $\Om$ as a domain in $\R^{n+1}$, and $\mathcal{Z}_n(\Om, \partial \Om)$ as the space of hypersurfaces $S\subset \Om$ with $\partial S\subset\partial \Om$. A continuous map $\phi: [0, 1]\rightarrow \mathcal{Z}_n(\Om, \partial \Om)$ is called a {\em sweep-out} if the disjoint union $\sqcup\{\phi(x): x\in[0, 1]\}$ covers $\Om$ exactly once (up to cancellations counting orientation). Let $\Pi$ be the homotopy class of $\phi$, then one can define the width as: 
$$W(\Om)=\inf_{\psi\in\Pi}\max_{x\in[0, 1]} \Area(\psi(x)).$$
The width $W(\Om)$ measures the least area to sweep out $\Om$ by surfaces in $\mathcal{Z}_n(\Om, \partial \Om)$. The proof of Theorem \ref{T:manifold} asserts that $W$ is achieved by the area of a free boundary minimal hypersurface $\Si$, i.e. $W=\Area(\Si)$ (counting multiplicities). The domain $\Omega$ in Figure \ref{nonproper minmax} consists of two parts: $\Om_1$ is a long and slowly decaying solid cylinder capped off by a half ball, and $\Om_2$ is a solid half ball with a finger pushing inward on the left. When the decaying cylinder is long enough, we conjecture that the width $W(\Om_1)$ is achieved by the area of $\Si$ as in the figure. Then one can make $\Om_2$ small enough so that $W(\Om)=W(\Om_1)$ and $\interior (\Si)\cap\partial\Om=\{p\}$ is non-empty.  Finally it is conceivable that there is no other free boundary minimal hypersurface in $\Om$ whose area can achieve $W(\Om)$, and thus $\Si$ must be the unique solution produced by Theorem \ref{T:manifold}.

\begin{figure}[h]
    \centering
        \includegraphics[width=0.9\textwidth]{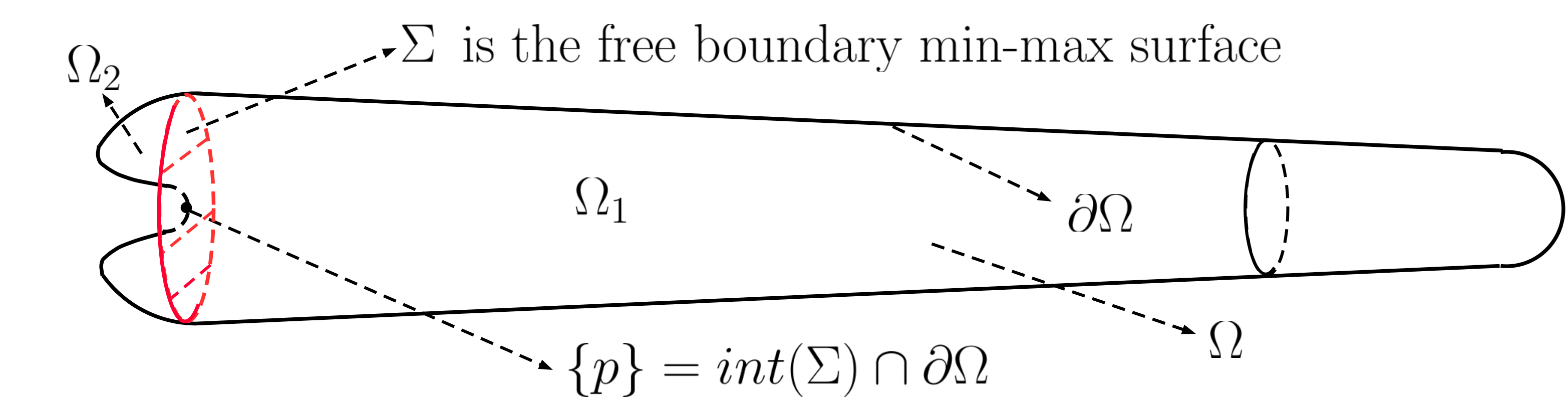}
    \caption{Conjectural domain in Conjecture \ref{Con: properness1}}
    \label{nonproper minmax}
\end{figure}

Although Conjecture \ref{Con: properness1} implies that one does not expect to find \emph{properly} embedded free boundary minimal hypersurfaces in \emph{every} compact domain $\Omega \subset \R^{n+1}$, we however conjecture that for \emph{generic} domains properness can be achieved.

\begin{conjecture}
\label{Con: properness2}
For a generic compact domain $\Om \subset \R^{n+1}$ with analytic boundary $\partial \Omega$, any free boundary min-max minimal hypersurface must be proper. 
\end{conjecture}

Here is why we believe Conjecture \ref{Con: properness2} may be true. By genericity, we can assume that $\partial \Omega$ contains no open subset which is a minimal hypersurface in $\R^{n+1}$. Suppose $\Sigma \subset \Om$ is a free boundary min-max minimal hypersurface which is not proper, then the touching set $\interior (\Si)\cap \partial \Om$ is very small (of measure zero) by the analyticity of $\partial\Om$. Once we know that the touching set is small enough, in principle it could be perturbed away by further generic deformations of $\partial \Om$. For example, in Figure \ref{nonproper minmax} the min-max hypersurface would be non-proper only if the tip of the finger from $\Omega_2$ is exactly touching $\Sigma$. It is easy to see that if the finger is a little longer, then $\Sigma$ would be a solution of annulus type, while if the finger is a bit short, then $\Sigma$ would be still a disk, but in either case the solution would be properly embedded.

\subsection{Main ideas of the proof} 
We now describe the outline of the proof of our main theorem (Theorem \ref{T:manifold}). The proof can be roughly divided into two parts. First, we use a tightening argument to establish the existence of solutions in a weak sense called stationary varifolds with free boundary. Second, we prove the regularity of our weak solution obtained by min-max construction utilizing the almost minimizing property. Although our general philosophy follows the ideas developed by Almgren \cite{Almgren65} and Pitts \cite{Pitts}, many new insights are needed to establish the full regularity near the boundary, as we will now discuss.

Since free boundary minimal hypersurfaces are critical points to the area functional with respect to deformations in $M$ which preserves $\partial M$ as a set, in our tightening argument we have to restrict to the class of ambient vector fields on $M$ which are tangential to $\partial M$. Varifolds in $M$ which are stationary with respect to such vector fields are called {\em stationary varifolds with free boundary}. However, these stationary varifolds may not have the desired properties at $\partial M$. For example, any constant multiple of a connected component of $\partial M$ is stationary even though it can be nothing like a minimal hypersurface in $M$. Therefore, one is not expected to be able to prove strong regularity results just from these first order considerations. Nonetheless, in the codimension one case, we were able to prove a maximum principle for general stationary varifolds with free boundary (see Theorem \ref{T:max-principle}) inside a convex domain whose boundary is orthogonal to $\partial M$. Such domains can be constructed locally near a point on $\partial M$ by taking images of sufficiently small balls under the Fermi coordinate systems (Lemma \ref{L:Fermi-convex}). Note that an optimal maximum principle was recently obtained in \cite{max-principle} which holds for free boundary minimal varieties in arbitrary codimension.

A serious difficulty towards proving regularity at points lying on the boundary $\partial M$ is that it is impossible to distinguish by measure theoretic means whether such a point on the support of a stationary varifold $V$ is a ``boundary point'' or an ``interior point'' touching $\partial M$ tangentially. For example, one can consider a multiplicity $1$ line touching $\partial M$ tangentially and a multiplicity $2$ half-line hitting $\partial M$ orthogonally. Both of them have density equal to $1$ at the boundary intersection point. However, if we \emph{aprior} know that the stationary varifold is supported on a \emph{smooth} embedded hypersurface contained in $M$ with boundary lying on $\partial M$, then there is a clear distinction between a \emph{true} boundary point of the smooth hypersurface and a \emph{false} boundary point which is just an interior point touching $\partial M$ tangentially. In the smooth case we can talk about the first and second variations, hence the notion of \emph{stability} makes sense without ambiguity. A key ingredient in our proof of regularity is the compactness theorem for stable minimal hypersurfaces with free boundary (Theorem \ref{T:freebdy-cpt}) which can have interior points touching $\partial M$ tangentially. The theorem follows from the uniform curvature estimates in \cite{Li-Zhou-A} and plays a crucial role in the regularity of the replacements constructed in Proposition \ref{P:good-replacement-property}.

To obtain a weak solution satisfying the crucial almost minimizing property in the min-max construction, we need to choose the total variational space as the space of equivalence classes of relative cycles. Geometrically the use of equivalence classes is very natural in our context, as we automatically ``forget" the part of the relative cycles lying on the boundary $\partial M$. Almgren also considered the min-max construction using equivalence classes in \cite{Almgren62}; however the relative cycles in \cite{Almgren62} are chosen to be integral cycles, while our relative cycles contain all integer rectifiable ones. One reason for our choice is that the area minimizing regularity theory by Gr\"{u}ter \cite{Gr87} requires the comparison relative cycles to be merely integer rectifiable but not necessarily integral. Actually we will show that our theory developed using integer rectifiable relative cycles is equivalent to the one established in \cite{Almgren62}. The reason for using equivalence classes instead of just the space of relative cycles is that we need a well-defined ``flat norm" under which there are natural compactness theorems for relative cycles/equivalence classes with bounded mass. As shown in \cite[4.4.4]{Federer}, the space of all integer rectifiable relative cycles with bounded mass is not compact under the usual flat norm; and \cite[4.4.4]{Federer} proposed a new norm which is not geometrically easy to use for the purpose of cut-and-paste arguments. Using the space of equivalence classes, we introduced a relative flat norm and a relative mass norm (Definition \ref{D:relative-cycles}), under which we proved a natural compactness result (Lemma \ref{L:compactness}). Although the use of equivalence classes brought in new technical difficulties, we eventually obtained many useful properties of the relative norms similar to those of the usual flat and mass norms, including two important isoperimetric lemmas (Section \ref{SS:isoperimetic}). 

Similar to \cite{Pitts}, we need to use (discrete) sweepouts which are continuous in the relative mass norm topology. To connect relative flat norm topology and relative mass norm topology, we established several interpolation results. In particular, one interpolation result (Lemma \ref{L:interpolation}) was used to show that the almost minimizing properties defined using both norms are equivalent, which is essential for the existence and regularity of almost minimizing varifolds. The proof of these interpolation results is based heavily on the isoperimetric lemmas under the relative norms, and the cut-and-paste techniques under Fermi coordinates.

While any stationary varifold in a compact submanifold \emph{without boundary} embedded in $\mathbb{R}^L$ automatically has bounded first variation as a varifold in $\mathbb{R}^L$, the same does not hold for a stationary varifold with free boundary in a compact submanifold \emph{with boundary}. Hence many of the classical results for varifolds in $\mathbb{R}^L$ with locally bounded first variation do not apply. In particular, we cannot directly apply the Rectifiability Theorem \cite[Theorem 42.4]{Simon} to conclude that our almost minimizing varifold $V$ is \emph{rectifiable} simply from a uniform positive lower bound on its density. We overcome this difficult by first establishing the rectifiability of its tangent cones at a boundary point, which leads to a classification of the tangent cones. Since the tangent cones can be shown to be almost everywhere \emph{unique}, we can still prove that $V$ is rectifiable (see Proposition \ref{P:tangent-cone}).

Finally, in the proof of the main regularity theorem (Theorem \ref{T:main-regularity}), where we have to show that two successive replacements can be glued together smoothly along a Fermi half-sphere, there are many more cases we have to consider compared to the classical situation without boundary. First, we have to handle the cases separately near a true boundary point or a false boundary point. Second, in the blow-up arguments to show the replacements glue in a $C^1$ manner, there are two different convergence scenarios (that we called Type I and II). In the first scenario the scales are small enough so that the blow-ups do not see the boundary $\partial M$ in the limit, while in the second case the boundary can be seen in the limit. The type of convergence scenarios depends on the ratio between distance of the center of blow-up point (which may not lie on $\partial M$) to the boundary $\partial M$ with the radius of the ball. We have to analyze the situation independently in both cases to obtain the required regularity.

Notice that we have chosen to work with Fermi coordinate systems around points on the boundary $\partial M$ mainly for two reasons. One reason is that the Fermi half-spheres give a nice smooth foliation of a relative open neighborhood of a boundary point on $\partial M$ by hypersurfaces meeting $\partial M$ \emph{orthogonally}. This orthogonality would come in handy when we want to apply the maximum principle in Theorem \ref{T:max-principle}. Another reason, which is of a more technical nature, is that in establishing the Interpolation Lemma (Lemma \ref{L:interpolation}), one has to carry out a ``\emph{cone construction}'' which can only be done if the domain under consideration is star-shaped (c.f. Lemma \ref{L:replace-cone}). The advantage of using Fermi coordinates system over a geodesic normal coordinate on an extension $\ti{M}$ of $M$ is that the comparison cone still lies inside $M$ in a Fermi coordinate chart.

\vspace{1em}

\noindent \textbf{Organization of the paper.} The paper is organized as follows. In Section \ref{S:background}, we set some notations and collect some preliminary results for varifolds in Riemannian manifolds which will be used throughout the paper. In particular, we define almost properly embedded hypersurfaces in $M$, discuss their first and second variations, and prove a compactness theorems for stable almost properly embedded hypersurfaces (Theorem \ref{T:freebdy-cpt}).
In Section \ref{S:amvarifold}, we define the important concept of \emph{almost minimizing varifolds} using equivalence classes of relative cycles, which we give a detailed construction. Two technical isoperimetric lemmas (Lemma \ref{L:isoperimetric-F} and \ref{L:isoperimetric-M}) are given and the equivalence of different definitions of almost minimizing varifolds is stated in Theorem \ref{T:def-equiv}, whose technical proof is given in the Appendix. In Section \ref{S:min-max}, we describe the general min-max construction developed by Almgren and Pitts adapted to the free boundary setting. We prove a version of the discretization and interpolation theorems in our setting (Theorem \ref{T:discretization} and \ref{T:interpolation}). We describe the tightening process (Proposition \ref{P:tightening}) and the combinatorial argument (Theorem \ref{T:combinatorial}) which imply the existence of stationary and almost minimizing varifolds with free boundary respectively (c.f. Proposition \ref{P:tightening} and Corollary \ref{C: existence of almost minimizing varifold}). Finally, in Section \ref{S:regam}, we give a complete proof of our regularity theorem (Theorem \ref{T:main-regularity}) up to the boundary.

\vspace{1em}
\textbf{Acknowledgement.} 
Both authors would like to thank Prof. Richard Schoen for introducing the problem to them, and his continuous encouragement, insightful discussions and interest in this work. 
The authors also want to thank Prof. Shing Tung Yau, Prof. Andre Neves for many fruitful discussions, and Prof. Fernando Coda Marques, Prof. Tobias Colding and Prof. Bill Minicozzi for their interest in this work. The first author is partially supported by a research grant from the Research Grants Council of the Hong Kong Special Administrative Region, China [Project No.: CUHK 24305115] and CUHK Direct Grant [Project Code: 4053118]. The second author is partially supported by  NSF grant DMS-1406337.

\section{Definitions and preliminary results}
\label{S:background}

In this section, we present some basic definitions and notations which will be used throughout the rest of the paper. First, we recall some basic notions in Euclidean spaces. Then we define stationary varifolds with free boundary in Riemannian manifolds with boundary and state two important theorems for this paper, the monotonicity formula (Theorem \ref{T:monotonicity}) and the maximum principle (Theorem \ref{T:max-principle}). 
Finally, we define an important notion of \emph{(almost) properly embedded submanifold} inside a compact manifold with boundary. Moreover, we also define the notion of \emph{stability} in the hypersurface case (Definition \ref{D:stable}) and prove a compactness theorem (Theorem \ref{T:freebdy-cpt}) which is a crucial ingredient in the regularity theory of this paper.

\subsection{Notations in Euclidean spaces}
\label{SS:notation}

We adopt the following notations in $\R^L$:
\[ \begin{array}{cl}
B_r(p) & \text{Euclidean open ball of radius $r$ centered at $p$} \\
A_{s,r}(p) & \text{Euclidean open annulus } B_r(p) \setminus \overline{B}_s(p) \\
\bmu_r & \text{the homothety map } x \mapsto r \, x \\
\btau_p & \text{the translation map } x \mapsto x-p \\
\bleta_{p,r} & \text{the composition } \bmu_{r^{-1}} \circ \btau_p \\
\Gr(L,k) & \text{the Grassmannian of unoriented $k$-dimensional subspaces in } \mathbb{R}^L \\
\Gr_k(A) & \text{the subset } A \times \Gr(L,k) \text{ for any Borel }A \subset \mathbb{R}^L\\
\Clos(A) & \text{the closure of a subset }A \subset \mathbb{R}^L\\
\mathcal{H}^k & \text{the $k$-dimensional Hausdorff measure in } \mathbb{R}^L\\
\omega_k & \text{the volume of the $k$-dimensional unit ball } B_1(0) \subset \R^k\\
\mathfrak{X}(\mathbb{R}^L) & \text{the space of smooth vector fields on $\mathbb{R}^L$}
\end{array} \]

Unless otherwise stated, $\mathbb{R}^L$ is always equipped with the standard inner product $\cdot$ and norm $| \cdot |$, with (flat) covariant derivative denoted by $D$. We often use the Euclidean coordinates given by $(x_1,\cdots,x_n,t)$. For any subset $S \subset \mathbb{R}^{n+1}$, we will define
$ S_\pm := S \cap \{ \pm t \geq 0\}$. For example, $\mathbb{R}^{n+1}_+$ is the upper half-space with boundary $\partial \mathbb{R}^{n+1}_+ \cong \mathbb{R}^n=\{t=0\}$. The \emph{tangent cone} of any subset $S \subset \mathbb{R}^L$ at $p$ is defined as 
\begin{equation}
\label{E:tangent-cone-def}
T_pS:=\left\{v \in \mathbb{R}^L : \begin{array}{l}
\text{for every $\ep>0$, there exists $x \in S$} \\
\text{and $r >0$ such that $|x-p| < \ep$} \\
\text{and $|\bmu_r \circ \btau_p (x)-v|< \ep$}
\end{array} \right\}.
\end{equation}

We now quickly recall some basic notions of varifolds in $\R^L$ and refer the readers to the standard references \cite{Allard72} and \cite{Simon} for details. 
For any measure $\mu$ on $\mathbb{R}^L$, $\Theta^k(\mu,p)$ will denote the $k$-dimensional density of $\mu$ at a point $p \in \R^L$.
The space of \emph{$k$-varifolds} in $\mathbb{R}^L$, denoted by $\bV_k(\mathbb{R}^L)$, is the set of all Radon measures on the Grassmannian $\mathbb{R}^L\times\Gr(L,k)$ equipped with the weak topology. The \emph{weight} and \emph{mass} of a varifold $V \in \bV_k(\mathbb{R}^L)$ is denoted respectively by $\|V\|$ and $\M(V):=\|V\|(\mathbb{R}^L)$. For any Borel set $A \subset \R^L$, we denote $V \lc A$ to be the \emph{restriction of $V$ to $\Gr_k(A)=A \times \Gr(L,k)$}. The \emph{support of $V$}, $\spt \|V\|$, is the smallest closed subset $B \subset \mathbb{R}^L$ such that $V\lc(\mathbb{R}^L \setminus B)=0$. See \cite[\S 38]{Simon}. The $\mF$-metric on $\bV_k(\mathbb{R}^L)$ as defined in \cite[2.1(19)]{Pitts} induces the weak topology on the set of mass-bounded $k$-varifolds. 
We are most interested in the class of \emph{(integer) rectifiable $k$-varifolds} (see \cite[Chapter 3 and 4]{Simon}). 
The set of rectifiable and integer rectifiable $k$-varifolds in $\mathbb{R}^L$ will be denoted by $\bRV_k(\mathbb{R}^L)$ and $\bIV_k(\mathbb{R}^L)$ respectively \cite[2.1(18)]{Pitts}. 
For any $C^1$ map $f:\mathbb{R}^L \to \mathbb{R}^{L}$ be a $C^1$, we have a continuous \emph{pushforward} map $f_\sharp:\bV_k(\mathbb{R}^L) \to \bV_k(\mathbb{R}^{L})$ as defined in \cite[\S 39]{Simon}. 
As in \cite[42.3]{Simon}, we denote $\VarTan (V,p)$ to be the set of \emph{varifold tangents} of a varifold $V \in \bV_k(\R^L)$ at some $p \in \spt\|V\|$. 
By the compactness of Radon measures \cite[Theorem 4.4]{Simon}, $\VarTan(V,p)$ is compact and non-empty provided that upper density $\Theta^{*k}(\|V\|,p)$ is finite. Moreover, there exists a non-zero element $C \in \VarTan(V,p)$ if and only if $\Theta^{*k}(\|V\|,p)>0$. See \cite[3.4]{Allard72}.


\subsection{Stationary varifolds with free boundary}
\label{SS:manifolds-with-boundary}

In this paper, $M^{n+1}$ is a smooth compact connected $(n+1)$-dimensional Riemannian manifold with nonempty boundary $\partial M$. We can always extend $M$ to a closed Riemannian manifold $\tM$ of the same dimension \cite{Pigola-Veronelli16} so that $M \subset \tM$. We will denote the intrinsic Riemannian metric by $\langle \cdot, \cdot \rangle$ and the Levi-Civita connection by $\nabla$. 
We will equip $M$ with the subspace topology induced from $\tM$. 
The following notations will be used throughout the paper (see Appendix \ref{A:Fermi} for the notations involving Fermi coordinates):
\[ \begin{array}{cl}
\nu_{\partial M} & \text{the inward unit normal of $\partial M$ with respect to $M$}\\
\tBcal_r(p) & \text{the open geodesic ball in $\tM$ of radius $r$ centered at $p$} \\
\tScal_r(p) & \text{the geodesic sphere in $\tM$ of radius $r$ centered at $p$}\\
\tAcal_{s,r}(p) & \text{the open geodesic annulus $\tBcal_r(p) \setminus \Clos(\tBcal_s(p))$ in $\tM$} \\
\tBcal^+_r(p) & \text{the (relatively) open Fermi half-ball of radius $r$ centered at $p \in \partial M$} \\
\tScal^+_r(p) & \text{the Fermi half-sphere of radius $r$ centered at $p \in \partial M$}\\
\end{array} \]

We now proceed to define varifolds in a Riemannian manifold $N$ (possibly with boundary), which can be assumed to be isometrically embedded as a closed subset of some $\R^L$ by Nash isometric embedding. Here we will follow mostly the notations in \cite{Pitts} (which is slight different from \cite{Allard72}). 
For a smooth submanifold $N^k \subset \mathbb{R}^L$ (with or without boundary), we define the following spaces of vector fields 
\begin{eqnarray}
\label{E:vector-fields}
\mathfrak{X}(N) &:=& \{X \in \mathfrak{X}(\mathbb{R}^L) : X(p) \in T_pN \text{ for all $p \in N$}.\} \\
\mathfrak{X}_{tan}(N) &:=& \{X \in \mathfrak{X}(N) : X(p) \in T_p(\partial N) \text{ for all $p \in \partial N$}. \} \nonumber
\end{eqnarray}
Note that using (\ref{E:tangent-cone-def}), at each $p \in \partial N$, $T_pN$ is a $k$-dimensional half-space in $\mathbb{R}^L$ with boundary $T_p(\partial N)$. 
We define the space of (integer) rectificable $k$-varifolds in $M$, denoted by $\RV_k(N)$ (resp. $\IV_k(N)$), as the set of all (integer) rectifiable $k$-varifolds in $\R^L$ with $\spt \|V\| \subset N$. Moreover, $\V_k(N):=\Clos (\RV_k(N)) \subset \bV_k(\mathbb{R}^L)$. 
If $f:\mathbb{R}^L \to \mathbb{R}^L$ is a $C^1$ map such that $f(N_1) \subset N_2$, then by \cite[2.1(18)(h)]{Pitts}, we have $f_\sharp(\V_k(N_1)) \subset \V_k(N_2)$, $f_\sharp(\RV_k(N_1)) \subset \RV_k(N_2)$ and $f_\sharp (\IV_k(N_1)) \subset \IV_k(N_2)$. Note that by \cite[2.1(18)(i)]{Pitts}, $\VarTan(V,p) \subset \V_k(T_pN)$ for any $V \in \V_k(N)$ and $p \in \spt \|V\|$.

Let $V \in \V_k(N)$, if $X \in \mathfrak{X}_{tan}(N)$ generates a one-parameter family of diffeomorphisms $\phi_t$ of $\mathbb{R}^L$ with $\phi_t(N)=N$, then $(\phi_t)_ \sharp V \in \V_k(N)$ and one can consider its first variation along the vector field $X$ \cite[39.2]{Simon}
\begin{equation}
\label{E:1st-variation}
\delta V (X) := \left. \frac{d}{dt}\right|_{t=0} \M((\phi_t)_\sharp V)=\int \Div_S X(x) \, dV(x, S), 
\end{equation}
where $\Div_S X(x)= \langle D_{e_i} X, e_i \rangle$ where $\{e_1,\cdots,e_k\} \subset S$ is any orthonormal basis.

\begin{definition}[Stationary varifolds with free boundary]
\label{D:freebdy}
Let $U \subset N$ be a relatively open subset. A varifold $V\in\V_k(N)$ is said to be \emph{stationary in $U$ with free boundary} if $\delta V(X) =0$ for any $X \in \mathfrak{X}_{tan}(N)$ compactly supported in $U$. 
\end{definition}

From the first variation formula it is clear that the set of all $V \in \V_k(N)$ which is stationary \emph{in $N$} with free boundary is a closed subset of $\V_k(N)$ in the weak topology. 
We will be mainly interested in the case $N=M$ or $T_pM$ in Definition \ref{D:freebdy}. Recall that in the later case, $T_pM$ is an $(n+1)$-dimensional half-space in $\mathbb{R}^L$ when $p \in \partial M$. The following reflection principle will be useful when one consider the tangent varifolds of $V \in \V_k(M)$ at some $p \in \partial M$ as $\VarTan(V,p) \subset \V_k(T_pM)$.

\begin{lemma}[Reflection principle]
\label{L:reflection}
Let $v \in \mathbb{R}^L$ be a unit vector. Suppose $P \subset \mathbb{R}^L$ is an $(n+1)$-dimensional subspace with $v\in P$ and denote $P_+:=\{u \in P: u \cdot v \geq 0\}$ to be the closed half-space. Let $\theta_v: \mathbb{R}^L \to \mathbb{R}^L$ denote reflection map about $v$, i.e. $\theta_v(u)=u - 2 (u \cdot v) v$. For any $V \in \V_k(P_+)$, define the doubled varifold
\begin{equation}
\label{E:double-varifold}
\overline{V}:=V+(\theta_v)_\sharp V \in \V_k(P).
\end{equation}
If $V$ is stationary in $P_+$ with free boundary, then $\overline{V}$ is stationary in $P$.
\end{lemma}

\begin{proof}
By the definition of pushforward of a varifold (\cite[\S 39]{Simon}), (\ref{E:1st-variation}) and that $\theta_v$ is an isometry of $\mathbb{R}^L$, one easily sees that $\delta \overline{V} (X)=  \delta V (X + (\theta_v)_* X)$ for any $X \in \mathfrak{X}(P)$ with compact support. Here $(\theta_v)_*X$ is the pushforward of the vector field $X$ by $\theta_v$. Since $X + (\theta_v)_* X \in \mathfrak{X}_{tan}(P_+)$, the assertion follows directly.
\end{proof}

We now state the monotonicity formula for stationary varifold with free boundary which holds near the boundary $\partial M$. Note that the monotonicity formula is stated relative to Euclidean balls $B_r(p) \subset \R^L$.

\begin{theorem}[Monotonicity formula]
\label{T:monotonicity}
Let $V \in \V_k(M)$ be a stationary varifold in $M$ with free boundary. 
Then $\Theta^k(\|V\|,p)$ exists at every $p \in \partial M$ and there exists a constant $C_{mono}>1$ depending only on $r_{mono}$ and $\Lambda_0$ such that for all $0 <\si < \rho <r_{\textrm{mono}}$, we have
\[   \frac{\|V\|(B_\si(p))}{\omega_k \si^k} \leq C_{mono} \frac{\|V\|(B_\rho(p))}{\omega_k \rho^k}. \]
\end{theorem}

\begin{proof}
See \cite[Theorem 3.5]{Li-Zhou-A}. Note that the constants $\kappa=R_0^{-1}$, $\gamma$ and $\Lambda$ in \cite[Theorem 3.3]{Li-Zhou-A} only depend on the isometric embedding $M \hookrightarrow \mathbb{R}^L$. Moreover, $d=\infty$ since $N=\partial M$ has no boundary.
\end{proof}

We will also need a maximum principle for stationary varifolds with free boundary in the codimension one case. The proof and a more general statement (which works in any codimension) can be found in \cite{max-principle}. For any subset $A \subset M$, we define the \emph{relative interior of $A$}, denoted by $\interior_M(A)$, to be the interior of $A$ with respect to the subspace topology of $M$. 
The \emph{relative boundary of $A$}, denoted by $\partial_{rel} A$, is the set of all the points in $M$ which is neither in the relative interior of $A$ or $M \setminus A$. 

\begin{definition}[Relative convexity]
\label{D:rel-convex}
A subset $\Omega \subset M$ is said to be a \emph{relatively convex domain in $M$} if it is a relatively open connected subset in $M$ whose relative boundary $\partial_{rel} \Omega$ is a smooth convex  
hypersurface (possibly with boundary) in $M$. 
\end{definition}


\begin{theorem}[Maximum principle]
\label{T:max-principle}
Let $k=n= \dim M -1$ and $V \in \V_n(M)$ be stationary in a relatively open subset $U \subset M$ with free boundary. Suppose $K \subset \subset U$ is a smooth relatively open connected subset in $M$ such that
\begin{itemize}
\item[(i)] $\partial_{rel} K$ meets $\partial M$ orthogonally,
\item[(ii)] $K$ is relatively convex in $M$,
\item[(iii)] $\spt \|V\| \subset \overline{K}$,
\end{itemize}
then we have $\spt \|V\| \cap \partial_{rel} K = \emptyset$.
\end{theorem}

\begin{proof}
See \cite{max-principle}.
\end{proof}

\subsection{Almost proper embeddings and stability}
\label{SS:properness}

As in section \ref{SS:manifolds-with-boundary}, we consider a smooth compact Riemannian manfiold $M^{n+1}$ with nonempty boundary $\partial M$. As before, we assume without loss of generality that $M$ is a smooth compact subdomain of a closed Riemannian manifold $\tM$ of the same dimension. 

\begin{definition}[Almost proper embeddings]
\label{D:proper}
Let $\Sigma^n$ be a smooth $n$-dimensional manifold with boundary $\partial \Sigma$ (possibly empty). A smooth embedding $\phi:\Sigma \to \tM$ is said to be an \emph{almost proper embedding of $\Sigma$ into $M$} if
\[ \phi(\Sigma) \subset M \qquad \text{and} \qquad \phi(\partial \Sigma) \subset \partial M,\] 
we would write $\phi:(\Sigma,\partial \Sigma) \to (M,\partial M)$. For simplicity, we often take $\phi$ as the inclusion map $\iota$ and write $(\Sigma,\partial \Sigma) \subset (M,\partial M)$. Given an almost properly embedded hypersurface $(\Sigma,\partial \Sigma) \subset (M,\partial M)$, we say that $p \in \Sigma \cap \partial M$ is a 
\begin{itemize}
\item \emph{true boundary point} if $p \in \partial \Sigma$; or a
\item \emph{false boundary point} otherwise.
\end{itemize}
If every $p \in \Sigma \cap \partial M$ is a true boundary point, i.e. $\Sigma \cap \partial M=\partial \Sigma$, we say that $(\Sigma,\partial \Sigma) \subset (M,\partial M)$ is \emph{properly embedded}.
\end{definition}

\begin{remark}
\label{R:false-boundary}
For an almost properly embedded hypersurface $(\Sigma,\partial \Sigma) \subset (M,\partial M)$, since $\Sigma \subset M$, $\Sigma$ must touch $\partial M$ tangentially from inside of $M$ at any false boundary points $p$, i.e. $T_p \Si=T_p(\partial M)$. 
\end{remark}

Given an almost properly embedded hypersurface $(\Sigma,\partial \Sigma) \subset (M,\partial M)$, we define (recall (\ref{E:vector-fields}))
\[ \mathfrak{X}(M,\Sigma):=\left\{ X \in \mathfrak{X}(M): \begin{array}{c}
X(q) \in T_q(\partial M) \text{ for all $q$ in an open }\\
\text{neighborhood of $\partial \Sigma$ in $\partial M$} \end{array} \right\}.\]
Any compactly supported $X \in \mathfrak{X}(M,\Sigma)$ generates a smooth one parameter family of diffeomorphisms $\phi_t$ of $\mathbb{R}^L$ such that $\Sigma_t:=\phi_t(\Sigma)$ is a family of almost properly embedded hypersurfaces in $M$.
By the first variation formula, we have
\begin{equation}
\label{E:1st-variation-area}
\delta \Sigma(X) := \left. \frac{d}{dt} \right|_{t=0} \textrm{Area}(\Sigma_t) = -\int_\Sigma \langle H,X \rangle \; da + \int_{\partial \Sigma} \langle \eta,X \rangle \; ds,
\end{equation}
where $H$ is the mean curvature vector of $\Sigma$, $\eta$ is the outward unit co-normal of $\partial \Sigma$.

\begin{definition}
\label{D:stationary}
An almost properly embedded hypersurface $(\Sigma,\partial \Sigma) \subset (M,\partial M)$ is said to be \emph{stationary} if $\delta \Sigma(X)=0$ in (\ref{E:1st-variation-area}) for any compactly supported $X\in \mathfrak{X}(M,\Sigma)$. 
\end{definition}

\begin{remark}
If $(\Sigma,\partial \Sigma) \subset (M,\partial M)$ is \emph{properly embedded}, then it is stationary if and only if $\delta \Sigma(X) =0$ for any compactly supported $X \in \mathfrak{X}_{tan}(M)$. 
Note that any stationary almost properly embedded $(\Sigma,\partial \Sigma) \subset (M,\partial M)$ is also stationary in $M$ with free boundary in the sense of Definition \ref{D:freebdy}. However, a stationary varifold with free boundary may not be a free boundary minimal submanifold even if it is smooth. For example, when $N=M \subset \tM$, then $\partial M$ is a stationary varifold with free boundary but not necessarily minimal in $\tM$. 
\end{remark}

From the first variation formula (\ref{E:1st-variation-area}), it is clear that an almost properly embedded hypersurface $(\Sigma,\partial \Sigma) \subset (M,\partial M)$ is stationary if and only if the mean curvature of $\Sigma$ vanishes identically and $\Sigma$ meets $\partial M$ orthogonally along $\partial \Sigma$, which are known as \emph{free boundary minimal hypersurfaces}. Given such a free boundary minimal hypersurface $(\Sigma,\partial \Sigma) \subset (M,\partial M)$, assuming furthermore that $\Sigma$ is \emph{two-sided} (i.e. there exists a unit normal $\nu$ continuously defined on $\Sigma$), consider a compactly supported normal variation vector field $X=f \nu $ along $\Sigma$, which can be extended to a globally defined vector field in $\mathfrak{X}(M,\Sigma)$, the second variation formula of area gives
\begin{equation}
\label{E:2nd-variation-area}
\delta^2 \Sigma(X) := \left. \frac{d^2}{dt^2} \right|_{t=0} \textrm{Area}(\Sigma_t) 
\end{equation}
\[ =-\int_\Sigma |\nabla^\Sigma f|^2 -(\Ric_M(\nu,\nu)+|A^\Sigma|^2)f^2 \; da - \int_{\partial \Sigma} h(\nu,\nu) f^2 \; ds \]
where $\nabla^\Sigma$ is the induced connection on $\Sigma$, $\Ric_M$ is the Ricci curvature of $M$, $A^\Sigma$ and $h$ are the second fundamental forms of the hypersurfaces $\Sigma$ and $\partial M$ with respect to the normals $\nu$ and $\nu_{\partial M}$ respectively.

\begin{definition}
\label{D:stable}
An almost properly embedded hypersurface $(\Sigma,\partial \Sigma) \subset (M,\partial M)$ is said to be \emph{stable} if 
\begin{itemize}
\item $\Sigma$ is two-sided,
\item it is stationary in the sense of Definition \ref{D:stationary},
\item $\delta^2 \Sigma(X) \geq 0$ for any compactly supported $X=f\nu \in \mathfrak{X}(M,\Sigma)$.
\end{itemize}
\end{definition}

\begin{remark}
If $(\Sigma,\partial \Sigma) \subset (M,\partial M)$ is a \emph{properly embedded}, two-sided free boundary minimal hypersurface, then it is stable if and only if $\delta^2 \Sigma(X) \geq 0$ for any compactly supported $X \in \mathfrak{X}_{tan}(M)$.
\end{remark}

The definitions above can be localized as follows: 

\begin{definition}
Let $U \subset M$ be a relatively open subset. We say that an almost properly embedded hypersurface $(\Sigma,\partial \Sigma) \subset (U,U \cap \partial M)$ is \emph{stationary} (resp. \emph{stable}) \emph{in $U$} if $\delta \Sigma(X)=0$ (resp. $\delta^2 \Sigma (X) \geq 0$) for any $X \in \mathfrak{X}(M,\Sigma)$ which is compactly supported in $U$.
\end{definition}

We shall need the following smooth compactness theorem for \emph{stable} almost properly embedded hypersurfaces satisfying a uniform area bound.

\begin{theorem}[Compactness theorem for stable almost properly embedded hypersurfaces]
\label{T:freebdy-cpt}
Let $2 \leq n \leq 6$ and $U \subset M$ be a simply connected relative open subset. If $(\Sigma_k,\partial \Sigma_k) \subset (U,U \cap \partial M)$ is a sequence of almost properly embedded free boundary minimal hypersurfaces which are stable in $U$ and $\sup_{k} \Area(\Sigma_k) < \infty$, 
then after passing to a subsequence, $(\Sigma_k,\partial \Sigma_k)$ converges (possibly with multiplicities) to some almost properly embedded free boundary minimal hypersurface $(\Sigma_\infty,\partial \Sigma_\infty) \subset (U,U \cap \partial M)$ which is stable in $U$. Moreover, the convergence is uniform and smooth on compact subsets of $U$.
\end{theorem}

\begin{proof}
This is a direct consequence of the uniform curvature estimates in \cite[Theorem 1.1]{Li-Zhou-A} and the classical maximum principle for minimal hypersurfaces with free boundary. Note that $U$ is assumed to be simply connected just to guarantee that all the embedded hypersurfaces are two-sided \cite{Samelson69}. See Figure \ref{stable convergence} for the two different possible convergence scenarios. In (A), we have a point on $\interior \Si_\infty \cap \partial M$ which is not the limit of any sequence of boundary points on $\Si_k$. In (B), the point on $\partial \Si_\infty$ is the limit of a sequence of boundary points on $\partial \Si_k$.
\end{proof}

\begin{figure}[h]
    \centering
    \begin{subfigure}{.47\textwidth}
        \centering
        \includegraphics[width=1\textwidth]{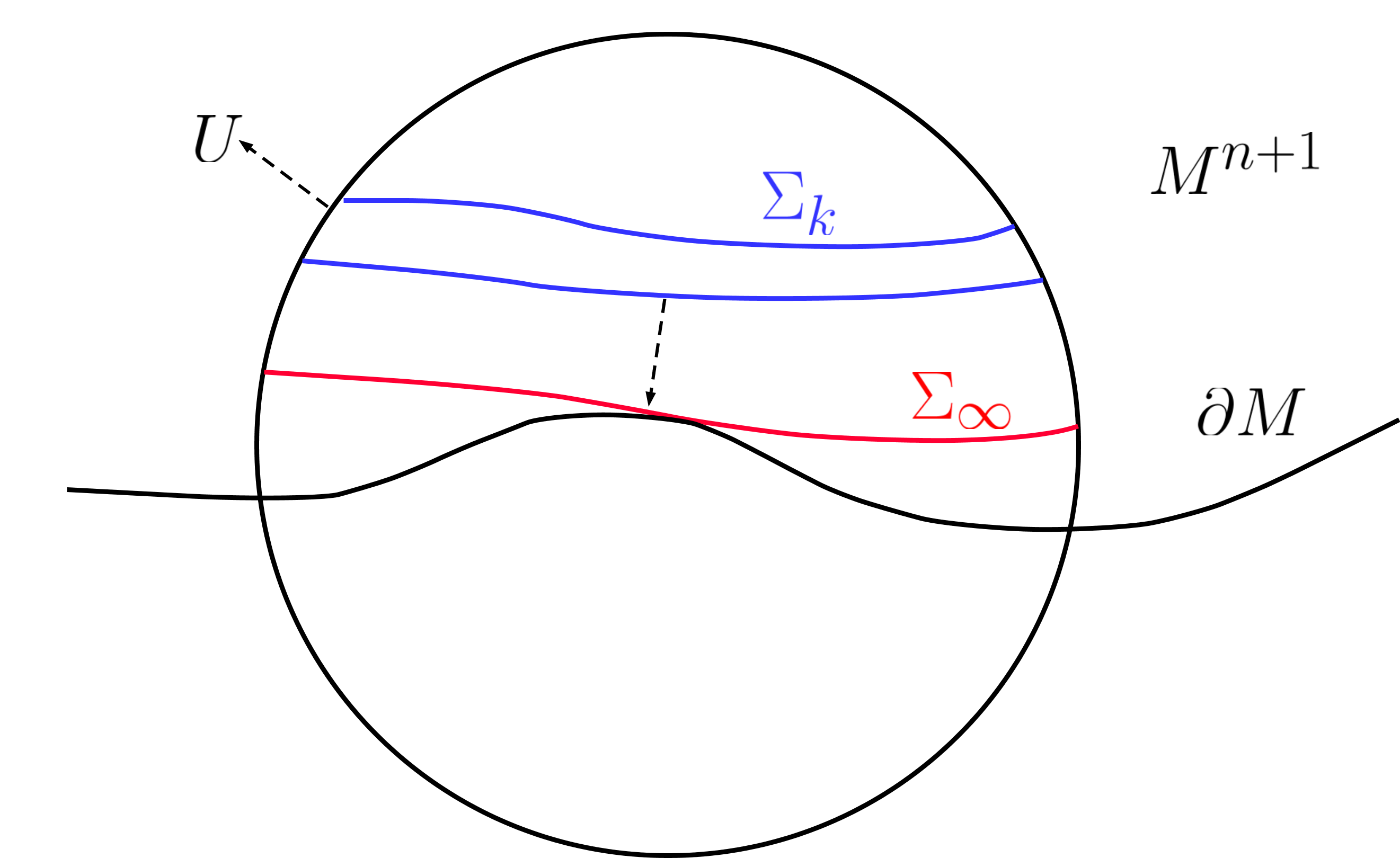}
       \caption{}
    \end{subfigure}%
    \begin{subfigure}{.43\textwidth}
        \centering
        \includegraphics[width=1\textwidth]{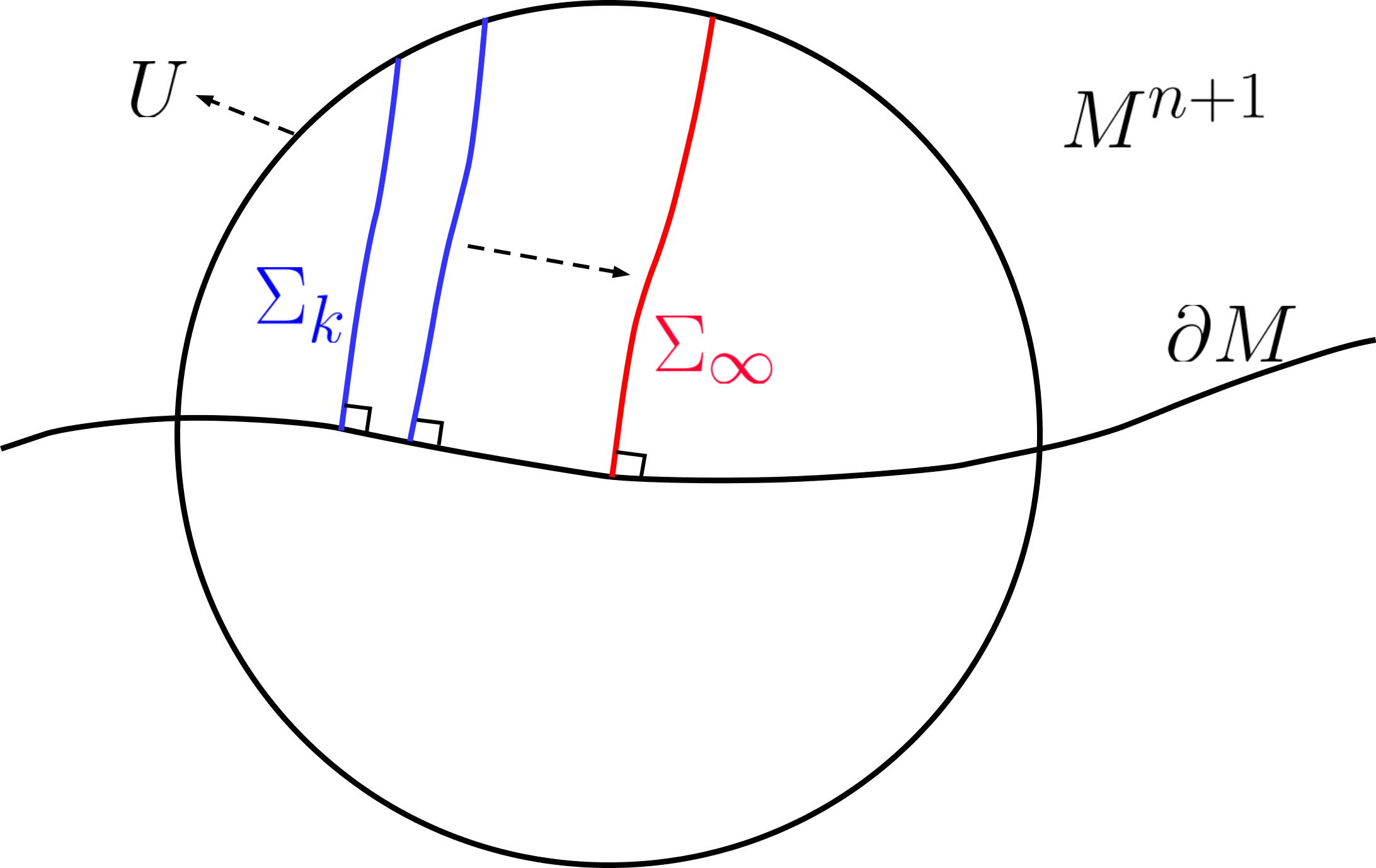}
        \caption{}
    \end{subfigure}
    \caption{Stable convergence.}
    \label{stable convergence}
\end{figure}

\begin{remark}
The compactness theorem above does not hold with \emph{almost properly embedded} replaced by \emph{properly embedded} hypersurfaces since a limit of properly embedded hypersurfaces may only be \emph{almost} properly embedded (see Figure \ref{stable convergence}(A)).
\end{remark}

\section{Almost minimizing varifolds with free boundary}
\label{S:amvarifold}

In this section, we define the important notion of \emph{almost minimizing varifolds with free boundary} (Definition \ref{D:am-varifolds}). Roughly speaking, a varifold $V \in \V_k(M)$ is almost minimizing with free boundary if $V$ can be approximated by currents with \emph{suitable} minimizing properties. One major reason that we have to work with currents is that it allows us to do cut-and-paste arguments. Unlike the closed case, since we will be looking at relative cycles with boundary lying on $\partial M$, to obtain a well-defined theory we are forced to consider equivalence classes of relative cycles (this idea already appeared in \cite{Almgren62}). We will first give a detailed description on the space of equivalence classes of relative cycles. Then, we prove two important (but technical) isoperimetric lemmas (Lemma \ref{L:isoperimetric-F} and \ref{L:isoperimetric-M}) which are crucial in the discretization and interpolation theorems in section \ref{SS:discret}. Finally, we give various definitions of almost minimizing varifolds with free boundary in Definition \ref{D:am-varifolds} and show that they are essentially equivalent (Theorem \ref{T:def-equiv}). We would like to point out that some of the results here were also independently obtained in \cite{Liokumovich-Marques-Neves}.

\subsection{Equivalence classes of relative cycles}
\label{SS:relative-cycles}

Recall that we have fixed an isometric embedding $M \subset \tM \hookrightarrow \mathbb{R}^L$. Let $\mR_k(M)$ be the space of integer rectifiable $k$-currents in $\mathbb{R}^L$ which is supported in $M$. The \emph{mass norm} and the \emph{flat semi-norm} (relative to some compact subset $K \subset M$) on $\mR_k(M)$ are denoted by $\M$ and $\F^K$ respectively (see \cite{Simon}). We fix throughout this subsection a relatively open subset $U \subset M$ which contains $K$. For simplicity of notations, we write $(A,B)=(U, U \cap \partial M)$. We define the following spaces of relative cycles:
\begin{equation}
\label{E:relative-cycles}
\begin{array}{c} 
Z_k(A, B):= \{T \in \mR_k(M) \; : \; \spt(T) \subset A, \; \spt(\partial T) \subset B\},\\
Z_k(B, B):= \{ T \in \mR_k(M) \; : \; \spt(T) \subset B\}.
\end{array}
\end{equation}
Note that $Z_k(B,B)$ is a subspace of $Z_k(A,B)$. It should be pointed out that our notation is slightly different from the one in \cite[(1.20)]{Almgren62} that $T \in Z_k(A,B)$ may not be an \emph{integral current} since $\partial T$ is not assumed to be integer rectifiable (see \cite[Remark 1.21]{Almgren62}). Nonetheless, it turns out that this formulation is equivalent to the one in \cite{Almgren62} as shwon in Lemma \ref{L:integral}. Following \cite[(1.20)]{Almgren62}, we make the following definitions.

\begin{definition}[Equivalence classes of relative cycles]
\label{D:relative-cycles}
We define the space of \emph{equivalence classes of relative cycles} as the quotient group
\[ \Z_k(A, B):=Z_k(A, B)/Z_k(B, B).\]
For any $\tau\in\Z_k(A, B)$, we define its (relative) {\em mass norm} and \emph{flat (semi)-norm} (with respect to some compact subset $K \subset A$) by
\[ \M(\tau):=\inf_{T \in \tau} \M(T), \qquad \F^K(\tau):=\inf_{T \in \tau} \F^K(T).\]
Similarly, we can define the local mass and flat norms $\M_{U'}$ and $\F^K_{U'}$ for a relatively open subset $U' \subset U$. The \emph{support of $\tau \in \Z_k(A,B)$} is defined as
\[ \spt(\tau):=\bigcap_{T \in \tau} \spt(T).\]
\end{definition}

For any $T \in Z_k(A,B)$, we have $T\lc B \in Z_k(B,B)$ and thus $T \lc (A \setminus B)=T - T \lc B$ lies in the same equivalence class as $T$. Conversely, $[T]=[T'] \in \Z_k(A,B)$ if and only if $T \lc (A \setminus B)=T' \lc (A \setminus B)$.

\begin{definition}
For any $\tau \in \Z_k(A,B)$, there exists a unique $T \in \tau$ such that $T \lc B=0$ (i.e. $T=T \lc (A \setminus B)$). We call $T$ the \emph{canonical representative of $\tau$}.
\end{definition}

The lemma below says that the mass and support of $\tau \in \Z_k(A,B)$ are the same as the usual mass and support of its canonical representative $T \in Z_k(A,B)$. A direct consequence is that $\M$ and $\F^K$ defines respectively a norm and a semi-norm on $\Z_k(A,B)$ satisfying $\F^K(\tau)\leq \M(\tau)$ for all $\tau \in \Z_k(A,B)$.

\begin{lemma}
\label{L:can-rep}
For any $\tau \in\Z_k(A, B)$, we have $\M(\tau)=\M(T)$ and $\spt(\tau)=\spt(T)$ where $T$ is the canonical representative of $\tau$.
\end{lemma}
\begin{proof}
By definition $\M(\tau)\leq \M(T)$ and $\spt(\tau)\subset \spt(T)$. On the other hand, if $T^{\pr} \in \tau$ is another representative, then $T' \lc (A \setminus B)=T$ and thus
\[ \M(T')=\M(T'\lc (A \setminus B))+\M(T'\lc B) \geq \M(T),\]
\[ \spt(T')=\spt(T'\lc (A \setminus B))\cup \spt(T'\lc B) \supset \spt(T ).\]
As $T' \in \tau$ is arbitrary, this implies $\M(\tau)=\M(T)$ and $\spt(\tau)=\spt(T)$.
\end{proof}

\begin{definition}[Weak convergence]
\label{D:weak-topology}
A sequence $\tau_i \in \Z_k(A,B)$ is said to be \emph{converging weakly to } $\tau_\infty \in \Z_k(A,B)$ if for every compact $K \subset A$, $\F^K(\tau_i -\tau_\infty) \to 0$ as $i \to \infty$. In this case, we write $\tau_i \weakto \tau_\infty$.
\end{definition}

\begin{lemma}
\label{L:weak-topology}
If $\tau_i \weakto \tau_\infty$ in $\Z_k(A,B)$, then $T_i$ converges weakly to $T_\infty$ as currents in the open subset $A \setminus B$, where $T_i \in \tau_i$, $T_\infty \in \tau_\infty$ are the canonical representatives.
\end{lemma}

\begin{proof}
Whenever $K \subset A \setminus B$, we have $\F^K(\tau)=\F^K(T)$ where $T \in \tau$ is the canonical representative. The rest follows from the fact that the semi-norms $\F^K$ induces the usual weak topology on currents \cite[31.2]{Simon}.
\end{proof}

\begin{remark}
Note that one cannot prove that $T_i \weakto T_\infty$ as currents \emph{in $A$} since $T_i$ can converge to a limit which is supported (but non-zero) on $B$.
\end{remark}

The next lemma shows that the mass norm is lower semi-continuous with respect to the weak topology defined in Definition \ref{D:weak-topology}. 

\begin{lemma}[Lower semi-continuity of $\M$]
\label{L:lower-semicts}
If $\tau_i \weakto \tau_\infty$ in $\Z_k(A, B)$, then 
\[ \M(\tau_{\infty})\leq\liminf_{i\rightarrow\infty}\M(\tau_i).\]
Moreover, if $\spt(\tau_i)\subset K$ for some compact $K \subset A$ for all $i$, then $\spt(\tau_{\infty})\subset K$.
\end{lemma}

\begin{proof}
Let $T_i \in \tau_i$, $T_\infty \in \tau_\infty$ be the canonical representatives. By Lemma \ref{L:weak-topology}, $T_i$ converges weakly to $T_\infty$ as currents in the open set $A \setminus B$. For any $\ep>0$, we can take a smooth cutoff function $\chi$, $0\leq \chi\leq 1$, supported in $A \setminus B$ such that 
\[ \M(T_{\infty}\lc \chi)\geq \M(T_{\infty})-\ep=\M(\tau_{\infty})-\ep, \]
where the last equality follows from Lemma \ref{L:can-rep}. Since $T_i \weakto T_{\infty}$, we also have $T_i\lc \chi \weakto T_{\infty}\lc \chi$. Hence the lower semi-continuity of the classical mass norm \cite[26.13]{Simon} implies that 
\[ \M(T_{\infty}\lc \chi)\leq \liminf_{i\rightarrow\infty}\M(T_i\lc \chi)\leq \liminf_{i\rightarrow\infty}\M(T_i)=\liminf_{i\rightarrow\infty}\M(\tau_i)\]
where the last equality follows again from Lemma \ref{L:can-rep}. Letting $\ep \to 0$ gives the desired inequality.

Suppose, in addition, that $\spt(\tau_i)\subset K$ for some compact $K \subset A$ for all $i$. Take a sequence of smooth cutoff functions $\{\chi_m\}$, $0\leq \chi_m\leq 1$, supported in $A \setminus B$ and converging pointwise to the characteristic function of $A \setminus B$. By Lemma \ref{L:can-rep}, 
\[ \bigcup_m \spt(T_{\infty}\lc \chi_m)=\spt(T_{\infty})=\spt(\tau_{\infty}).\]
For each $m$ and $i$, we have $\spt(T_i\lc \chi_m)\subset \spt(T_i)=\spt(\tau_i) \subset K$, hence we obtain $\spt(T_{\infty}\lc \chi_m)\subset K$. Taking union gives $\spt(\tau_\infty) \subset K$.
\end{proof}

Given $\tau \in \Z_k(A, B)$, in the definition of $\M(\tau)$ the infimum is taken over all \emph{integer rectifiable currents} in $Z_k(A, B)$, the next lemma says that the mass $\M(\tau)$ can also be computed as the infimum over all \emph{integral currents} in the class $\tau$.

\begin{lemma}
\label{L:integral}
Give $\tau\in\Z_k(A, B)$, there exists a sequence $T_i\in \tau$, where each $T_i$ is an integral current with $\spt(T_i) \subset A$, such that $\lim_{i\rightarrow \infty}\M(T_i)=\M(\tau)$.
\end{lemma}

\begin{proof}
Let $T$ be the canonical representative of $\tau$. Consider the intrinsic distance function $\rho(\cdot) =\dist_A(\cdot, B)$, which is a Lipschitz function on $A$. Note that $\lim_{t\rightarrow 0}\M(T\lc \{x:\ \rho(x)< t\})=\M(T\lc B)=0$. By \cite[28.4]{Simon}, there exists a sequence $t_i \to 0$ such that each slice $\lan T, \rho, t_i \ran$ of $T$ by $\rho$ at $t_i$ is an integer rectifiable $(k-1)$-current in $A \setminus B$. Also by \cite[28.5(2)]{Simon}, as $\spt(\partial T) \subset B$,
\begin{equation}
\label{E:slice-boundary} 
\partial(T\lc\{\rho<t_i\})=\partial T\lc \{\rho<t_i\}+ \lan T, \rho, t_i \ran=\partial T+\lan T, \rho, t_i \ran.
\end{equation}
Denote $\pi:A \to B$ as the nearest point projection map onto $B$, which is Lipschitz, and define
\[ T_i:=T-{\pi}_{\sharp}(T\lc \{\rho<t_i\}).\]
By \cite[4.1.30]{Federer}, the pushforward ${\pi}_{\sharp}(T\lc \{\rho<t_i\})$ is an integer rectifiable $k$-current supported on $B$, so $T_i \in \tau$. Since $\pi$ is Lipschitz, using Lemma \ref{L:can-rep} we have
\[ \M(T_i)\leq \M(T)+\M({\pi}_{\sharp}(T\lc \{\rho<t_i\}))\leq \M(\tau)+C\; \M(T\lc\{\rho<t_i\})\]
for some constant $C>0$. As $\M(T\lc\{\rho<t_i\}) \to 0$ when $t_i \to 0$, we have $\lim_{i \to \infty} \M(T_i)=\M(\tau)$. Finally, to show that $T_i$ is an \emph{integral current}, i.e. $\partial T_i$ is integer rectifiable, using (\ref{E:slice-boundary}) and that $\partial \circ \pi_\sharp = \pi_\sharp \circ \partial$,
\begin{displaymath}
\begin{split}
\partial T_i & = \partial T-\partial {\pi}_{\sharp}(T\lc \{\rho<t_i\})=\partial T-{\pi}_{\sharp}(\partial(T\lc \{\rho<t_i\}))\\
                    & = \partial T-{\pi}_{\sharp}(\partial T)-{\pi}_{\sharp} \lan T, \rho, t_i \ran =-{\pi}_{\sharp} \lan T, \rho, t_i \ran,
\end{split}
\end{displaymath}
where we have used that $\partial T$ is a $(k-1)$-current in $B$ (Note that $S-{\pi}_{\sharp} S=0$ is not true in general if we only assume that $\spt S \subset B$ since $S$ may not be a current in $B$). Hence $\partial T_i$ is integer rectifiable as $\lan T, \rho, t_i \ran$ is integer rectifiable and $\pi$ is Lipschitz.
\end{proof}

\begin{remark}
This lemma implies that the quotient groups defined in \cite[(1.20)]{Almgren62} give an equivalent theory compared to our definition.
\end{remark}

Next, we show that any subset of $\Z_k(A, B)$ with uniformly compact support and bounded mass is compact under the relative flat norm (with respect to a slightly larger compact set). Recall the notion of a {\em compact Lipschitz neighborhood retract}, abbreviated as {\em CLNR}, in \cite[\S 1]{Almgren62}. Recall that $\interior_M(L)$ is the relative interior of a subset $L \subset M$ (see paragraph before Definition \ref{D:rel-convex}).

\begin{lemma}[Compactness theorem for relative cycles]
\label{L:compactness}
For any given real number $M>0$, and CLNRs $K,K' \subset A$ such that $K \subset \interior_M(K')$, 
\[ \Z_{k, K, M}(A, B):=\{\tau\in\Z_{k}(A, B) \; : \; \spt(\tau)\subset K,\ \M(\tau)\leq M\}\]
is (sequentially) compact under the $\F^{K^{\pr}}$-norm.
\end{lemma}

\begin{proof}
Let $\tau_i$ be a sequence in $\Z_{k, K, M}(A, B)$ with canonical representatives $T_i \in \tau_i$. By Lemma \ref{L:can-rep}, $\spt(T_i)=\spt(\tau_i) \subset K$ for all $i$. Let $\pi:A \to B$ be the nearest point projection as in the proof of Lemma \ref{L:integral}. As $K\subset \interior_M(K^{\pr})$, we can choose $\ep_0>0$ small enough such that $\pi(K\cap\{0<\rho<\ep_0\})\subset K^{\pr}$, where $\rho$ is the intrinsic distance to $B$ as in Lemma \ref{L:integral}. For each $i$, by Lemma \ref{L:can-rep}, 
\[ \M(T_i\lc\{0<\rho<\ep_0\})\leq \M(T_i)=\M(\tau_i)\leq M,\]
and hence \cite[28.4 and 28.5]{Simon} implies the existence of some $t_i\in (0, \ep_0)$, such that the slice $\lan T_i, \rho, t_i \ran$ is an integer rectifiable $(k-1)$-current with $\M(\lan T_i, \rho, t_i \ran)\leq 2M/\ep_0$. Define 
\[ T_i^{\pr}:=T_i-{\pi}_{\sharp}(T_i\lc\{0<\rho<t_i\}),\]
then $\spt(T_i^{\pr})\subset \spt(T_i)\cup \pi (K\cap\{0\leq\rho<\ep_0\})\subset K^{\pr}$. By the proof of Lemma \ref{L:integral}, $T_i' \in \tau_i$ are integral currents and $\partial T_i^{\pr}=-{\pi}_{\sharp} \lan T_i, \rho, t_i \ran$. This implies that $\spt(\partial T_i^{\pr})\subset K^{\pr}\cap B$, $\M(\partial T_i^{\pr})\leq 2CM/\ep_0$ and
\begin{equation}
\label{E:M(T_i)}
\M(T_i^{\pr})\leq \M(T_i)+C \; \M(T_i\lc\{0<\rho<t_i\}),
\end{equation}
where $C>0$ is a constant depending only on the Lipschitz constant of $\pi$.
As $\M(T_i^{\pr})+\M(\partial T_i^{\pr})\leq C_1 M$ and $\spt(T_i^{\pr})\subset K^{\pr}$, by the compactness and boundary rectifiability theorem \cite[27.3 and 30.3]{Simon}, a subsequence of $T_i^{\pr}$ would converge weakly to an integral current $T'_{\infty}$. Since $\spt(T_i') \subset K'$ and $\spt(\partial T_i^{\pr})\subset K^{\pr}\cap B$ for all $i$, we have $\spt (T'_{\infty}) \subset K'$ and $\spt(\partial T'_{\infty}) \subset K' \cap B$. Therefore, $T'_\infty \in Z_k(A,B)$ and thus defines a class $\tau_\infty =[T'_\infty] \in \Z_k(A,B)$. 

It remains to show that $\F^{K'}(\tau_i -\tau_\infty) \to 0$ and $\tau_\infty \in \Z_{k,K,M}(A,B)$. Let $T_\infty=T'_\infty \lc (A\setminus B)$ be the canonical representative of $\tau_\infty$. Since $\spt(T_i') \subset K'$ for all $i$, we have $\spt(T_\infty') \subset K'$. As $T'_i \weakto T'_\infty$, we have $\F^{K'}(T'_i-T'_\infty) \to 0$, which implies that $\F^{K'}(\tau_i -\tau_\infty) \to 0$. 
By lower semi-continuity of $\M$ (Lemma \ref{L:lower-semicts}), we have $\M(\tau_\infty) \leq \liminf_i  \M(\tau_i) = \liminf_i \M(T_i) \leq M$.
Finally, since $\spt(\tau_i) \subset K$ for all $i$ and $\tau_i \weakto \tau_\infty$, we have $\spt(\tau_\infty) \subset K$ by Lemma \ref{L:lower-semicts}, hence $\tau_{\infty} \in\Z_{k, K, M}(A,B)$.
\end{proof}

\begin{remark}
Note that \cite[4.4.4]{Federer} discussed the compactness for compactly supported and mass bounded subsets of $Z_k(A, B)$, where they used a notion of $\F_{K, B}$ flat norm. By using the equivalence classes, the definition of our $\F^K$ norms and the proof of Lemma \ref{L:compactness} are more geometric and much simpler than those in \cite[4.4.4]{Federer}. Our compactness should have the same spirit as \cite[4.4.4]{Federer}, while our compactness under the relative flat norm is crucial in defining \emph{almost minimizing varifolds with free boundary} in Definition \ref{D:am-varifolds}.
\end{remark}

Now take $U=M$. We also need the following $\mF$-metric on $\Z_k(M, \partial M)$. When $\partial M=\emptyset$, this is the same as the one introduced by Pitts in \cite[2.1(20)]{Pitts}. For simplicity, we write $\F$ to denote $\F^M$.

\begin{definition}[$\mF$-metric]
\label{D:F-metric}
Given $\tau, \si\in\Z_k(M, \partial M)$ with canonical representatives $T\in \tau$ and $S \in \si$, the \emph{$\mF$-distance between $\tau$ and $\si$} is defined as
\[ \mF(\tau, \si):=\F(\tau-\si)+\mF(|T|, |S|).\]
where $|T|, |S|$ denote respectively the rectifiable varifolds corresponding to $T, S$, and $\mF$ on the right hand side is the varifold distance function on $\V_k(M)$ \cite[2.1(19)]{Pitts}.
\end{definition}

\begin{lemma}
\label{L:F-metric}
Let $\tau, \tau_i \in \Z_k(M, \partial M)$, $i=1,2,3,\cdots$. Then $\mF(\tau_i,\tau) \to 0$ if and only if $\F(\tau_i-\tau) \to 0$ and $\M(\tau_i) \to \M(\tau)$.
\end{lemma}

\begin{remark}
When $\partial M=\emptyset$, the result was observed by Pitts \cite[page 68]{Pitts} (see also \cite[Lemma 4.1]{Marques-Neves14}). In our case, the main difficulty is due to the use of the equivalence classes of relative cycles.
\end{remark}

\begin{proof}
We only need to check the ``if'' part. Let $T_i$ and $T$ be the canonical representative of $\tau_i$ and $\tau$ respectively. Assume that $\F(\tau_i-\tau) \to 0$ and $\M(T_i)\to \M(T)$, we need to show that $|T_i|\to |T|$ as varifolds. It suffices to prove that $\lim_{i\rightarrow\infty}T_i=T$ as currents, (together with $\M(T_i)\to \M(T)$, one directly derives $\mF(|T_i|, |T|)\to 0$ by \cite[2.1(18)(f)]{Pitts}). Apriori $\tau_i\to\tau$ does not imply that the canonical representatives converge; while it implies that there exist $T_i^{\pr}\in\tau_i$, $T^{\pr}\in\tau$, such that $T_i^{\pr}\to T^{\pr}$ as currents. If we let $\ti{T}_i=T_i^{\pr}-(T^{\pr}-T)$, then $\lim_{i \to \infty} \ti{T}_i=T$ as currents. We can show that $\lim_{i\rightarrow\infty}\ti{T}_i\lc (\partial M)=0$ using the following cutoff trick. For any $\ep>0$, since $\M(T_i)\to\M(T)$, we can find a small neighborhood $N_{\ep}$ of $\partial M$, such that $\M(T_i\lc N_{\ep})\leq \ep$, $\M(T\lc N_{\ep})\leq \ep$ for all $i$ sufficiently large. 
To see why $\lim_{i\rightarrow\infty}\ti{T}_i\lc (\partial M)=0$, note that each $\ti{T}_i$ is rectifiable and so $\ti{T}_i \lc (\partial M)$ is also a rectifiable current in $\partial M$. Moreover, for any $k$-form $\om$ on $\partial M$, one can extend it to a smooth $k$-form $\ti{\om}_{\ep}$ supported in $N_{\ep}$ such that $\|\ti{\om}_{\ep}\|\leq C\|\om\|$ for some $C$ independent of $\ep$. Then $|\ti{T}_i\lc(\partial M)(\om)-\ti{T}_i(\ti{\om}_{\ep})|=|T_i(\ti{\om}_{\ep})|\leq C\ep\|\om\|$, and $\lim\ti{T}_i(\ti{\om}_{\ep})=T(\ti{\om}_{\ep})$ with $|T(\ti{\om}_{\ep})|\leq C\ep\|\om\|$; so $\lim_{i \to \infty} \ti{T}_i\lc(\partial M)(\om)=0$ and this finishes the proof.
\end{proof}

\subsection{Two isoperimetric lemmas}
\label{SS:isoperimetic}

We prove below two important isoperimetric lemmas for equivalence classes of relative cycles, which will be used in many places in the subsequent sections. Note that Almgren also proved two similar isoperimetric lemmas in \cite{Almgren62} which are slightly different from ours. Recall that $\F=\F^M$ in the following discussions.

\begin{lemma}[$\F$-isoperimetric lemma]
\label{L:isoperimetric-F}
There exists $\ep_M>0$ and $C_M \geq 1$ depending only on the isometric embedding $M \hookrightarrow \R^L$ such that for any $\tau_1, \tau_2\in \Z_k(M, \partial M)$ with 
\[ \F(\tau_2-\tau_1) < \ep_M,\]
there exists an integral $(k+1)$-current $Q$, called an \emph{isoperimetric choice}, such that
\begin{itemize}
\item $\spt(Q) \subset M$,
\item $\spt(T_2-T_1-\partial Q)\subset\partial M$,
\item $\M(Q)\leq C_M \F(\tau_2-\tau_1)$,
\end{itemize}
where $T_1,T_2$ are the canonical representatives of $\tau_1,\tau_2$ respectively. 
\end{lemma}

\begin{proof}
Let $C_0\geq 1$ be a constant (which exists by \cite[Theorem 1.19]{Almgren62}) depending only on the isometric embedding $M \hookrightarrow \R^L$ such that
\[ \F^M(T) \leq \F^{\partial M}(T) \leq C_0 \F^M(T) \]
for any integral current $T$ (of any dimension) with $\spt(T) \subset \partial M$. Let $\nu^M, \nu^{\partial M}>0$ be the constants given by \cite[Corollary 1.14]{Almgren62} with $A=M$ and $\partial M$ respectively. We define
\[ \ep_M=\frac{1}{4C_0}\min\{ \nu^M, \nu^{\partial M}\}.\]
By Definition \ref{D:relative-cycles}, the definition of $\F=\F^M$ \cite[1.7(6)]{Almgren62} and the proof of Lemma \ref{L:compactness}, for $i=1,2$, we can find an integral $k$-current $T^{\pr}_i\in\tau_i$, an integral $(k+1)$-current $Q^{\pr}$ and an integral $k$-current $R^{\pr}$ such that (see Figure \ref{isoperimetric choice with boundary})
\begin{itemize}
\item $\spt(Q') \subset M$, $\spt(R') \subset M$,
\item $T^{\pr}_2-T^{\pr}_1=\partial Q^{\pr} +R^{\pr}$,
\item $\M(Q^{\pr})+\M(R^{\pr})\leq 2\F (\tau_1- \tau_2) < 2 \ep_M$.
\end{itemize}

\begin{figure}[h]
\centering
\includegraphics[height=1.6in]{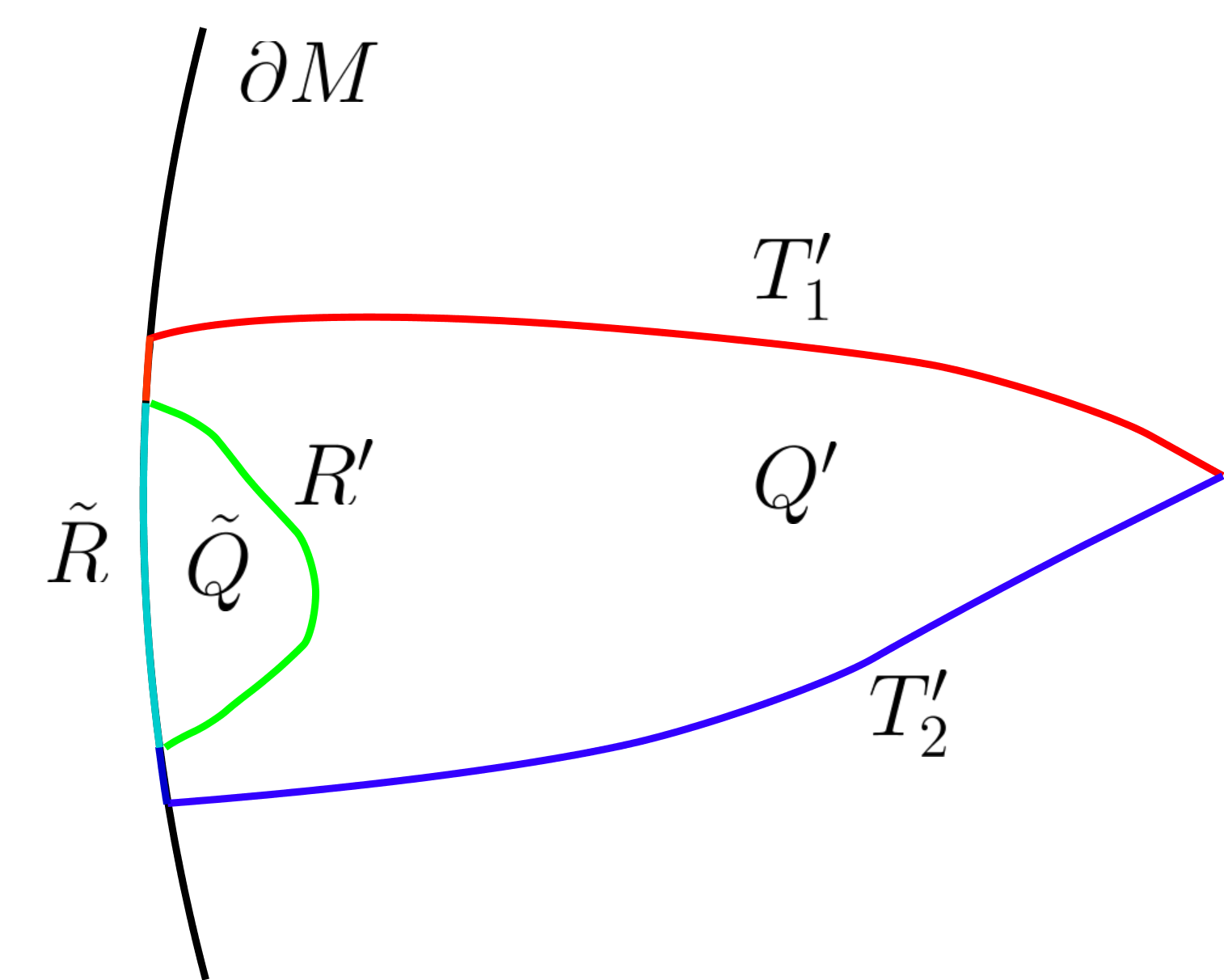}
\caption{Isoperimetric choice with boundary.}
\label{isoperimetric choice with boundary}
\end{figure}

The second point above implies that $\partial R^{\pr}=\partial T^{\pr}_2-\partial T^{\pr}_1$, hence $\spt(\partial R') \subset \partial M$.  Therefore, the integral $(k-1)$-cycle $\partial R^{\pr}$ satisfies
$$\F^{\partial M}(\partial R^{\pr})\leq C_0\F^M(\partial R^{\pr})\leq C_0 \M(R^{\pr})<\nu^{\partial M}.$$
Using \cite[Corollary 1.14]{Almgren62}, there exists an integral $k$-current $\ti{R}$ such that (see Figure \ref{isoperimetric choice with boundary})
\begin{itemize}
\item $\spt(\ti{R}) \subset \partial M$,
\item $\partial \ti{R}=\partial R^{\pr}$,
\item $\M(\ti{R})=\F^{\partial M}(\partial R^{\pr})\leq C_0\M(R^{\pr}).$
\end{itemize}
Consider the $k$-cycle $R^{\pr}-\ti{R}$ with $\spt(R^{\pr}-\ti{R}) \subset M$, note that
$$\F^M(R^{\pr}-\ti{R})\leq \M(R^{\pr})+\M(\ti{R})
\leq \nu^M.$$
Using \cite[Corollary 1.14]{Almgren62} again, there exists an integral $(k+1)$-current $\ti{Q}$ such that (see Figure \ref{isoperimetric choice with boundary})
\begin{itemize}
\item $\spt(\ti{Q}) \subset M$,
\item $\partial \ti{Q}=R^{\pr}-\ti{R}$,
\item $\M(\ti{Q})=\F^M(R^{\pr}-\ti{R})\leq \M(R^{\pr})+\M(\ti{R})\leq 2(1+C_0)\F(\tau_1- \tau_2)$.
\end{itemize}
Now, if we define our isoperimetric choice to be the integral $(k+1)$-current
\[ Q:=Q^{\pr}+\ti{Q},\]
we will check that $Q$ satisfies all the desired properties. First notice that $\spt(Q) \subset M$ and
\[ T^{\pr}_2-T^{\pr}_1-\partial (Q^{\pr}+\ti{Q})=-\ti{R}\]
which is an integral $k$-current supported on $\partial M$. 
This implies that $T_2-T_1-\partial Q$ is also supported on $\partial M$ since $\spt(T'_i-T_i) \subset \partial M$. Moreover, 
\[ \M(Q)\leq \M(Q^{\pr})+\M(\ti{Q})\leq 2(2+C_0)\F (\tau_1- \tau_2).\]
This completes the proof with $C_M=2(2+C_0)$ which depends only on the isometric embedding $M \hookrightarrow \R^L$.
\end{proof}

\begin{remark}
Notice that we do not have good control on  the mass $\M(T_2-T_1-\partial Q)$ in Lemma \ref{L:isoperimetric-F} since we do not have control of $\M(T_1-T_1')$ and $\M(T_2-T_2')$. In contrast, our next lemma has better control on $\M(T_2-T_1-\partial Q)$ but it requires a stronger assumption that $\tau_1$ and $\tau_2$ are close in the $\M$-topology.
\end{remark}

\begin{lemma}[$\M$-isoperimetric lemma]
\label{L:isoperimetric-M}
There exists $\ep_M>0$ and $C_M \geq 1$ depending only on the isometric embedding $M \hookrightarrow \R^L$ such that for any $\tau_1, \tau_2\in \Z_k(M, \partial M)$ with 
\[ \M(\tau_2-\tau_1) < \ep_M,\]
there exists an integral $(k+1)$-current $Q$ and an integer rectifiable $k$-current $R$ such that
\begin{itemize}
\item $\spt(Q) \subset M$, $\spt(R) \subset \partial M$
\item $T_2-T_1=\partial Q+R$,
\item $\M(Q)+\M(R)\leq C_M \M (\tau_2-\tau_1)$,
\end{itemize}
where $T_1,T_2$ are the canonical representatives of $\tau_1,\tau_2$ respectively.
\end{lemma}

\begin{proof}
The proof is an easy adaption of Lemma \ref{L:isoperimetric-F}. The only thing we should take care is that $\partial (T_2-T_1)$ may not be an integer rectifiable current. We can use the slicing trick as in the proof of Lemma \ref{L:compactness} to overcome this difficulty. In particular, we can take integral current representatives $T^{\pr}_i \in \tau_i$, $i=1, 2$, such that 
\begin{itemize}
\item $\M (\partial(T^{\pr}_2-T^{\pr}_1))\leq c_1 \M(T_2-T_1)$, 
\item $\M ((T_2^{\pr}-T_2^{\pr})-(T_2-T_1))\leq c_2 \M(T_2-T_1)$,
\end{itemize}
for some universal constants $c_1, c_2>0$. By \cite[Corollary 1.14]{Almgren62}, $R^{\pr}$ can be taken as the isoperimetric choice for $\partial(T_2^{\pr}-T_1^{\pr})$ in $\partial M$, and $Q$ the isoperimetric choice for $T_2^{\pr}-T_1^{\pr}-R^{\pr}$ in $M$. Finally if we define 
\[ R=R^{\pr}+(T_2-T_1)-(T_2^{\pr}-T_1^{\pr}),\]
then it is straightforward to check that $Q, R$ satisfy all the required properties.
\end{proof}

\subsection{Definitions of almost minimizing varifolds}
\label{SS:Def-am-varifolds}

We are now ready to define the notion of almost minimizing varifolds with free boundary. In what follows, we take $(A,B)=(U, U \cap \partial M)$ as before where $U\subset M$ is a relatively open subset. We also fix a compact subset $C \subset M$ containing $A$ (e.g. $C=\bar{A}$) and write $\F=\F^C$. Recall also the $\mF$-metric defined in Definition \ref{D:F-metric}.

\begin{definition}[Almost minimizing varifolds with free boundary]
\label{D:am-varifolds}
Let $\nu = \F$, $\M$ or $\mF$. For any given $\ep, \de>0$, we define $\sA_k(U; \ep, \de; \nu)$ to be the set of all $\tau\in\Z_k(M, \partial M)$ such that if $\tau=\tau_0, \tau_1, \tau_2, \cdots, \tau_m\in\Z_k(M, \partial M)$ is a sequence with:
\begin{itemize}
\item $\spt(\tau_i-\tau)\subset U$;
\item $\nu(\tau_{i+1}- \tau_i)\leq \de$;
\item $\M(\tau_i)\leq \M(\tau)+\de$, for $i=1, \cdots, m$,
\end{itemize}
then $\M(\tau_m)\geq \M(\tau)-\ep$.

We say that a varifold $V\in\V_k(M)$ is {\em almost minimizing in $U$ with free boundary} if there exist sequences $\ep_i \to 0$, $\de_i \to 0$, and $\tau_i\in \sA_k(U; \ep_i, \de_i; \F)$, such that if $T_i \in \tau_i$ are the canonical representatives, then $\mF(|T_i|, V)\leq \ep_i$.
\end{definition}

\begin{remark}
Roughly speaking, $\tau \in \sA_k(U;\ep,\de;\nu)$ means that if we deform $\tau$ by a discrete family, while keeping the perturbations to be supported in $U$, and with mass not increased too much (measured by the parameter $\de$), then at the end of the deformation the mass cannot be deformed down too much (measured by the parameter $\ep$).  Note that for $V \in \V_k(M)$ to be almost minimizing, the key point is that the varifold $V$ can be approximated by $|T_i|$ with no mass on $\partial M$ since $T_i$ are \emph{canonical representatives} of $\tau_i$.
\end{remark}

The following theorem plays an essential role when we prove the existence of almost minimizing varifolds using Almgren-Pitts' combinatorial argument (Theorem \ref{T:combinatorial}). Note that in our definition of almost minimizing varifolds with free boundary, the approximating sequence $\tau_i \in \sA_k(U; \ep_i, \de_i; \F)$ is close in the $\F$-norm. The following theorem says that we can actually use the $\M$-norm instead at the expense of shrinking the relatively open subset $U \subset M$.

\begin{theorem}
\label{T:def-equiv}
Given $V\in \V_k(M)$, then the following statements satisfy $(a)\Longrightarrow (b)\Longrightarrow (c)\Longrightarrow (d)$:
\begin{itemize}
\item[$(a)$]  $V$ is almost minimizing in $U$ with free boundary.
\item[$(b)$]  For any $\ep>0$, there exists $\de>0$ and $\tau \, \in \, \sA_k(U; \ep, \de; \mF)$ such that $\mF_U(V, |T|)<\ep$ where $T \in \tau$ is the canonical representative.
\item[$(c)$]  For any $\ep>0$, there exists $\de>0$ and $\tau\in \sA_k(U; \ep, \de; \M)$ such that $\mF_U(V, |T|)<\ep$ where $T \in \tau$ is the canonical representative.
\item[$(d)$] $V$ is almost minimizing in $W$ with free boundary for any relatively open subset $W \subset \subset U$.
\end{itemize}
\end{theorem}

\begin{proof}
The proof of Theorem \ref{T:def-equiv} is rather technical and lengthy so it is postponed to the appendix.
\end{proof}

\begin{proposition}
\label{P:stability-am}
If $V \in \V_k(M)$ is almost minimizing in a relatively open set $U \subset M$ with free boundary, then $V$ is stationary in $U$ with free boundary.
\end{proposition}

\begin{proof}
The proof is exactly the same as \cite[Theorem 3.3]{Pitts}.
\end{proof}


\section{The min-max construction}
\label{S:min-max}

In this section, we describe the min-max construction following most of the ideas in \cite{Pitts}. An important technical step is a bridge between a continuous and discrete family of (equivalence classes of) relative cycles. This is done in section \ref{SS:discret} by a discretization theorem (Theorem \ref{T:discretization}) and an interpolation theorem (Theorem \ref{T:interpolation}). Then we present the tightening process to any sweepout to ensure that almost maximal implies almost stationary. 
Finally, we establish through the combinatorial argument of Almgren and Pitts the existence of almost minimizing varifolds with free boundary (Theorem \ref{T:combinatorial}). We will focus on the case of codimension one relative cycles, i.e. $k=n$ in this section, although many parts still hold true for $2\leq k\leq n$.

\subsection{Homotopy relations}
\label{SS:homotopy}

We describe the homotopy relations introduced in \cite[\S 4.1]{Pitts}. 

\begin{definition}[Cell complex]
\label{D:cell}
We will adopt the following notations:
\begin{enumerate}
\item $I^m=[0, 1]^m$, $I^m_0=\partial I^m=I^m \setminus (0, 1)^m$.
\item For $j\in\N$, $I(1, j)$ is the cell complex of $I$, whose $1$-cells are all closed intervals of the form $[\frac{i}{3^{j}}, \frac{i+1}{3^{j}}]$, and $0$-cells are all points $[\frac{i}{3^{j}}]$. Similarly, 
\[ I(m, j)=I(1,j)^{\otimes m}=I(1, j)\otimes \cdots \otimes I(1, j)\]
defines a cell complex on $I^m$.
\item For $p\in\N$, $p\leq m$, $\al=\al_1\otimes\cdots\otimes\al_m$ is a \emph{$p$-cell} if for each $i$, $\al_i$ is a cell of $I(1, j)$, and $\sum_{i=1}^m \dim(\al_i)=p$. A \emph{$0$-cell} is called a \emph{vertex}.
\item $I(m, j)_p$ denotes the set of all $p$-cells in $I(m, j)$, and $I_0(m, j)_p$ denotes the set of $p$-cells of $I(m, j)$ supported on $I^m_0$.
\item Given a $p$-cell $\al\in I(m, j)_p$, and $k\in\N$, $\al(k)$ denotes the $p$-dimensional sub-complex of $I(m, j+k)$ formed by all cells contained in $\al$. For $q\in\N$, $q\leq p$, $\al(k)_q$ and $\al_0(k)_q$ denote respectively the set of all $q$-cells of $I(m, j+k)$ contained in $\al$, or in the boundary of $\al$.
\item The boundary homeomorphism $\partial: I(m, j)\rightarrow I(m, j)$ is given by
$$\partial(\al_1\otimes\cdots\otimes\al_m):=\sum_{i=1}^m(-1)^{\si(i)}\al_1\otimes\cdots\otimes \partial \al_i\otimes\cdots\otimes\al_m,$$
where $\si(i)=\sum_{l<i}\dim(\al_l)$, $\partial[a, b]=[b]-[a]$ if $[a, b]\in I(1, j)_1$, and $\partial[a]=0$ if $[a]\in I(1, j)_0$.
\item The distance function $\md: I(m, j)_0\times I(m, j)_0\rightarrow\N$ on vertices is defined as
$$\md(x, y):=3^{j}\sum_{i=1}^m|x_i-y_i|.$$
\item The map $\n(i, j): I(m, i)_{0}\rightarrow I(m, j)_{0}$ is defined as follows: $\n(i, j)(x)\in I(m, j)_{0}$ is the unique element of $I(m, j)_0$ such that 
$$\md\big(x, \n(i, j)(x)\big)=\inf\big\{\md(x, y): y\in I(m, j)_{0}\big\}.$$
\end{enumerate}
\end{definition}

\begin{definition}[Fineness]
\label{D:fineness}
For any $\phi: I(m, j)_{0}\rightarrow\Z_{n}(M,\partial M)$, we define the \emph{$\M$-fineness of $\phi$} to be
\begin{equation*}
\mf_\M(\phi):=\sup \left\{\frac{\M\big(\phi(x)-\phi(y)\big)}{\md(x, y)}:\ x, y\in I(m, j)_{0}, x\neq y \right\}.
\end{equation*}
\end{definition}

Note that $\mf_\M(\phi) \leq \de$ if and only if $\M(\phi(x)-\phi(y)) \leq \de$ for any $x,y \in I(m,j)_0$ with $\md(x,y)=1$. From now on, $\phi: I(m, j)_{0}\rightarrow\big(\Z_{n}(M, \partial M), \{0\}\big)$ denotes a mapping such that $\phi\big(I(m, j)_{0}\big)\subset\Z_{n}(M, \partial M)$ and $\phi|_{I_{0}(m, j)_{0}}=0$. Moreover, we use $\dom(\phi)$ to denote the domain of definition of $\phi$.

\begin{definition}[Homotopy for mappings]
\label{D:homotopy-maps}
Given $\phi_{i}: I(m, j_{i})_{0}\rightarrow\big(\Z_{n}(M, \partial M), \{0\}\big)$, $i=1, 2$, and $\de>0$, we say \emph{$\phi_{1}$ is $m$-homotopic to $\phi_{2}$ in $\big(\Z_{k}(M, \partial M), \{0\}\big)$ with $\M$-fineness $\de$}, if there exists $ j_{3}\in\N$ where $j_{3}\geq\max\{j_{1}, j_{2}\}$, and
$$\psi: I(1, j_{3})_{0}\times I(m, j_{3})_{0}\rightarrow \Z_{n}(M, \partial M),$$
such that
\begin{itemize}
\item $\mf_\M(\psi)\leq \de$;
\item $\psi([i-1], x)=\phi_{i}\big(\n(j_{3}, j_{i})(x)\big)$, $i=1, 2$;
\item $\psi\big(I(1, j_{3})_{0}\times I_{0}(m, j_{3})_{0}\big)=0$.
\end{itemize}
\end{definition}

\begin{definition}
\label{D:(m, M)-homotopy-seq}
A sequence of mappings 
$$\phi_{i}: I(m, j_{i})_{0}\rightarrow\big(\Z_{n}(M, \partial M), \{0\}\big)$$
is said be an \emph{$(m, \M)$-homotopy sequence of mappings into $\big(\Z_{n}(M, \partial M), \{0\}\big)$} if each $\phi_{i}$ is $m$-homotopic to $\phi_{i+1}$ in $\big(\Z_{n}(M, \partial M), \{0\}\big)$ with $\M$-fineness $\de_{i}>0$ such that $\de_i \to 0$ and
\[ \sup_{i}\big\{\M(\phi_{i}(x)):\ x\in \dom(\phi_i) \big\}<+\infty. \]
\end{definition}

\begin{definition}[Homotopy for sequences of mappings]
\label{D:homotopy-seq}
Let $S_{1}=\{\phi^{1}_{i}\}_{i\in\N}$ and $S_{2}=\{\phi^{2}_{i}\}_{i\in\N}$ be two $(m, \M)$-homotopy sequences of mappings into $\big(\Z_{n}(M, \partial M), \{0\}\big)$. We say that \emph{$S_{1}$ is homotopic to  $S_{2}$} if there exists a sequence $\de_{i} \to 0$ such that each $\phi^{1}_{i}$ is $m$-homotopic to $\phi^{2}_{i}$ in $\big(\Z_{n}(M, \partial M), \{0\}\big)$ with $\M$-fineness $\de_{i}$.
\end{definition}

It follows from \cite[\S 4.1.2]{Pitts} that the homotopy relation above is an equivalence relation on the space of $(m, \M)$-homotopy sequences of mapping into $\big(\Z_{n}(M, \partial M), \{0\}\big)$.

\begin{definition} 
\label{D:m-M-homotopy-class}
Denote $\pi^{\sharp}_{m}\big(\Z_n(M, \partial M, \M), \{0\}\big)$ to be the set of all equivalence classes of $(m, \M)$-homotopy sequences of mappings into $\big(\Z_n(M, \partial M), \{0\}\big)$. An equivalence class $\Pi \in \pi^{\sharp}_{m}\big(\Z_n(M, \partial M, \M), \{0\}\big)$ is called an \emph{$(m,\M)$-homotopy class of mappings into $\big(\Z_n(M, \partial M), \{0\}\big)$}.
\end{definition}

We can define similarly all the notions with $\M$ replaced by $\F$ in Definitions \ref{D:fineness} - \ref{D:m-M-homotopy-class}. Hence, we can define an $(m, \F)$-homotopy sequence of mappings into $\big(\Z_n(M, \partial M), \{0\}\big)$. In this case, we denote the set of all equivalence classes by $\pi^{\sharp}_m\big(\Z_n(M, \partial M, \F), \{0\}\big)$.

\begin{definition}[Width and critical sequence]
Let $\Pi\in\pi^{\sharp}_{m}\big(\Z_n(M, \partial M, \M), \{0\}\big)$. For any $S=\{\phi_i\}_{i \in \N} \in \Pi$, we define 
\[ \bL(S):=\limsup_{i\rightarrow\infty} \max_{x \in \dom(\phi_i)} \M(\phi_{i}(x)).\]
The \emph{width of $\Pi$} is defined as
\begin{equation}
\label{E:width}
\bL(\Pi):=\inf\{\bL(S):\ S\in\Pi\}.
\end{equation}
An $(m,\M)$-homotopy sequence $S\in\Pi$ is called a \emph{critical sequence for $\Pi$} provided that $\bL(S)=\bL(\Pi)$.
\end{definition}

\begin{lemma}
Every $\Pi\in\pi^{\sharp}_{m}\big(\Z_n(M, \partial M, \M), \{0\}\big)$ contains a critical sequence.
\end{lemma}
\begin{proof}
See \cite[4.1(4)]{Pitts}.
\end{proof}

\begin{definition}[Critical set]
Let $\Pi\in\pi^{\sharp}_{m}\big(\Z_n(M, \partial M, \M), \{0\}\big)$. For each $S \in \Pi$, we define $K(S) \subset \V_n(M)$ to be the set of all $n$-varifolds $V \in \V_n(M)$ such that
\[ V=\lim_{j \to \infty} |T_j| \qquad \text{(as varifolds)} \]
for some subsequence $\{\phi_{i_j}\} \subset \{\phi_i\}_{i \in \N}=S$, $x_j \in \dom (\phi_{i_j})$, and $T_j \in \phi_{i_j}(x_j)$ is the canonical representative.
The \emph{critical set of $S$} is then defined as
\[ C(S):=\{V \in K(S) \; : \; \|V\|(M)=\bL(S)\}.\]
\end{definition}

\begin{lemma}
\label{L:critical-set}
$C(S)$ is compact and non-empty.
\end{lemma}
\begin{proof}
See \cite[4.2]{Pitts}.
\end{proof}

\subsection{Discretization and interpolation}
\label{SS:discret}

We first prove a discretization theorem which says that a map $I^m \to \Z_{n}(M, \partial M)$ which is continuous under the $\F$-topology can be discretized into an $(m, \M)$-homotopy sequence of mappings without too much increase on the mass. The discretized mapping will be close to the original map under the $\F$-topology. Marques and Neves proved the first discretization theorem when $\partial M=\emptyset$ in \cite[Theorem 13.1]{Marques-Neves14}\cite[Theorem 3.9]{Marques-Neves17} under a technical condition called the no mass concentration assumption, i.e. condition (b) in Theorem \ref{T:discretization}. A general version of the discretization theorem without the no mass concentration assumption was proved by the second author in \cite[Theorem 5.1]{Zhou15}. In the current context, we will present a result under the no mass concentration assumption generalizing Marque-Neves' version to manifolds with boundary.

We first make a definition (c.f. \cite[\S 4.2]{Marques-Neves14} for the case of closed manifolds).

\begin{definition}
Given a map
$$\Phi: I^m\rightarrow \Z_{n}(M, \partial M),$$
continuous in the $\F$-topology, one can define a quantity
\begin{equation}
\m(\Phi, r)=\sup\big\{\| |T_x| \| (\tBcal_r(p)):\ p\in M, x\in I^m\big\},
\end{equation}
where $T_x$ is the canonical representative of $\Phi(x) \in \Z_n(M, \partial M)$, and $\tBcal_r(p)$ is the open geodesic ball in $\tM$ with radius $r>0$ centered at $p$.
\end{definition}

\begin{theorem}[Discretization Theorem]
\label{T:discretization}
Given a map 
\[ \Phi: I^m\rightarrow \Z_{n}(M, \partial M)\]
which is continuous in the $\F$-topology satisfying the following
\begin{itemize}
\item[$(a)$] $\sup_{x\in I^m}\M (\Phi(x))<\infty$,
\item[$(b)$] $\lim_{r\rightarrow 0}\m(\Phi, r)=0$,
\item[$(c)$] $\Phi|_{I^m_0}$ is continuous in the $\mF$-metric defined in Definition \ref{D:F-metric},
\end{itemize}
there exists a sequence of mappings
$$\phi_i: I(m, j_i)_0 \rightarrow \Z_n(M, \partial M),$$
with $j_i<j_{i+1}$ and a sequence of positive numbers $\de_i \to 0$  
such that 
\begin{itemize}
\item[(i)] $S=\{\phi_i\}_{i\in\N}$ is an $(m, \M)$-homotopy sequence into $\Z_{n}(M, \partial M)$ with $\M$-fineness $\mf_\M (\phi_i)<\de_i$;
\item[(ii)] There exists some sequence $k_i \to +\infty$ such that for all $x\in I(m, j_i)_0$,
\[ \M (\phi_i(x)) \leq \sup \{ \M (\Phi(y)) : \al\in I(m, k_i)_m, x, y \in\al\}+\de_i.\]
In particular, we have $\bL(S)\leq\sup_{x\in I^m}\M(\Phi(x))$;
\item[(iii)] $\sup\{\F(\phi_i(x)-\Phi(x)):\ x\in I(m, j_i)_{0}\}<\de_i$;
\item[(iv)] $\M(\phi_i(x))\leq \M(\Phi(x))+\de_i$ for all $x\in I_0(m, j_i)_0$.
\end{itemize}
\end{theorem}

\begin{remark}
If we assume that $\Phi|_{I^m_0}=0$, then $S$ can be taken to be an $(m, \M)$-homotopy sequence into $\big(\Z_{n}(M, \partial M), \{0\}\big)$.
\end{remark}

\begin{proof}
The proof is parallel to the one in \cite[Theorem 13.1]{Marques-Neves14}, and we will only point out the necessary modifications. In fact, the only place where Marques-Neves did explicit operations on integral cycles appeared in \cite[Lemma 13.4]{Marques-Neves14}. 
The proof of \cite[Lemma 13.4]{Marques-Neves14} can be straightforwardly adapted to our case using relative cycles and relative flat and mass norms by arguments similar to those in the proof of Case 1 in Lemma \ref{L:pre-interpolation-lemma}. In particular, we can use Lemma \ref{L:isoperimetric-F} in place of \cite[Corollary 1.14]{Almgren62} used on \cite[page 748]{Marques-Neves14}, and adapt the defining equation for $\psi_j$ on \cite[page 749]{Marques-Neves14} according to equation (\ref{currents connecting S_j to S_0}), then all the remaining arguments in the proof of \cite[Lemma 13.4]{Marques-Neves14} can be adapted identically to our case. Besides \cite[Lemma 13.4]{Marques-Neves14}, all other arguments in the proof of \cite[Theorem 13.1]{Marques-Neves14} are purely combinatorial, and can be identically adapted to our case by simply changing integral cycles, the usual flat norm, mass norm and $\mF$-metric to relative cycles, relative flat norm, mass norm and $\mF$-metric.
\end{proof}

The second theorem is an interpolation result which says that a discrete map from $I(m, k)_0$ into $\Z_n(M, \partial M)$ with small enough $\M$-fineness can be approximated by a continuous map in the $\M$-topology. The first interpolation result when $\partial M=\emptyset$ was proved by Marques and Neves in \cite[Theorem 14.1]{Marques-Neves14} \cite[Theorem 3.10]{Marques-Neves17} based on the work of Almgren \cite{Almgren62} and Pitts \cite{Pitts}. Here we present a result generalizing Marques-Neves' theorem to manifolds with boundary.

\begin{theorem}[Interpolation Theorem]
\label{T:interpolation}
There exists $C_0>0$ and $\de_0>0$, depending only on $m$ and the isometric embedding $M \hookrightarrow \R^L$, such that for every map
$$\psi: I(m, k)_0\rightarrow \Z_n(M, \partial M)$$
with $\f_\M(\psi)<\de_0$, there exists a map
$$\Psi: I^m\rightarrow \Z_n(M, \partial M)$$
which is continuous in the $\M$-topology such that
\begin{itemize}
\item[$(\rom{1})$] $\Psi(x)=\psi(x)$ for all $x\in I(m, k)_0$;
\item[$(\rom{2})$] for every $\al\in I(m, k)_p$, $\Psi|_{\al}$ depends only on the restriction of $\psi$ on the vertices of $\alpha$, and
\[ \max\{\M(\Psi(x)-\Psi(y)): x, y\in\al\}\leq C_0\f_\M(\psi).\]
\end{itemize}
\end{theorem}

\begin{proof}
As in \cite[Theorem 14.1]{Marques-Neves14}, we only need to consider the extension for
$$\psi: I(m, 0)_0\rightarrow \Z_n(M, \partial M)$$
with $\f_\M(\psi) < \de_0$. If $\de_0 < \ep_M$ in Lemma \ref{L:isoperimetric-M}, for every $\al\in I(m, 0)_1$ with $\partial \al=[b]-[a]$, $\psi([a])=\tau_1$, $\psi([b])=\tau_2$, we can find an integral $(n+1)$-current $Q(\al)$ supported in $M$ and an integer rectifiable $n$-current $R(\al)$ supported in $\partial M$ with
\begin{itemize}
\item $T_2-T_1=\partial Q(\al)+R(\al)$, 
\item $\M(Q(\al))+\M(R(\al))\leq C_M \M(\tau_2-\tau_1)$,
\end{itemize}
where $T_1,T_2$ are the canonical representatives of $\tau_1, \tau_2$ respectively.
The only place in \cite[Theorem 14.1]{Marques-Neves14} we need to modify to adapt to our case is the formula \cite[(86)]{Marques-Neves14}. In fact, \cite[(86)]{Marques-Neves14} is a special case of \cite[Interpolation formula 6.3]{Almgren62}. In our case, \cite[Interpolation formula 6.3]{Almgren62} can be written in the following form: given a $p$-cell $\al\in I(m, 0)_p$, for every $(x_1, \cdots, x_p)\in I^p$
\begin{eqnarray*}
 \ti{h}_{\al}(x_1, \cdots, x_p)&=& \sum_{\ga\in\Ga_{\al}}\textrm{sign}(\ga) \circ \sum_{s_1,\cdots, s_p\in \triangle} \D(s_1, x_1)\circ\cdots\circ\D(s_p, x_p)\circ\\
&& \{\partial\circ C_{\La(\ga_p)}(s_p)\circ\cdots\circ C_{\La(\ga_1)}(s_1)(Q(\ga_1)) \\
&& \phantom{aaaaaaaaa} +C_{\La(\ga_p)}(s_p)\circ\cdots\circ C_{\La(\ga_1)}(s_1)(R(\ga_1))\}.
\end{eqnarray*}
Here $\Ga_{\al}$ is the set of all sequences $\{\ga_i\}_{i=1}^p$, with $\ga_p=\al$ and $\ga_i$ is a face of $\ga_{i+1}$ with dimension $\dim(\ga_{i+1})-1$ for $1\leq i\leq q-1$; $\triangle$ is a differentiable triangulation of $(M, \partial M)$; $\D(s_i, x_i)$ is the deformation map in \cite[(82)]{Marques-Neves14}\cite[\S 4.5]{Pitts}; $C_{\La(\ga_i)}(s_i)$ is the cutting function in \cite[Theorem 5.8]{Almgren62}\cite[page 762]{Marques-Neves14}. By construction $\ti{h}_\al(x_1, \cdots, x_p)$ is an integer rectifiable $n$-current but not necessarily an integral current. Note that there is an essential typo in \cite[Definition (6.2)(3)(c)]{Almgren62}, where the formula should read 
\[ j-1\neq \sum_{q=1}^{i-1}(1-[\dim(\ga_{q+1})-\dim(\ga_q)]).\]
Note that $\ti{h}_{\al}$ is continuous in the usual mass norm. Therefore we can then define for every $p$-cell $\al\in I(m, 0)_p$ the function $h_a: I^p\rightarrow \Z_n(M, \partial M)$
\[ h_a(x_1, \cdots, x_n)=[\ti{h}_{\al}(x_1, \cdots, x_p)] \]
which is continuous in the $\M$-topology. Then we can follow \cite[(6.5)]{Almgren62} to construct $\Psi$, and follow identically as \cite[page 763]{Marques-Neves14} to finish the proof.
\end{proof}

\subsection{Tightening}
\label{SS:tightening}

In this section, we carry out a \emph{tightening process} to a critical sequence $S \in \Pi$ so that \emph{every} $V \in C(S)$ is a stationary varifold with free boundary (recall Definition \ref{D:freebdy}). We will use the ambient isotopies of $M$ to deform the relative cycles. Any such isotopy $\{f_t\}_{t \in [0,1]}$ is generated by a unique vector field $X \in \mathfrak{X}_{tan}(M)$. For each $X \in \mathfrak{X}_{tan}(M)$, we have a continuous map (with respect to the weak topology on $\V_n(M)$)
\[  H^X: [0, 1] \times \V_n(M) \rightarrow \V_n(M) \]
defined by $H^X(t, V)= (f_t)_\sharp V$ where $\{f_t\}_{t \in [0,1]}$ is the flow generated by $X$.

\begin{definition}[Pushforward map]
\label{D:pushforward}
For each $X \in \mathfrak{X}_{tan}(M)$, we define a map
\[ H^X : [0,1] \times \Z_n(M,\partial M) \to \Z_n(M,\partial M) \]
by $H^X(t,\tau)=[(f_t)_\sharp T]$ where $T \in \tau$ is the canonical representative.
\end{definition}

\begin{remark}
Since $f_t(\partial M)=\partial M$ for all $t$, it is easy to see that one can use any $T \in \tau$ (not necessarily the canonical representative) in the definition of $H^X$, and that $H^X$ is continuous with respect to the $\mF$-metric (Definition \ref{D:F-metric}).
\end{remark}

We now proceed to the main result of this subsection. For $\partial M=\emptyset$, it was first proved in \cite[Theorem 4.3]{Pitts} and later on by Marques-Neves in \cite[Proposition 8.5]{Marques-Neves14}, where they filled in a minor gap left in \cite{Pitts} using their discretization and interpolation theorems. For $\partial M\neq\emptyset$, the first author proved a similar tightening result in \cite[Proposition 5.1]{Li15} for smooth sweepouts when $n=2$. Here, we present a general tightening result for $\partial M \neq \emptyset$. 

\begin{proposition}[Tightening]
\label{P:tightening}
Let $\Pi \in \pi^{\sharp}_{m}\big(\Z_{n}(M, \partial M, \M), \{0\}\big)$. For any critical sequence $S^*$ for $\Pi$, there exists another critical sequence $S$ for $\Pi$ such that $C(S) \subset C(S^*)$ and each $V\in C(S)$ is stationary in $M$ with free boundary.
\end{proposition}

\begin{proof}
Let $S^*=\{\phi_i^*\}_{i \in \N} \in \Pi$ and $C=\sup_{x \in \dom(\phi^*_i), i \in \N} \M(\phi_i^*(x))$. Define
\[ \V^C(M):=\{V\in\V_n(M) \; : \;  \|V\|(M)\leq C\} \]
and let $\V^C_{\infty}(M)=\{V\in \V^C (M): V \text{ is stationary in $M$ with free boundary}\}$. Note that both $\V^C(M)$ and $\V^C_{\infty}(M)$ are compact under the weak topology induced by the varifold distance function $\mF$ (see \cite[Lemma 3.4]{Li15}). 

By considering only vector fields $X \in \mathfrak{X}_{tan}(M)$, we can follow \cite[Theorem 4.3]{Pitts}\cite[Proposition 8.5]{Marques-Neves14} (see also \cite[Proposition 5.1]{Li15}) to define a map $\Psi : \V^C(M)\rightarrow \mathfrak{X}_{tan}(M)$ (which is continuous with respect to the $\mF$-metric topology on $\V^C(M)$ and the $C^1$-topology on $\mathfrak{X}_{tan}(M)$) such that
\begin{itemize}
\item $\de V\big(\Psi(V)\big)=0$ if $V\in\V^C_{\infty}(M)$;
\item $\de V\big(\Psi(V)\big)<0$ if $V\in\V^C(M) \setminus \V^C_{\infty}(M)$.
\end{itemize}
Hence we can find a continuous function $h: \V^C(M)\rightarrow [0, 1]$ with $h(V)=0$ if $V\in \V^C_{\infty}(M)$, and $h(V)>0$ otherwise, such that $\|(f_s)_\sharp V \|(M)< \| (f_t)_\sharp V \|(M)$ whenever $0 \leq t<s \leq h(V)$. Here $\{f_s\}_{x \in [0,1]}$ is the flow generated by $\Psi(V) \in \mathfrak{X}_{tan}(M)$.

Furthermore we can follow exactly as \cite[pages 765-768]{Marques-Neves14} using Theorem \ref{T:discretization} and Theorem \ref{T:interpolation} in place of \cite[Theorem 13.1, Theorem 14.1]{Marques-Neves14} to construct $S=\{\phi_i\}$ out of $S^*=\{\phi^*_i\}$. In particular, if $\tau =\phi^*_i(x) \in \Z_n(M, \partial M)$ with canonical representative $T$, then $\phi_i(x)$ is roughly $H^{\Psi(|T|)}\big(h(|T|), \tau \big)$ (recall Definition \ref{D:pushforward}) after an interpolation and discretization process. Here the only difference is that we need to use Lemma \ref{L:F-metric} in place of \cite[Lemma 4.1]{Marques-Neves14}. One can straightforwardly follow \cite{Marques-Neves14} to check that $\{\phi_i\}$ satisfies the required properties.
\end{proof}

\subsection{Existence of almost minimizing varifold}
\label{SS:existence-am} 

We will need the following notation for suitable annular neighborhoods centered at a point $p \in M$. Recall the various notations of balls in section \ref{S:background}.

\begin{definition}[Annulus neighborhood]
\label{D:annulus}
Let $p \in M$ and $r>0$. Assume in addition $r < \dist_M(p,\partial M)$ if $p \notin \partial M$. We define for each $0<s<r$, the (relatively) open annular neighborhood
\[ \A_{s,r}(p)= \left\{ \begin{array}{cl}
\tBcal_r(p) \setminus \tBcal_s(p) & \text{ if } p \in M \setminus \partial M, \\
\tBcal^+_r(p) \setminus \tBcal^+_s(p) & \text{ if } p \in \partial M. 
\end{array} \right. \]
\end{definition}

\begin{definition}
\label{D:am-annuli}
A varifold $V \in \V_n(M)$ is said to be \emph{almost minimizing in small annuli with free boundary} if for each $p\in M$, there exists $r_{am}(p) >0$ such that $V$ is almost minimizing in $\mathcal{A}_{s,r}(p)$ (or $A_{s, r}(p)$) for all $0<s<r\leq r_{am}(p)$. If $p \notin \partial M$, we further require that $r_{am}(p) < \dist_M(p,\partial M)$ (or $r_{am}(p)<\dist_{\R^L}(p, \partial M)$).
\end{definition}
\begin{remark}
Note that by Lemma \ref{L:Fermi-convex}, the above definitions with respect to $\mathcal{A}_{s,r}(p)$ or $A_{s, r}(p)$ are equivalent after possibly shrinking $r_{am}(p)$.
\end{remark}

\begin{theorem}
\label{T:combinatorial}
For any $\Pi  \in \pi^{\sharp}_{m}\big(\Z_{n}(M, \partial M, \M), \{0\}\big)$, there exists $V\in\V_n(M)$ such that
\begin{itemize}
\item[(i)] $\|V\|(M)=\bL(\Pi)$;
\item[(ii)] $V$ is stationary in $M$ with free boundary;
\item[(iii)] $V$ is almost minimizing in small annuli with free boundary.
\end{itemize}
\end{theorem}

\begin{proof}
The case when $\bL(\Pi)=0$ is trivial by taking $V=0$, so we assume that $\bL(\Pi)>0$. Take any critical sequence $S^*\in\Pi$, and let $S$ be the other critical sequence obtained from $S^*$ by Proposition \ref{P:tightening}. We can follow the same procedure in the proof of \cite[Theorem 4.10]{Pitts} to prove that at least one element in $C(S)$ satisfies (i)-(iii). The proof proceeds with a  contradiction argument. If the results were not true, then we can find for each $V\in C(S)$ a point $p_V \in M$ such that $V$ is not almost minimizing in some small annulus $A_{s,r}(p_V)$. If $p_V \notin \partial M$, then we can deform $S$ homotopically to $\ti{S}$ following the same procedures as in \cite[Theorem 4.10]{Pitts}, so that $\bL(\ti{S})<\bL(S)$, which is a contradiction. Therefore, we can assume $p_V\in\partial M$. In this case, we can also adapt the proof of \cite[Theorem 4.10]{Pitts} almost identically besides the following three simple modifications. First we should use Theorem \ref{T:def-equiv} in place of \cite[Theorem 3.9]{Pitts} in \cite[page 164, Part 2]{Pitts}. (Note that \cite[Theorem 3.9]{Pitts} was essentially used but not explicitly mentioned in \cite[page 164, Part 2]{Pitts}.) Second we should use Lemma \ref{L:isoperimetric-M} in place of the use of isoperimetric choice (e.g. \cite[Corollary 1.14]{Almgren62}) in \cite[page 165, Part 5(c)]{Pitts}. Finally we should use the cut-and-paste method in the proof of Case 1 in Lemma \ref{L:pre-interpolation-lemma} in place of the cut-and-paste trick used in \cite[page 166, Part 9]{Pitts}. Then we can also deform $S$ homotopically to $\ti{S}$ with $\bL(\ti{S})<\bL(S)$ to obtain a contradiction.
\end{proof}

Now by \cite[Theorem 7.5]{Almgren62} and the simple modification used in \cite[Theorem 4.6]{Pitts}, we know that $\pi^{\sharp}_1\big(\Z_{n}(M, \partial M, \M), \{0\}\big)$ is isomorphic to the top relative integral homology group $H_{n+1}(M, \partial M) \cong \mathbb{Z}$. If $\Pi_M$ corresponds to the fundamental class $[M]$ in $H_{n+1}(M, \partial M)$, then $\bL(\Pi_M)>0$. So we have the following corollary.

\begin{corollary}
\label{C: existence of almost minimizing varifold}
There always exists a nontrivial varifold $V\in\V_n(M)$, $V\neq 0$, which satisfies properties (i)-(iii) in Theorem \ref{T:combinatorial}.
\end{corollary}

\section{Regularity of almost minimizing varifolds}
\label{S:regam}

In this section we prove the major result of this paper about the regularity of the almost-minimizing varifold $V$ from the min-max construction in Theorem \ref{T:combinatorial}. We only need to consider the regularity at the \emph{free boundary} (i.e. points in $\spt \|V\| \cap \partial M$) as the interior regularity was already proved in the seminal work of \cite{Pitts} and \cite{Schoen-Simon81} (when $n=6$). Recall that $M$ has dimension $(n+1)$. We have the following:

\begin{theorem}[Interior regularity]
\label{T:interior-regularity}
Let $2 \leq n \leq 6$. Suppose $V \in \V_n(M)$ is a varifold which is 
\begin{itemize}
\item stationary in $M$ with free boundary and 
\item almost minimizing in small annuli with free boundary,
\end{itemize}
then $\spt \|V\| \cap M \setminus \partial M$ is a smooth embedded minimal hypersurface $\Sigma$ in $M \setminus \partial M$. Furthermore, there exists $n_i \in \mathbb{N}$ such that
\[ V \lc \Gr_k(M \setminus \partial M)= \sum_{i=1}^N n_i |\Sigma_i|, \]
where $\{\Sigma_i\}_{i=1}^N$ are the connected components of $\Sigma$. (Here $N$ could be $+\infty$.)
\end{theorem}
\begin{proof}
This is a direct consequence of Theorem 3.13 and 7.12 in \cite{Pitts} and the constancy theorem \cite[2.4(5)]{Pitts}.
\end{proof}


Note that $\Sigma_i$ are hypersurfaces in $M$ \emph{without boundary}. We will show that in fact $N < \infty$ and each $\Sigma_i$ can be smoothly extended up to $\partial M$ as a hypersurface $\widetilde{\Sigma}_i$ which may now possess a free boundary lying on $\partial M$. However, as explained in section \ref{S:intro}, the extension $\widetilde{\Sigma}_i$ may not be \emph{proper} and thus could have interior points touching $\partial M$. The precise statement of our regularity result is the following:

\begin{theorem}[Main regularity]
\label{T:main-regularity}
Let $2 \leq n \leq 6$. Suppose $V \in \V_n(M)$ is a varifold which is 
\begin{itemize}
\item stationary in $M$ with free boundary and 
\item almost minimizing in small annuli with free boundary,
\end{itemize}
then there exists $N \in \N$ and $n_i \in \N$, $i=1,\cdots,N$, such that
\[ V= \sum_{i=1}^N n_i |\Sigma_i| \]
where each $(\Sigma_i,\partial \Sigma_i) \subset (M,\partial M)$ is a smooth, compact, connected, almost properly embedded free boundary minimal hypersurface.
\end{theorem}

Our proof of Theorem \ref{T:main-regularity} consists roughly of three parts. First, we show that almost minimizing varifolds have certain \emph{good replacement property} (Proposition \ref{P:good-replacement-property}), which allows the regularity theory for stable free boundary minimal hypersurfaces to be carried over to our case at hand. Second, we show that the tangent cones of an almost minimizing varifold at a point on $\partial M$ are either hyperplanes or half-hyperplanes meeting $\partial M$ orthogonally (Proposition \ref{P:tangent-cone}). Finally, we prove Theorem \ref{T:main-regularity} by establishing the regularity up to the boundary $\partial M$.

\subsection{Good replacement property}
\label{SS:replacement}

We now establish the most important property of almost minimizing varifolds which is crucial in establishing their regularity. Roughly speaking, we show that an almost minimizing varifold $V$ can be replaced by another almost minimizing varifold $V^*$ which possess better a-priori regularity properties than $V$. This idea of using replacements to establish the regularity of $V$ goes back to the seminal work of Pitts \cite[Theorem 3.11]{Pitts} (which he referred to as ``\emph{comparison surfaces}''). This part holds true for all $k$-dimensional almost minimizing varifolds. As the regularity of $V$ in $M \setminus \partial M$ is covered in Theorem \ref{T:interior-regularity}, we only need to focus on the regularity of $V$ on neighborhoods $U$ of $M$ such that $U \cap \partial M \neq \emptyset$.

\begin{proposition}[Existence of replacements]
\label{P:good-replacement-property}
Let $V\in\V_k(M)$ be almost minimizing in a relatively open set $U \subset M$ with free boundary and $K \subset U$ be a compact subset, then there exists $V^{*}\in \V_k(M)$, called \emph{a replacement of $V$ in $K$} such that
\begin{enumerate}
\item[(i)] $V\lc (M\setminus K) =V^{*}\lc (M\setminus K)$;
\item[(ii)] $\|V\|(M)=\|V^{*}\|(M)$;
\item[(iii)] $V^{*}$ is almost minimizing in $U$ with free boundary;
\item[(iv)] $V^{*} \lc U =\lim_{i \to \infty} |T_i|$ as varifolds in $U$ for some $T_i\in Z_k(M,(M \setminus U) \cup \partial M)$ such that $T_i \lc Z$ is locally mass minimizing with respect to $\interior_M(K)$ (relative to $\partial M$);
\end{enumerate}
\end{proposition}

\begin{remark}
\label{remark:replacement}
Note that by Proposition \ref{P:stability-am}, (iii) implies that $V^*$ is stationary in $U$ with free boundary. This is an important yet subtle point. Although it is clear from (iv) that $V^*$ is stationary in $\interior_M(K) \subset U$ but the stationarity holds even across $\partial_{rel}K$ with $V=V^*$ outside of $K$ by (i). 
\end{remark}

\begin{proof}
The proof consists of two parts. Throughout the proof, we denote $\F=\F^K$.

\vspace{0.3cm}
\noindent \textbf{Step 1:} \textit{A constrained minimization problem.}
\vspace{0.3cm}

Let $\ep, \de>0$ be given and fix any $\tau \in \sA_k(U;\ep,\de;\F)$. Let $\C_{\tau}$ be the set of all $\si \in \Z_k(M, \partial M)$ such that there exists a sequence $\tau=\tau_0, \tau_1, \cdots, \tau_m=\si$ in $\Z_k(M, \partial M)$ satisfying:
 \begin{itemize}
\item $\spt(\tau_i-\tau)\subset K$;
\item $\F(\tau_i - \tau_{i+1})\leq \de$;
\item $\M(\tau_i)\leq \M(\tau)+\de$, for $i=1, \cdots, m$.
\end{itemize}

\textit{Claim 1: There exists $\tau^* \in \C_{\tau}$ such that }
\[ \M(\tau^*)=\inf\{\M(\si):\ \si\in\C_{\tau}\}.\]

\vspace{0.3cm}
\textit{Proof of Claim 1:} Take any minimizing sequence $\{\si_j\}\subset\C_{\tau}$, i.e.
\[ \lim_{j \to \infty} \M(\si_j) = \inf\{\M(\si):\ \si\in\C_{\tau}\}.\]
Notice that $\spt(\si_j-\tau) \subset K$ and $\M(\si_j-\tau) \leq 2 \M(\tau) + \de$ for all $j$. By compactness theorem (Lemma \ref{L:compactness}), after passing to a subsequence, $\si_j-\tau$ converges weakly to some $\al \in\Z_k(M, \partial M)$ with $\spt(\al) \subset K$ and $\M(\al) \leq 2 \M(\tau) + \de$. We will show that $\tau^*:=\al+\tau$ is our desired minimizer. Since $\sigma_j$ converges weakly to $\tau^*$, by Lemma \ref{L:lower-semicts}, 
\begin{equation}
\label{E:tau^*}
\M(\tau^*) \leq  \inf\{\M(\si):\ \si\in\C_{\tau}\}.
\end{equation}
It remains to show that $\tau^*\in \C_{\tau}$. For $j$ sufficiently large, we have $\F(\si_j-\tau^*)<\de$. Since $\sigma_j \in \C_{\tau}$, there exists a sequence $\tau=\tau_0, \tau_1, \cdots, \tau_m=\si_j$ in $\Z_k(M, \partial M)$ satisfying the defining conditions for $\C_{\tau}$. Consider now the sequence $\tau=\tau_0, \tau_1, \cdots, \tau_m=\si_j, \tau_{m+1}=\tau^*$ in $\Z_k(M, \partial M)$, the first two conditions in the definition of $\C_{\tau}$ are trivially satisfied. Moreover, using (\ref{E:tau^*}), we also have
\[ \M(\tau^*)\leq\M(\si_j)\leq\M(\tau)+\de.\]
Therefore, $\tau^*\in\C_{\tau}$ and the claim is proved.

\vspace{0.3cm}
\textit{Claim 2: The canonical representative $T^* \in \tau^*$ is locally mass minimizing in $int_M(K)$.}
\vspace{0.3cm}

\textit{Proof of Claim 2:} For $p \in \interior_M(K) \setminus \partial M$, the proof that $T^*$ is mass minimizing in a small ball around $p$ is given in \cite[3.10]{Pitts}. For $p \in \interior_M(K) \cap \partial M$, we claim that there exists a small Fermi half-ball $\tBcal^+_r(p) \subset \interior_M(K)$ such that
\begin{equation}
\label{E:T^*}
\M(T^*)\leq \M(T^*+S),
\end{equation}
for any $S\in Z_k(A, B)$ with $\spt(S)\subset \tBcal^+_r(p)$, where $(A,B)=(\tBcal^+_r(p),\tBcal^+_r(p) \cap \partial M)$. To establish (\ref{E:T^*}), first choose $r>0$ small so that $\M(T^* \lc \tBcal^+_r(p)) < \de/2$ (this is possible since $T^*$ is rectifiable and has finite density). Suppose (\ref{E:T^*}) is false, then there exists $S \in Z_k(A, B)$ with support in $\tBcal^+_r(p)$ such that $\M(T^*+S)<\M(T^*)$. Let $\tau^{\pr}=[T^*+S] \in\Z_k(M, \partial M)$. We will show that $\tau^{\pr}\in\C_{\tau}$, which contradicts that $\tau^*$ is a minimizer in Claim 1 as 
\begin{equation}
\label{E:tau'}
\M(\tau^{\pr})\leq\M(T^*+S)<\M(T^*)=\M(\tau^*). 
\end{equation}
To see why $\tau' \in \C_\tau$, take a sequence $\tau=\tau_0, \tau_1, \cdots, \tau_m=\tau^*$ in $\Z_k(M, \partial M)$ by the definition of $\C_{\tau}$ and append $\tau_{m+1}=\tau'$ to the sequence. Since $\spt(\tau^*-\tau) \subset K$ and $\spt(\tau^{\pr}-\tau^*) \subset \spt(S)\subset K$, we have $\spt(\tau^{\pr}-\tau)\subset K$. Moreover $\F(\tau^{\pr}-\tau^*) \leq \M(\tau^{\pr}-\tau^*)$ and 
\[ \M(\tau^{\pr}-\tau^*) \leq \M((T^*+S)-T^*) < 2 \M(T^* \lc \tBcal^+_r(p)) < \de\]
Finally, by (\ref{E:tau'}), $\M(\tau^{\pr})<\M(\tau^*)\leq \M(\tau)+\de$. Therefore $\tau^{\pr}\in\C_{\tau}$ and this proves Claim 2.

\vspace{0.3cm}
\textit{Claim 3: $\tau^*\in \sA_k(U;\ep,\de;\F)$.}
\vspace{0.3cm}

\textit{Proof of Claim 3:} Suppose the claim is false. Then by Definition \ref{D:am-varifolds} there exists a sequence $\tau^*=\tau^*_0, \tau^*_1, \cdots, \tau^*_\ell $ in $\Z_k(M, \partial M)$ satisfying
\begin{itemize}
\item $\spt(\tau^*_i-\tau^*)\subset U$;
\item $\F(\tau^*_i - \tau^*_{i+1})\leq \de$;
\item $\M(\tau^*_i)\leq \M(\tau^*)+\de$, for $i=1, \cdots, \ell$,
\end{itemize}
but $\M(\tau^*_\ell) < \M(\tau^*)-\ep$. Since $\tau^* \in \C_\tau$ by Claim 1, there exists a sequence $\tau=\tau_0, \tau_1, \cdots, \tau_m=\tau^*$ satisfying the defining conditions for $\C_\tau$. Then the sequence $\tau=\tau_0,\tau_1,\cdots,\tau_m, \tau^*_1, \cdots, \tau^*_\ell$ in $\Z_k(M, \partial M)$ still satisfies the three conditions above as $\M(\tau^*)\leq \M(\tau)$. Therefore $\tau\in \sA_k(U;\ep,\de;\F)$ implies that we have $\M(\tau^*_\ell)\geq \M(\tau)-\ep\geq \M(\tau^*)-\ep$, which is a contradiction. This proves Claim 3.

\vspace{0.3cm}
\noindent \textbf{Step 2:} \textit{Construction of the replacement $V^*$.}

Let $V\in \V_k(M)$ be almost minimizing in $U$ with free boundary. By definition there exists a sequence $\tau_i\in \sA_k(U; \ep_i, \de_i; \F)$ with $\ep_i, \de_i \to 0$ such that $V$ is the varifold limit of $|T_i|$, where $T_i \in \tau_i$ is the canonical representative. By Step 1 we can construct a minimizer $\tau_i^* \in \C_{\tau_i}$ for each $i$ with canonical representative $T_i^* \in \tau_i^*$. Since $\M(T_i^*)=\M(\tau_i^*)$ is uniformly bounded, by compactness there exists a subsequence $|T_i^*|$ converging as varifolds to some $V^* \in \V_k(M)$. We claim that $V^*$ satisfies (i)-(iv) in Proposition \ref{P:good-replacement-property} and thus is our desired replacement. 

First, by Claim 1 of Step 1, $\tau_i^* \in \C_{\tau_i}$ and thus $\spt(\tau_i^*- \tau_i) \subset K$. Since $T_i^*-T_i \in \tau_i^*-\tau_i$ is the canonical representative, we have $\spt(T_i^* -T_i) \subset K$, i.e. $T_i^*\lc (M \setminus  K)=T_i\lc (M \setminus K)$. Hence the varifold limits satisfy $V^* \lc (M \setminus K) = V \lc (M \setminus K)$. Second, as $\tau_i \in G_k(U,\ep_i,\de_i,\F)$ and $\tau_i^* \in \C_{\tau_i}$, we have 
\[  \M(\tau_i)-\ep_i\leq \M(\tau_i^*)\leq \M(\tau_i),  \]
thus $\M(T_i)-\ep_i\leq \M(T_i^*)\leq \M(T_i)$. Taking $i \to \infty$, we have $\|V\|(M)=\|V^{*}\|(M)$. Since each $\tau_i^* \in \sA_k(U; \ep_i, \de_i; \F)$ by Claim 3 above, by definition $V^*$ is almost minimizing in $U$ with free boundary. Finally, (iv) follows from Claim 2 above. 
\end{proof}

Property (iv) in Proposition \ref{P:good-replacement-property} says that the replacement $V^*$ is a varifold limit of locally area minimizing integral currents in $\interior_M(K)$, which possess nice regularity properties when $k=n$. We will make use of this and the compactness of stable minimal hypersurfaces with free boundary in Theorem \ref{T:freebdy-cpt} to prove the regularity of $V^*$.

\begin{lemma}[Regularity of replacement]
\label{L:reg-replacement}
Let $2 \leq k=n \leq 6$. Under the same hypotheses of Proposition \ref{P:good-replacement-property} and assume further that $\interior_M(K)$ is simply connected, then the restriction of the replacement $V^* \lc \interior_M(K)$ is integer rectifiable and if $\Sigma:=\spt \|V^*\| \cap \interior_M(K)$, then $(\Sigma,\partial \Sigma) \subset (\interior_M(K),\interior_M(K) \cap \partial M)$ is a smooth, stable, almost properly embedded, free boundary minimal hypersurface.
\end{lemma}

\begin{proof}
By the regularity of locally area minimizing currents in codimension one \cite[Theorem 4.7]{Gr87} (although the proofs by Gr\"uter were only written when the ambient space is the Euclidean space, it is straightforward to adapt to the case of a Riemannian manifold), the $T_i$'s in (iv) of Proposition \ref{P:good-replacement-property} are (integer multiples of) smooth, properly embedded, free boundary minimal hypersurfaces in $\interior_M(K)$. 

\vspace{0.3cm}
\textit{Claim: $T_i$ is stable in $\interior_M(K)$.}

\textit{Proof of Claim:} Write $T=T_i$ and suppose on the contrary that $T$ is not stable. Since $T$ is properly embedded in $\interior_M(K)$, there exists a smooth one parameter family $\{T_s\}_{s \in (-s_0,s_0)}$ of properly embedded hypersurfaces such that $T_0=T$ and $T_s=T$ outside $\interior_M(K)$ and that $\M(T_s) < \M(T)$ for all $s \in (-s_0,s_0)$. We would show that this contradicts the fact that $[T]=\tau^*$ is a minimizer for the constrained minimization problem in Step 1 of Proposition \ref{P:good-replacement-property}. As each $T_s$ is properly embedded, they are the canonical representative for the equivalence classes $\tau_s:=[T_s]$ with $\M(\tau_s)=\M(T_s)<\M(T)$. It suffices to show that $\tau_s \in \C_\tau$ for $s$ sufficiently close to $0$. Since $T_s=T$ outside $\interior_M(K)$ and $\M(T_s)<\M(T)$, we clearly have $
\spt(\tau_s-\tau) \subset K$ and $\M(\tau_s) \leq \M(\tau)+\delta$. On the other hand, as $T_s \to T$ in the classical flat norm, by definition of the relative flat norm we have $\F(\tau_s -\tau) \to 0$ as $s \to 0$. Therefore, $\F(\tau_s -\tau) \leq \delta$ and thus $\tau_s \in \C_\tau$ as long as $s$ is sufficiently small, which gives the desired contradiction.

\vspace{0.3cm}
Finally the regularity of $V^*\lc \interior_M(K)$ follows from the compactness theorem (Theorem \ref{T:freebdy-cpt}). Note that even though each $T_i$ is properly embedded, the properness may be lost after taking the limit as in Theorem \ref{T:freebdy-cpt}. Therefore, the final replacement $V^*$ may not be \emph{properly} embedded anymore.
\end{proof}

\subsection{Tangent cones and rectifiability}
\label{SS:tangent-cone}

Now we go back to the case $k=n$. We make use of the good replacement property, Proposition \ref{P:good-replacement-property} and Lemma \ref{L:reg-replacement}, to show that $V$ is rectifiable. Furthermore, we classify the tangent cones of $V$ at \emph{every} $p \in \spt \|V\| \cap \partial M$.

\begin{lemma}[Uniform volume ratio bound]
\label{L:uniform-density-bdd}
Let $2 \leq n \leq 6$. Suppose $V \in \V_n(M)$ is almost minimizing in small annuli and stationary in $M$ with free boundary. There exists a constant $c_1>1$ depending only on the isometric embedding $M \hookrightarrow \R^L$ such that
\begin{equation}
\label{E:density-bdd}
c_1^{-1} \leq \frac{\|V\|(B_\rho(p))}{\omega_n \rho^n} \leq c_1  \frac{\|V\|(B_{\rho_0}(p))}{\omega_n \rho_0^n}
\end{equation}
for all $p \in \spt\|V\| \cap \partial M$, $\rho \in (0,\rho_0)$ where $\rho_0 = \frac{1}{4}\min\{r_{am}(p),r_{mono},r_{Fermi}\}$. Here, $r_{am}(p),r_{mono},r_{Fermi}>0$ are as in Definition \ref{D:am-annuli}, Theorem \ref{T:monotonicity} and Lemma \ref{L:Fermi-convex} respectively. In particular, $\Theta^n(\|V\|,p) \geq \theta_0>0$ for some constant $\theta_0>0$ at all $p \in \spt\|V\| \cap \partial M$.
\end{lemma}

\begin{proof}
We will first prove the second assertion that $\Theta^k(\|V\|,p) \geq \theta_0>0$ for some constant $\theta_0>0$ (depending only on the isometric embedding $M \hookrightarrow \R^L$) at all $p \in \spt\|V\| \cap \partial M$. The uniform volume ratio bound (\ref{E:density-bdd}) then follows easily from the monotonicity formula in Theorem \ref{T:monotonicity}, where the constant $c_1>0$ depends only on the isometric embedding $M \hookrightarrow \mathbb{R}^L$. 

Fix any $p \in \spt \|V\| \cap \partial M$, let $\rho_0=\frac{1}{4}\min\{r_{am}(p),r_{mono},r_{Fermi}\}$ and $r \in (0,\rho_0)$. Since $V$ is almost minimizing in $A_{r/2,4r}(p) \cap M$ with free boundary, we can apply Proposition \ref{P:good-replacement-property} to obtain a replacement $V^*$ of $V$ in $K=\Clos(A_{r,2r}(p)) \cap M$. First of all, notice that 
\begin{equation}
\label{E:nonzero-in-annulus}
\|V^*\| \lc A_{r,2r}(p) \neq 0.
\end{equation}
Otherwise, there exists a smallest number $s \in (0,r)$ such that $\spt \|V^*\| \cap B_{2r}(p)$ is contained in a closed Fermi half-ball $\Clos(\tBcal^+_s(p))$, contradicting the maximum principle (Theorem \ref{T:max-principle}) since $\tBcal^+_s(p)$ is relatively convex by Lemma \ref{L:Fermi-convex}. Note that $V^* \lc B_r(p)=V \lc B_r(p) \neq 0$ by Proposition \ref{P:good-replacement-property} (i) since $p \in \spt \|V\|$. By the regularity of replacements (Lemma \ref{L:reg-replacement}), if we let $\Si:=\spt \|V^*\| \cap A_{r,2r}(p)$, then 
\[ (\Si, \partial \Si) \subset (A_{r,2r}(p) \cap M, A_{r,2r} \cap \partial M) \]
is a smooth, stable almost properly embedded free boundary minimal hypersurface. Moreover, $\Sigma \neq \emptyset$ by (\ref{E:nonzero-in-annulus}). If $\partial \Sigma =\emptyset$, then $\Sigma$ is a stationary varifold in $A_{r,2r}(p) \cap \tM$ and the assertion follows from the arguments for the interior case in \cite[3.13]{Pitts}. If $\partial \Sigma \neq \emptyset$, there exists $q \in \partial \Sigma$ with $\Theta^n(\|V^*\|,q) \geq 1/2$ by Lemma \ref{L:reg-replacement}. Note that $B_{2r}(q) \subset B_{4r}(p)$. By the monotonicity formula (Theorem \ref{T:monotonicity}) and Proposition \ref{P:good-replacement-property} (i) and (ii), we have
\begin{displaymath}
\begin{split}
\frac{\|V\|(B_{4r}(p))}{\omega_n (4r)^n} & = \frac{\|V^*\|(B_{4r}(p))}{\omega_n (4r)^n} \geq \frac{\|V^*\|(B_{2r}(q))}{\omega_n (4r)^n} \geq \frac{1}{2^nC_{mono}} \lim_{s\rightarrow 0}\frac{\|V^*\|(B_{s}(q))}{\omega_n s^n}\\
                                             & = \frac{1}{2^nC_{mono}} \Theta^n(\|V^*\|,q) \geq \frac{1}{2^{n+1}C_{mono}}>0.
\end{split}
\end{displaymath}
Since the above inequality holds for any $0<r < \rho_0$, by letting $r \to 0$, we get the uniform lower bound on the density $\Theta^n(\|V\|,p)$ by taking $\theta_0=2^{-(n+1)}C^{-1}_{mono}$.
\end{proof}

\begin{remark}
We cannot deduce from Lemma \ref{L:uniform-density-bdd} yet that $V$ is rectifiable as we do not know whether $V$ has locally bounded first variation as a varifold in $\mathbb{R}^L$. This is an additional difficulty which does not appear in the interior regularity theory since any stationary varifold $V \in \V_k(M)$ has locally bounded first varifold as a varifold in $\R^L$ when $M$ is a \emph{closed} submanifold in $\R^L$ by \cite[Remark 4.4]{Allard72}. We will deal with this extra difficulty later in Proposition \ref{P:tangent-cone}.
\end{remark}

\begin{lemma}[Rectifiability of tangent cones]
\label{L:tangent-cone}
Let $2 \leq n \leq 6$. Suppose $V \in \V_n(M)$ is almost minimizing in small annuli and stationary in $M$ with free boundary. For any $C \in \VarTan(V,p)$ where $p \in \spt \|V\| \cap \partial M$, we have
\begin{itemize}
\item[(a)] $C \in \RV_n(T_pM)$,
\item[(b)] $C$ is stationary in $T_pM$ with free boundary,
\item[(c)] $\bmu_{r \sharp} C=C$ for all $r>0$.
\end{itemize}
In other words, $C$ is a stationary rectifiable cone in $T_pM$ with free boundary.
\end{lemma}

\begin{proof}
Let $V_i=(\bleta_{p,r_i})_\sharp V$ for some $r_i \to 0$ such that $C=\lim V_i$ as varifolds. Conclusion (b) is trivial as the notion of stationary with free boundary is invariant under translation and scaling in $\mathbb{R}^L$ and that $\bleta_{p,r_i}(M)$ converges smoothly to $T_pM$. 

\vspace{0.3cm}

\textit{Claim 1: $\spt \|V_i\|$ converges to $\spt \|C\|$ in the Hausdorff topology.}

\vspace{0.3cm}

\textit{Proof of Claim 1:} 
If $q_i \in \spt \|V_i\|$ is a sequence converging to some $q$, we show that $q \in \spt \|C\|$. In this situation, either we can take all $q_i \in \bleta_{p,r_i}(\partial M)$ or else $\spt \|V_i\| \cap \bleta_{p,r_i}(\partial M) =\emptyset$ in some fixed neighborhood of $q$ for all $i$ sufficiently large. In the first case, by the lower bound on volume ratio (\ref{E:density-bdd}) we have: for each fixed $\si>0$, as long as $q_i \in B_{\si/2}(q)$ (which holds for $i$ sufficiently large),
\[ c_1^{-1} \leq \frac{\|V_i\|(B_{\si/2}(q_i))}{\omega_n (\si/2)^n} \leq \frac{\|V_i\|(B_\si(q))}{\omega_n (\si/2)^n}.\]
Letting $i \to \infty$, we have $\|C\|(B_\si(q)) \geq c_1^{-1} \omega_n (\si/2)^n >0$ for all $\si>0$. Hence, $q \in \spt \|C\|$. In the second case each $V_i$ is indeed stationary in $B_\sigma(q) \cap \bleta_{p,r_i}(\tM)$ for which the argument is standard \cite[(5.8)]{Schoen-Simon81}. 

\vspace{0.3cm}

\textit{Claim 2: $\Theta^n(\|C\|,q) \geq \theta_0>0$ for all $q \in \spt \|C\|$.}

\vspace{0.3cm}

\textit{Proof of Claim 2:} It suffices to consider $q \in T_p(\partial M)$ since $C$ is integer rectifiable in $T_pM \setminus T_p(\partial M)$ by Theorem \ref{T:interior-regularity} and hence $\Theta^n(\|C\|,q) \geq 1$ at every $q \in \spt \|C\| \setminus T_p(\partial M)$. By claim 1, there exists $q_i \in \spt \|V_i\|$ such that $q_i \to q$ and we can divide into two cases as in claim 1 where each $q_i \in \bleta_{p,r_i}(\partial M)$ or $\spt \|V_i\| \cap \bleta_{p,r_i}(\partial M) =\emptyset$ in some fixed neighborhood of $q$. The first case follows from the uniform lower bound (\ref{E:density-bdd}) and the second case is standard \cite[Theorem 40.6]{Simon}.

\vspace{0.3cm}
To prove (a) and (c), we consider the doubled varifold $\overline{C}$ which is stationary in $T_p\tM$ by (b) and Lemma \ref{L:reflection}. We now argue that $\Theta^n(\|\overline{C}\|,q)>0$ for $\|\overline{C}\|$-a.e. $q$, which by the standard Rectifiability Theorem \cite[42.4]{Simon} would imply that $\overline{C}$ is rectifiable. Notice that 
\[ \Theta^n(\|\overline{C}\|,q)=\Theta^n(\|\overline{C}\|,\theta_{\nu_{\partial M}(p)}(q))=\Theta^n(\|C\|,q) \text{ for all $q \in T_pM \setminus T_p(\partial M)$,} \]
\[ \Theta^n(\|\overline{C}\|,q)=2 \Theta^n(\|C\|,q) \text{ for all $q \in T_p(\partial M)$.} \]
Hence $\overline{C}$ is rectifiable by claim 2 and thus $C$ is also rectifiable. Finally, (c) then follows from \cite[Theorem 19.3]{Simon}.




\end{proof}

\begin{remark}
\label{R:tangent-cone}
By the same argument, one can show that (a) and (b) still hold if $C$ is the varifold limit of $(\bleta_{p,r_i})_\sharp V_i$ for a sequence $r_i \to 0$ and $V_i \in \V_n(M)$ which is almost minimizing in small annuli and stationary in $M$ with free boundary, satisfying the uniform density bound (\ref{E:density-bdd}).
\end{remark}

We now prove the main result of this subsection on the classification of tangent cones and rectifiability of almost minimizing varifolds with free boundary.

\begin{proposition}[Classification of tangent cones]
\label{P:tangent-cone}
Let $2 \leq n \leq 6$. Suppose $V \in \V_n(M)$ is almost minimizing in small annuli and stationary in $M$ with free boundary. Then for any $C \in \VarTan(V,p)$ with $p \in \spt\|V\| \cap \partial M$, either 
\begin{itemize}
\item[(i)] $C=\Theta^n(\|V\|,p) \; |T_p(\partial M)|$ where $\Theta^n(\|V\|,p) \in \mathbb{N}$, or 
\item[(ii)] $C=2 \Theta^n (\|V\|,p) \; |S \cap T_pM|$ for some $S \in \Gr(L,n)$ such that $S \subset T_p \tM$ and $S \perp T_p (\partial M)\}$, and $2 \Theta^n(\|V\|,p) \in \mathbb{N}$.
\end{itemize}
Moreover, for $\|V\|$-a.e. $p \in \partial M$, the tangent varifold of $V$ at $p$ is unique, and the set of $p \in \partial M$ in which case (ii) occurs as its unique tangent cone has $\|V\|$-measure zero; hence $V$ is rectifiable.
\end{proposition}

\begin{remark}
It is important to note that up to now we do not know that the tangent cone is unique at \emph{every} $p \in \spt \|V\| \cap \partial M$. In fact we do not even know whether both (i) and (ii) could occur at such a point. This is a technical difficulty which will eventually be overcome until much of the regularity results have been established in the next subsection.
\end{remark}

\begin{proof}
Let $p \in \spt\|V\| \cap \partial M$ and $r_i \to 0$ such that $C \in \VarTan(V,p)$ is the varifold limit
\[ C= \lim_i (\bleta_{p,r_i})_\sharp V \in \V_n(T_pM).\]
Assume WLOG that $r_i < \frac{1}{4}\min\{r_{am}(p),r_{mono},r_{Fermi}\}$ for all $i$. 
By Lemma \ref{L:tangent-cone}, $C$ is a rectifiable cone which is stationary in $T_pM$ with free boundary. 

Let $\al\in(0, 1/4)$. For each $i$, we can apply Proposition \ref{P:good-replacement-property} with the compact set $K_i=\Clos(A_{\al r_i, r_i}(p)) \cap M$ to obtain a replacement $V^*_i \in \V_n(M)$ of $V$ in $K_i$. As in (\ref{E:nonzero-in-annulus}), $\|V^*_i\| \lc A_{\al r_i,r_i}(p) \neq 0$. By Proposition \ref{P:good-replacement-property} (iii) and Remark \ref{remark:replacement}, each $V^*_i$ is almost minimizing in small annuli and stationary in $M$ with free boundary. By Proposition \ref{P:good-replacement-property} (ii) and the compactness of Radon measures, after passing to a subsequence, we obtain a limit as varifolds
\[ D:=\lim_{i\to \infty} (\bleta_{p,r_i})_\sharp V^*_i   \in \V_n(T_pM). \]
By Remark \ref{R:tangent-cone}, $D$ is rectifiable and stationary in $T_pM$ with free boundary. 

\vspace{.3cm}

\textit{Claim 1:  Let $\Si_\infty=\spt\|D\| \cap A_{\al, 1}(0)$, then 
\[ (\Sigma_\infty,\partial \Sigma_\infty) \subset ( T_p M,T_p (\partial M))\]
is an almost properly embedded hypersurface which is smooth and stable in $A_{\al,1}(0)$.} 

\vspace{.3cm}

\textit{Proof of Claim 1: } Let $\Si_i^*:=\spt \|V^*_i\| \cap A_{\al r_i, r_i}(p)$. By Lemma \ref{L:reg-replacement}, 
\[ (\Si_i^*,\partial \Si_i^*) \subset (M, \partial M)\]
is an almost properly embedded hypersurface which is smooth and stable inside $A_{\al r_i, r_i}(p) \cap M$.
Consider the blow-ups $\Si_i:=\bleta_{p,r_i}(\Si_i^*)$ such that 
\[ (\Si_i,\partial \Si_i) \subset ( \bleta_{p,r_i}(M),\bleta_{p,r_i}(\partial M))\]
is an almost properly embedded hypersurface which is smooth and stable inside $A_{\al,1}(0) \cap \bleta_{p,r_i}(M)$. By the monotonicity formula (Theorem \ref{T:monotonicity}), Proposition \ref{P:good-replacement-property} (i) and (ii), Lemma \ref{L:reg-replacement}, we have that $\Sigma_i$ has uniformly bounded mass:
\[ \mathcal{H}^n(\Si_i)= \frac{\mathcal{H}^n(\Si^*_i)}{r_i^n}\leq \frac{\|V^*_i\|(B_{r_i}(p))}{r_i^n}=  \frac{\|V\|(B_{r_i}(p))}{r_i^n}\leq C_{mono} \frac{\|V\|(B_{r_1}(p))}{r_1^n}.\]
By Theorem \ref{T:freebdy-cpt} (which still holds for a sequence of converging metrics), after passing to a subsequence, 
$(\Si_i,\partial \Si_i)$ converges (with multiplicity) to some almost properly embedded 
\[ (\Si_\infty,\partial \Si_\infty) \subset (T_p M,T_p(\partial M))\]
which is smooth and stable in $A_{\al, 1}(0)\cap T_p M$. 
By construction, it is clear that $\spt \|D\| \cap A_{\al, 1}(0)=\Si_\infty$ and hence \emph{Claim 1} is proved. 

Now consider the doubled varifolds $\overline{C}, \overline{D} \in \V_n(T_p\tM)$ 
of $C,D \in \V_n(T_pM)$ as in Lemma \ref{L:reflection}. We want to show that $\overline{C}$ and $\overline{D}$ coincide.

\vspace{.3cm}

\textit{Claim 2: $\overline{C}=\overline{D}$.}

\vspace{.3cm}

\textit{Proof of Claim 2:} 
By Lemma \ref{L:reflection}, $\overline{C}$ is stationary cone in $T_p \tM$ and hence
\[ \frac{\|\overline{C}\|(B_r(0))}{\omega_n r^n} \equiv \text{ constant}\]
By Lemma \ref{L:reflection} again, $\overline{D}$ is also stationary in $T_p \tM$. Using Proposition \ref{P:good-replacement-property} (i) and (ii), we have $\|\overline{D}\|(B_2(0))=\|\overline{C}\|(B_2(0))$ and 
\[ \overline{D}\lc (\mathbb{R}^L \setminus \Clos(A_{\al, 1}(0)))=\overline{C}\lc (\R^{L}\setminus \Clos(A_{\al, 1}(0))).\]
The classical monotonicity formula then implies that 
\[ \frac{\|\overline{D}\|(B_r(0))}{\omega_n r^n} \equiv \text{ constant}\]
and thus $\overline{D}$ is also a cone as it is rectifiable (see the proof of \cite[Theorem 19.3]{Simon}). Since both $\overline{C}$ and $\overline{D}$ are cones and they agree outside $A_{\al,1}(0)$, we must have $\overline{C}=\overline{D}$. This completes the prove of \emph{Claim 2}.

By Claim 2 we are left to prove that $\overline{D}$ is a hyperplane in $T_p\tM$. Since the double of a stable minimal hypersurface with free boundary is a stable minimal hypersurface \emph{without boundary}, we have that $\overline{\Si}_\infty:=\spt \|\overline{D}\| \cap A_{\al,1}(0)$ is a smooth and stable hypersurface without boundary in $A_{\al, 1}(0)$ (even if there exists false boundary points). Therefore, $\spt \| \overline{D}\|$ is a stable properly embedded minimal cone in $T_p \tM \subset \R^L$ which is smooth except possibly at the vertex. By the non-existence of stable cones in dimensions $2 \leq n \leq 6$ (\cite{Simons} or \cite[Theorem 7.6]{Pitts}), $\spt \| \overline{D}\|$ must be a hyperplane $P \subset T_p \tM$. Since $P$ is invariant under the refection $\theta_{\nu_{\partial M}(p)}$ by construction, we conclude that $C$ must be one of the two types in (i) or (ii).

\vspace{.3cm}

\textit{Claim 3: The tangent cone of $V$ is unique at $\|V\|$-a.e. $p$.}

\vspace{.3cm}

\textit{Proof of Claim 3:} Assume $p \in \spt\|V\| \cap \partial M$. By \cite[Lemma 38.4]{Simon} for $\|V\|$-a.e. $x \in M$, there exists a Radon measure $\eta_V^x$ on $\Gr (L, n)$ such that for any continuous function $\beta$ on $\Gr (L, n)$,
\begin{equation}
\label{E: eta measure}
\int_{\Gr(L, n)}\beta(S)\; d\eta_V^x(S)=\lim_{\rho\to 0}\frac{\int_{\Gr_n(B_\rho(x))}\beta(S) \; dV(y, S)}{\|V\|(B_\rho(x))}.
\end{equation}
Fix any such $p \in \spt \|V\| \cap \partial M$ such that $\eta_V^p$ exists. If $C\in \VarTan(V,p)$ belongs to case (i), then (\ref{E: eta measure}) implies that 
\[ \int_{\Gr(L, n)}\beta(S) \; d\eta_V^p(S)=\beta(T_p(\partial M)).\]
If $C\in \VarTan(V,p)$ belongs to case (ii), then we have instead 
\[ \int_{\Gr(L, n)}\beta(S) \; d\eta_V^p(S)=\beta(P).\]
Since the right hand side of the two equations above agree for all $\beta$. It then follows that the tangent cones are unique at such $p$. 

\vspace{.3cm}

\textit{Claim 4: $V$ is rectifiable.}

\vspace{.3cm}

\textit{Proof of Claim 4:} Let $A$ be the set of $p \in \spt \|V\| \cap \partial M$ at which the tangent cone of $V$ is unique and belongs to case (ii). We claim that $\|V\|(A)=0$. Suppose not, by \cite[Theorem 3.2 (2)]{Simon}, $\Theta^{*n}(\|V\|\lc A, \cdot)>0$ on a subset of $A$ with positive $\|V\|$-measure. Choose $x\in A$, such that $\Theta^{*n}(\|V\|\lc A, x)>0$ and the tangent varifold of $V$ at $x$ is unique (which is possible by Claim 3) and given by $C=2\Theta^n(\|V\|, x) |S \cap T_x M|$ as in (ii). Since $A \subset \partial M$, we have
\begin{displaymath}
\begin{split}
\Theta^{*n}(\|V\|\lc A, x) & =\lim_{r\rightarrow 0}\frac{\|V\|(B_r(x)\cap A)}{\omega_n r^n}\leq \lim_{r\rightarrow 0}\frac{\|V\|(B_r(x)\cap \partial M)}{\omega_n r^n}\\
                                      & = \|C\| (B_1(0)\cap T_x (\partial M))=0,
\end{split}
\end{displaymath}
contradicting $\Theta^{*n}(\|V\|\lc A, x)>0$. Hence $\|V\|(A)=0$. As $V$ has a tangent space with positive multiplicity (Lemma \ref{L:uniform-density-bdd} and Theorem \ref{T:interior-regularity}) $\|V\|$-a.e., $V$ is rectifiable by \cite[Theorem 38.3]{Simon}.
\end{proof}

\subsection{Regularity of almost minimizing varifolds}
\label{SS:regularity-am}

We now prove our main regularity theorem (Theorem \ref{T:main-regularity}) for an almost minimizing varifold $V$ at points on $\partial M$.

\begin{proof}[Proof of Theorem \ref{T:main-regularity}]
Let $p \in \spt \|V\| \cap \partial M$. Fix $r>0$ small so that
\[ r < \frac{1}{4}\min\{r_{am}(p),r_{mono},r_{Fermi}\}. \]
By the same argument as in (\ref{E:nonzero-in-annulus}), we have the following: if $W \in \V_n(M)$ is stationary in $\tBcal^+_r(p)$ with free boundary and $W \neq 0$ in $\tBcal^+_r(p)$, then (recall Definition \ref{D:Fermi-balls} and \ref{D:am-annuli})
\begin{equation}
\label{E:max-principle}
\emptyset \neq \spt \|W\| \cap \tScal^+_t(p) = \textrm{Clos}\big[ \spt \|W\| \setminus \Clos(\tBcal^+_t(p)) \big] \cap \tScal^+_t(p).
\end{equation}

\vspace{.3cm}

\noindent \textbf{Step 1:} \textit{Constructing successive replacements $V^*$ and $V^{**}$ on two overlapping annuli. (See Figure \ref{consecutive replacement})}

\begin{figure}[h]
    \centering
        \includegraphics[width=0.4\textwidth]{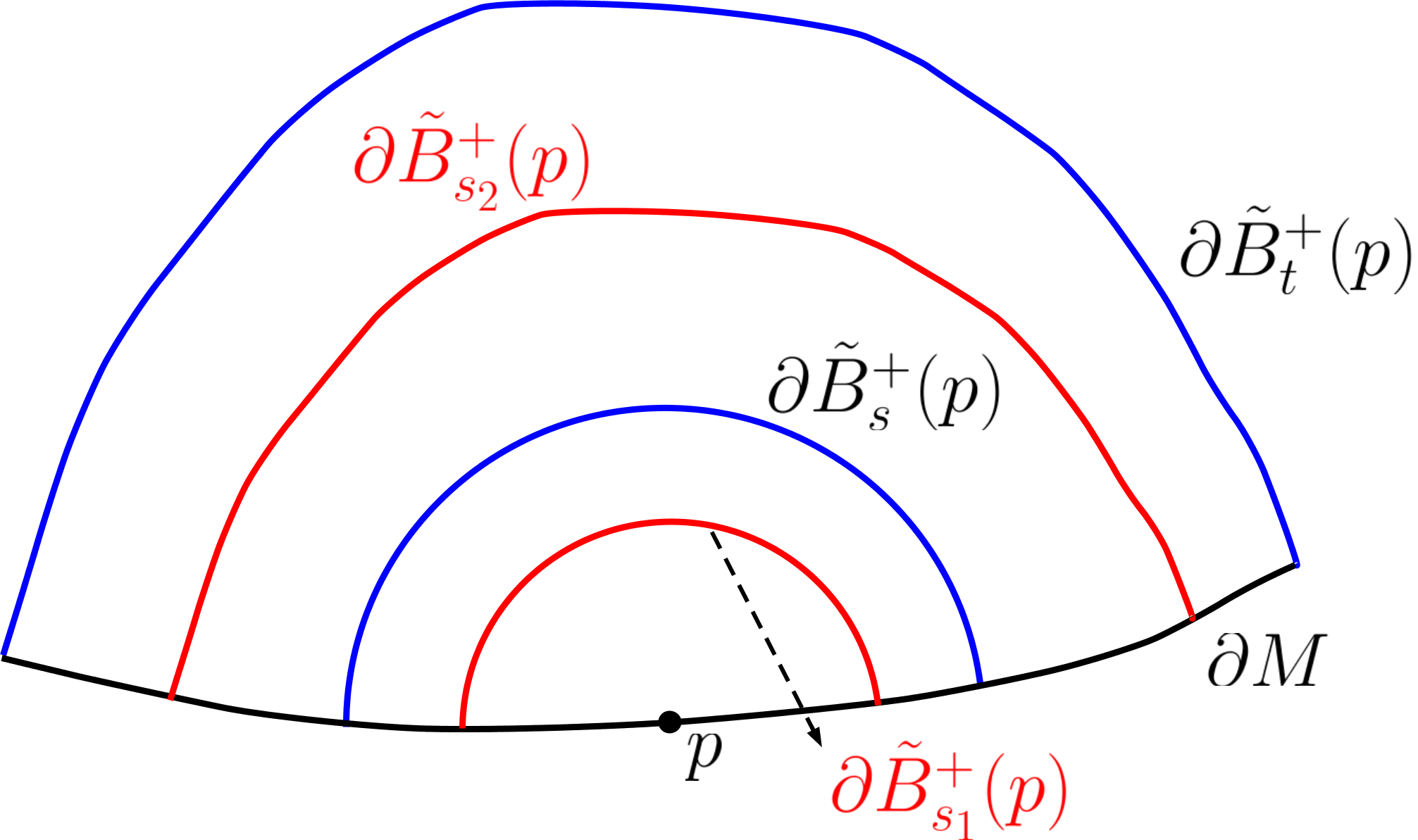}
    \caption{Successive replacements on overlapping Fermi annuli}
    \label{consecutive replacement}
\end{figure}

Fix any $0<s<t<r$. Since $V$ is almost minimizing on small annuli, we can apply Proposition \ref{P:good-replacement-property} to obtain a first replacement $V^*$ of $V$ on $K=\Clos(\A_{s,t}(p))$. By Lemma \ref{L:reg-replacement}, if 
\[ \Sigma_1:=\spt \|V^*\| \cap \A_{s,t}(p),\]
then $(\Sigma_1,\partial \Sigma_1) \subset (M, \partial M)$ is an almost properly embedded hypersurface which is smooth and stable in $\A_{s,t}(p)$. By Sard's theorem we can choose $s_2 \in (s,t)$ such that $\tScal^+_{s_2}(p)$ intersects $\Sigma_1$ transversally (even at $\partial \Sigma_1$). Fix any $s_1 \in (0,s)$. By Proposition \ref{P:good-replacement-property} (iii), one can construct another replacement $V^{**}$ of $V^*$ on $K=\Clos(\A_{s_1,s_2}(p))$, which remains almost minimizing in small annuli and stationary in $M$ with free boundary. By Lemma \ref{L:reg-replacement} again, if 
\[ \Sigma_2:=\spt \|V^{**}\| \cap \A_{s_1,s_2}(p),\]
then $(\Sigma_2,\partial \Sigma_2) \subset (M, \partial M)$ is a smooth, stable, almost properly embedded hypersurface which is smooth and stable in $\A_{s_1,s_2}(p)$.

\vspace{.3cm}

\noindent \textbf{Step 2:} \textit{Gluing $\Sigma_1$ and $\Sigma_2$ smoothly across $\tScal^+_{s_2}(p)$.}

\vspace{.3cm}

Take any $q \in \Sigma_1 \cap \tScal^+_{s_2}(p)$. If $q \in M \setminus \partial M$, the arguments in the proof of \cite[Theorem 4]{Schoen-Simon81} together with the convexity of $\tScal^+_{s_2}(p)$ (Lemma \ref{L:Fermi-convex}) and the maximum principle (Theorem \ref{T:max-principle}) implies that $\Sigma_1$ and $\Sigma_2$ glue together smoothly near $q$. Therefore, we focus on the case $q \in \Sigma_1 \cap \partial M$, which we further divide into two sub-cases.

\begin{figure}[h]
    \centering
    \begin{subfigure}{.50\textwidth}
        \centering
        \includegraphics[width=1\textwidth]{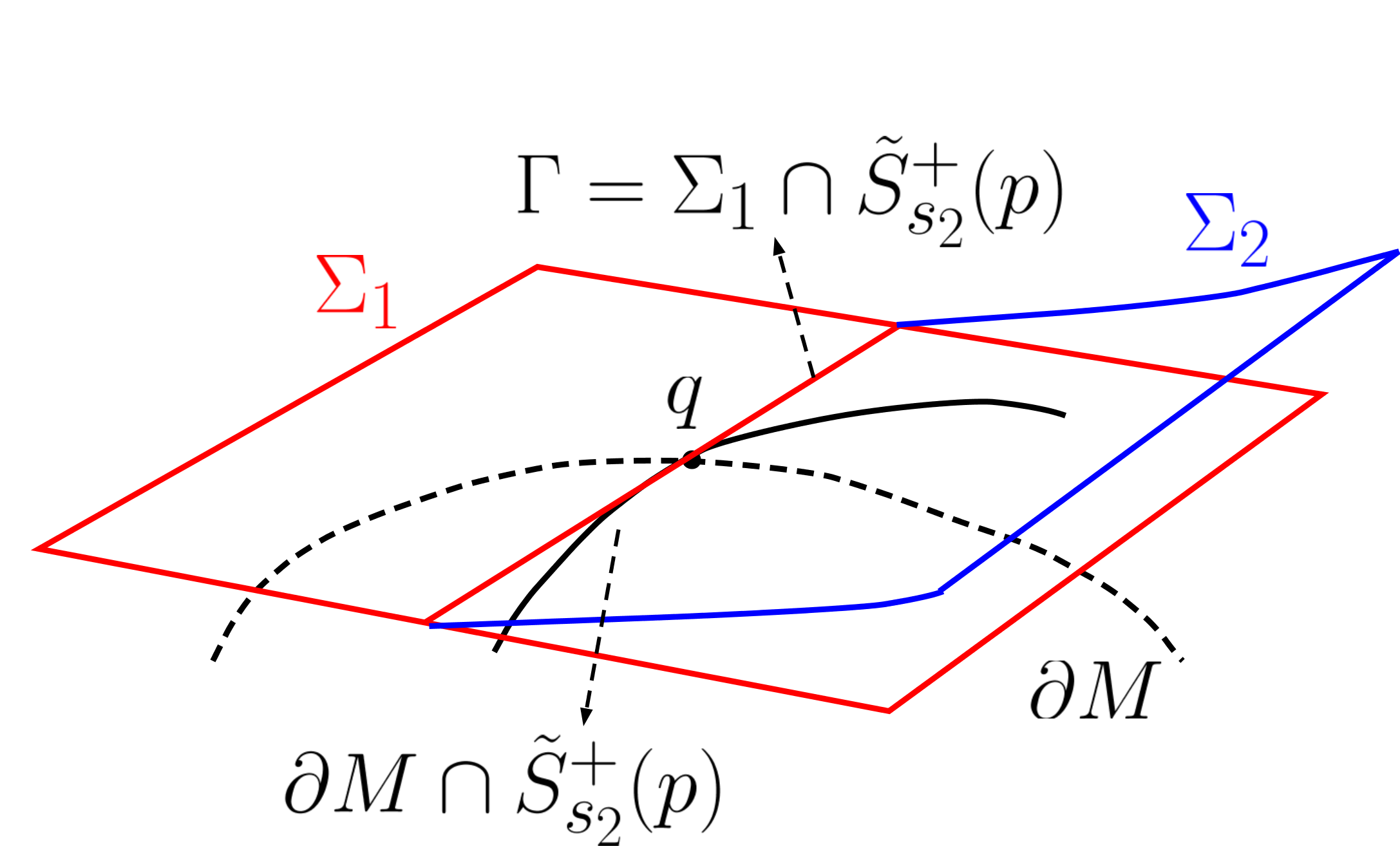}
        \caption{$q$ is a false boundary point.}
        \label{sub case A}
    \end{subfigure}%
    \begin{subfigure}{.50\textwidth}
        \centering
        \includegraphics[width=1\textwidth]{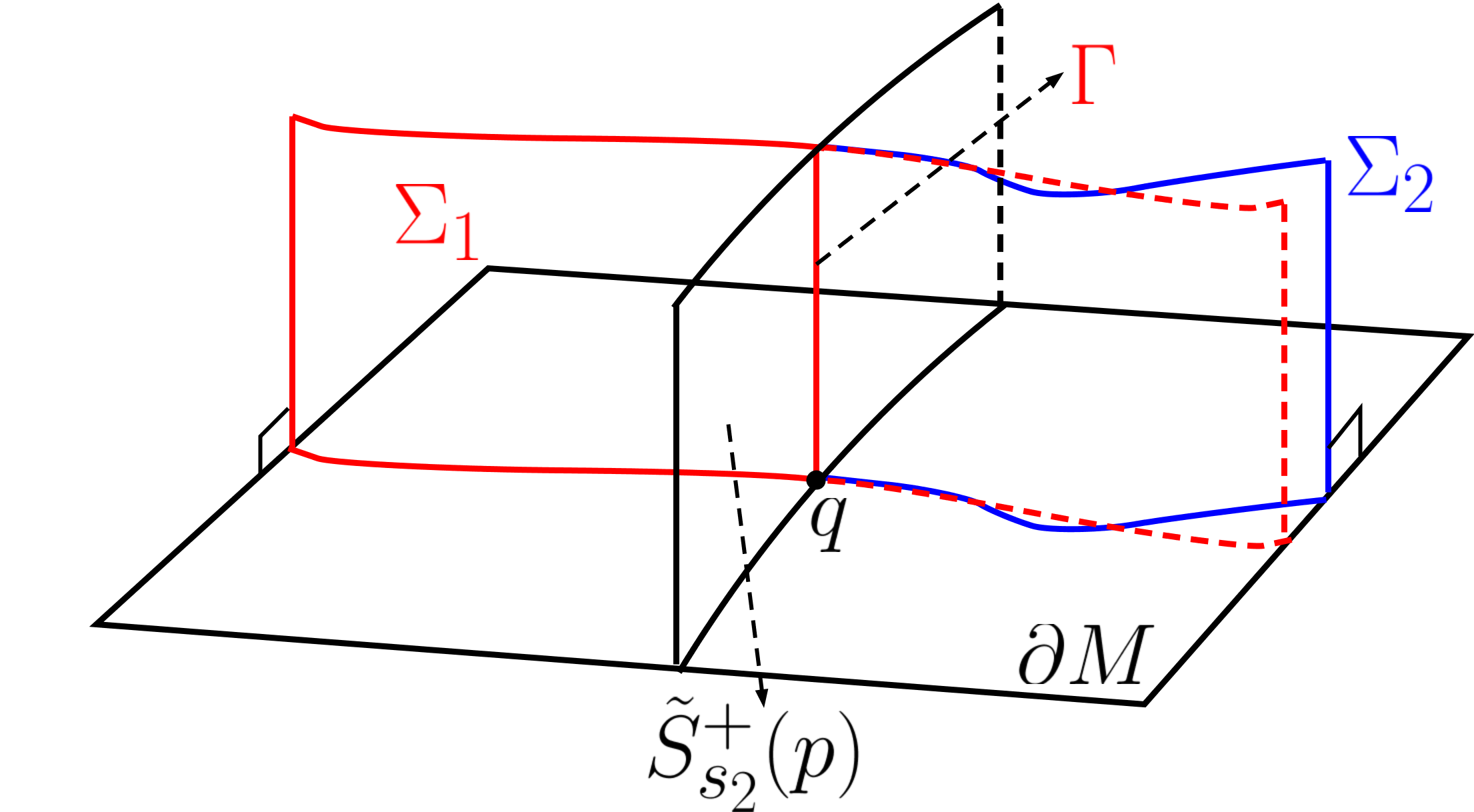}
        \caption{$q$ is a true boundary point}
        \label{sub case B}
    \end{subfigure}
    \caption{Two different cases of gluing $\Sigma_1$ and $\Sigma_2$ near a boundary point.}
\end{figure}

\vspace{.3cm}

\textbf{Sub-case (A)}: $q$ is a false boundary point of $\Sigma_1$ (See Figure \ref{sub case A}).

\vspace{.3cm}

By Remark \ref{R:false-boundary}, $T_q \Sigma_1=T_q(\partial M)$. Moreover, $\Sigma_1$ is a smooth, connected, embedded minimal hypersurface (without boundary) in a neighorhood of $q$ in $\tM$. Define the intersection set
\begin{equation}
\label{E:Gamma-definition}
\Gamma:=\Sigma_1 \cap \tScal^+_{s_2}(p)
\end{equation}
which is a smooth hypersurface (without boundary) in a neighborhood of $q$ in $\Sigma_1$ by the choice of $s_2$. By the maximum principle (\ref{E:max-principle}), we have 
\[ \Clos(\Sigma_2) \cap \tScal^+_{s_2}(p)  \subset \Gamma \]
in a neighborhood of $q$. In fact with a bit more effort we can show that equality holds using transversality and the classification of tangent cones (Proposition \ref{P:tangent-cone}). This implies that we can glue $\Sigma_2$ to $\Sigma_1$ along $\Gamma$ \emph{continuously}. We want to show that the tangent planes of $\Sigma_1$ and $\Sigma_2$ agree as well, i.e. the gluing is $C^1$. Fix any $x \in \Gamma$, we will further assume that $x \in \partial M$ as the case for $x\in M \setminus \partial M$ follows from the interior argument \cite[(7.33)]{Schoen-Simon81}.

\vspace{.3cm}

\textit{Claim 1(A): For any sequence $x_i \to x$ with $x_i \in \Gamma$ and $x \in \Gamma\cap\partial M$, and $r_i\to 0$, we have}
\[ \lim_{i \to \infty} (\bleta_{x_i,r_i})_\sharp V^{**} = \Theta^n(\|V^*\|,x) \; |T_x \Sigma_1|.\]

\textit{Proof of Claim 1(A):}  
By Proposition \ref{P:good-replacement-property} (i), we have
\begin{equation}
\label{E:agreement-with-sigma1}
\spt \|V^{**}\| = \spt \|V^*\| =  \Si_1 \qquad \text{ in } \A_{s_2, t}(p).
\end{equation}
Given $y\in \Gamma$, using the above together with the transversality of $\Sigma_1$ with $\tScal^+_{s_2}(p)$, Proposition \ref{P:tangent-cone} (when $y\in\partial M$) and \cite[Theorem 7.8]{Pitts} (when $y\notin \partial M$), 
we know that 
\begin{equation}
\label{E: convergence to tangent cones for the second replacement}
\VarTan(V^{**}, y) = \{ \Theta^n(\|V^*\|,y ) \; |T_y \Sigma_1| \}.
\end{equation}
Notice that by (\ref{E: convergence to tangent cones for the second replacement}) and the smoothness of $\Sigma_1=\spt \|V^*\| \cap \A_{s,t}(p)$, there exists a positive integer $l$ such that
\begin{equation}
\label{E: density of V* and V**1}
\Theta^n(\|V^{**}\|, y)= \Theta^n(\|V^{*}\|, y)=l 
\end{equation}
for all $y\in\Gamma$ near $x$. 
Given the hypothesis in Claim 1(A), after passing to a subsequence, we have
\begin{equation}
\label{E: limit varifold of V** at xi}
\lim_{i \to \infty} (\bleta_{x_i,r_i})_\sharp V^{**}=C \in\V_n(T_x \tM).
\end{equation}
To prove the claim, we have to show that $C=\Theta^n(\|V^*\|,x) \; |T_x \Sigma_1|$.

Since the blow-up sequence $(\bleta_{x_i,r_i})_\sharp V^{**}$ is taken at a sequence of \emph{varying base points}. There will be two different types of convergence scenario:
\begin{itemize}
\item Type I: $\liminf_{i \to \infty} \dist_{\R^L}(x_i, \Gamma\cap \partial M)/r_i=\infty$,
\item Type II: $\liminf_{i \to \infty} \dist_{\R^L}(x_i, \Gamma\cap \partial M)/r_i < \infty$.
\end{itemize}
In case of Type I, for any $R>0$, $\Gamma\cap B_{r_i R}(x_i) \subset M \setminus \partial M$ for $i$ large enough and thus the blow up sequence does not see the boundary. This case is already covered by the interior regularity theorem (Theorem \ref{T:interior-regularity}) so we are mainly interested in Type II.

For Type II scenario we will carry out the following ``\emph{projection trick}'' to make sure that $x_i \in \partial M$ so that we can apply the monotonicity formula in Theorem \ref{T:monotonicity}. After passing to a subsequence, we can assume that there exists a constant $C>0$ such that
\[ \dist_{\R^L}(x_i, \Gamma\cap \partial M) \leq C r_i \]
for all $i$. Let $x_i' \in \Gamma \cap \partial M$ be the nearest point to $x_i$ so $\dist_{\R^L}(x_i,x_i') \leq C r_i$. Consider the limit
\[ \lim_{i \to \infty} (\bleta_{x_i',r_i})_\sharp V^{**}=C' \in \V_n(T_x \tM). \]
Then $C'$ differs from the limit $C$ in (\ref{E: limit varifold of V** at xi}) by a translation along $T_x \Gamma \subset T_x\Sigma_1$. More precisely, there exists $v \in T_x \Gamma$ such that $(\btau_v)_\sharp C=C'$. Therefore, to prove Claim 1(A) we can assume that $x_i \in \Gamma \cap \partial M$ for all $i$. 

As in Remark \ref{R:tangent-cone}, $C$ is rectifiable and stationary in $T_x M$ with free boundary. We will now show that $C$ is also a \emph{cone}. Note that $x_i \in \Gamma \cap \partial M$. By monotonicity formula (Theorem \ref{T:monotonicity}) and (\ref{E: density of V* and V**1}), we have for all $\tau>0$ sufficiently small (but independent of $i$)
\begin{equation}
\label{E: volume ratio bound for V**}
l=\Theta^n(\|V^{**}\|, x_i)\leq e^{n\La_0 \tau}\frac{\|V^{**}\|(B_\tau(x_i))}{\om_n \tau^n}.  
\end{equation}
On the other hand, using the monotonicity formula with a standard upper semi-continuity argument (c.f. \cite[40.6]{Simon}) we have 
\begin{equation}
\label{E: uniform convergence of volume ratio for V**}
\lim_{\rho\to 0} \frac{\|V^{**}\|(B_\rho(y))}{\om_n \rho^n}=l,
\end{equation}
uniformly in $y$ on compact subsets of $\Gamma\cap \partial M$ near $x$.
Combining (\ref{E: volume ratio bound for V**}) and (\ref{E: uniform convergence of volume ratio for V**}), we deduce that for all $r>0$
\[ l\leq \Theta^n(\|C\|, 0)\leq \frac{\|C\|(B_r(0))}{\om_n r^n}=l, \]
so $C$ is a \emph{cone}, i.e. $(\bmu_{\la})_{\#} C= C$ for all $\la>0$. Therefore, the double $\overline{C}$ of $C$ (relative to the halfspace $T_xM \subset T_x \tM$) is a rectifiable cone which is stationary in $T_x \tM$. As half of $\spt \|\overline{C}\|$ overlaps with a halfspace of $T_x \Sigma_1$ by (\ref{E:agreement-with-sigma1}), using \cite[Lemma 7.9]{Pitts}, $\overline{C}=2 l |T_x\Si_1|$, and hence $C=l  |T_x\Si_1|=\Theta^n(\|V^*\|, x)|T_x\Si_1|$, which established \emph{Claim 1(A)}.

By the same argument as in the proof of claim 1 in Lemma \ref{L:tangent-cone}, we have
\begin{equation}
\label{E:Hausdorff-convergence}
\spt \|(\bleta_{x_i,r_i})_\sharp V^{**}\| \to T_x \Sigma_1 \text{ in the Hausdorff topology} 
\end{equation}
which implies (7.26) in \cite{Schoen-Simon81}. As a direct consequence of Claim 1(A) and Theorem \ref{T:freebdy-cpt}, we have
\begin{equation}
\label{E:no-boundary}
\partial \Sigma_2 = \emptyset \text{ in a neighborhood of $q$.}
\end{equation}

Let $\nu_1$ and $\nu_2$ be the unit normal of $\Sigma_1$ and $\Sigma_2$ respectively which is continuously defined in a neighborhood of $q$. To show that $\Sigma_1$ and $\Sigma_2$ glue together along $\Gamma$ in a $C^1$ manner near $q$. We need to show that:

\vspace{.3cm}

\textit{Claim 2(A): For each $x \in \Gamma\cap \partial M$, we have}
\[ \lim_{\substack{z \to x\\
z \in \Sigma_2}} \nu_2(z) = \nu_1(x).\]
\textit{Moreover, the convergence is uniform in $x$ on compact subsets of $\Gamma$ near $q$.}

\vspace{.3cm}

\begin{figure}[h]
    \centering
        \includegraphics[width=0.6\textwidth]{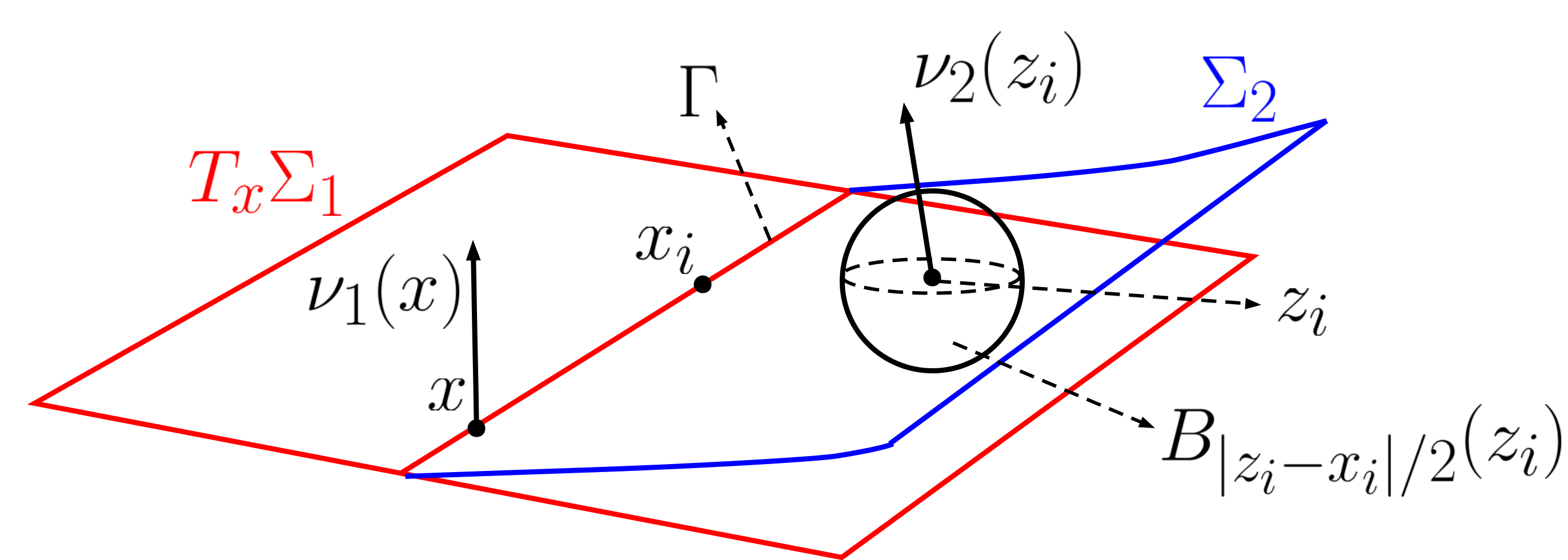}
    \caption{Matching the unit normals along $\Gamma$}
    \label{claim3}
\end{figure}

\textit{Proof of Claim 2(A):} Suppose not, then we can find a sequence $z_i\in \Sigma_2$ converging to some $x\in \Ga$ but no subsequence of $\nu_2(z_i)$ converge to $\nu_1(x)$. Take $x_i \in \Ga$ to be the nearest point projection (in $\R^L$) of $z_i$ to $\Ga$ and $r_i=|z_i-x_i|$. (See Figure \ref{claim3} for an illustration of the following notations.) Note that $x_i \to x \in \Ga \cap \partial M$ and $r_i \to 0$ so we are in the situation of Claim 1(A). By (\ref{E:no-boundary}), for $i$ sufficiently large, $\Si_2 \cap B_{r_i/2}(z_i)$ is a stable minimal hypersurface in $\tM$ without boundary. By the classical curvature estimates in \cite{Schoen-Simon-Yau75} (see also \cite[Theorem 7.5]{Pitts}), a subsequence of the blow-ups $\bleta_{x_i,r_i} (\Sigma_2 \cap B_{r_i/2}(z_i))$ converges smoothly to a smooth stable minimal hypersurface $\Sigma_\infty$ (without boundary) contained in a half-space of $T_x\tM$. On the other hand, (\ref{E:Hausdorff-convergence}) implies that $\bleta_{x_i,r_i} (\Sigma_2 \cap B_{r_i/2}(z_i))$ converges in the Hausdorff topology to a domain in $T_x \Sigma_1$. Therefore, we have $\Sigma_\infty \subset T_x \Sigma_1$. The smooth convergence then implies that $\nu_2(z_i)$ converges to the unit normal $\nu_1(x)$ of $T_x \Sigma_1$. This is a contradiction so \emph{Claim 2(A)} is proved.

\vspace{.3cm}

Now we have $\Sigma_1$ and $\Sigma_2$ glue together as a $C^1$ hypersurface along $\Gamma$. The higher regularity follows from a standard elliptic PDE argument as in \cite{Schoen-Simon81}. This finishes Step 2 for \textbf{Sub-case (A)}.

\vspace{.3cm}

\textbf{Sub-case (B)}: $q$ is a true boundary point of $\Sigma_1$ (See Figure \ref{sub case B}).

\vspace{.3cm}

In this case, $(\Sigma_1,\partial \Sigma_1) \subset (M,\partial M)$ is a smooth properly embedded hypersurface in a neighborhood of $q$. The intersection set $\Gamma$ as defined in (\ref{E:Gamma-definition}) then has boundary lying on $\partial \Sigma_1$. We will show that $\Sigma_1$ and $\Sigma_2$ glue together smoothly at each $x \in \Gamma$. As before, we just have to focus on the case $x \in \Ga \cap \partial M$. By Proposition \ref{P:tangent-cone} and the arguments which lead to (\ref{E: convergence to tangent cones for the second replacement}), we have
\[ \VarTan(V^{**},x)=\{2 \Theta^n(V^*,x) |T_x \Sigma_1|\}. \]
We have the following claim similar to Claim 1(A). Here $\overline{T_x\Si_1}$ is the unique $n$-dimensional subspace in $\R^L$ containing the half-space $T_x\Si_1$.

\vspace{.3cm}

\textit{Claim 1(B): For any sequence $x_i \to x$ with $x_i \in \Gamma$ and $x \in \Gamma\cap\partial M$, and $r_i\to 0$, we have}
\[ \lim_{i \to \infty} (\bleta_{x_i,r_i})_\sharp V^{**} =\left\{ \begin{array}{cl}
2\Theta^n(\|V^*\|,x) |\overline{T_x \Sigma_1}| & \text{ for Type I convergence} \\
2\Theta^n(\|V^*\|,x) |\btau_v(T_x \Sigma_1)| & \text{ for Type II convergence} \end{array} \right. \]
where $v \in T_x \Gamma$ is the limit of the sequence of vectors $(x_i-x)/r_i \in \R^L$.

\textit{Proof of Claim 1(B):} It is analogous to the proof of Claim 1(A). The only difference is that in the scenario of Type II convergence, when we apply the ``\emph{projection trick}'' to make sure that the base points lie on $\partial M$, the new limit $C'$ differs from $C$ by a translation by some $v \in T_x \Ga$, which may no longer preserve the half-space $T_x \Si_1$. This proves the claim.

From the above claim together with Theorem \ref{T:freebdy-cpt}, one sees immediately that $q \in \Clos(\partial \Sigma_2)$. The same argument as in Claim 2(A), using the Compactness Theorem \ref{T:freebdy-cpt}, then implies that $\Sigma_1$ and $\Sigma_2$ glue together as $C^1$ hypersurface with boundary. The higher regularity then follows from the standard PDE (with boundary conditions) by writing the hypersurfaces $\Sigma_1$ and $\Sigma_2$ as graphs over domains of $\overline{T_x \Si_1}$. This finishes Step 2 for \textbf{Sub-case (B)}.

\vspace{.3cm}

\noindent \textbf{Step 3:} \textit{Unique continuation up to the point $p$.}

\vspace{.3cm}

Using Step 2, (\ref{E:max-principle}) and that $\Clos(\Si_2) \cap \tScal^+_{s_2}(p)=\Gamma$, by varying $s_1$ we obtain a smooth almost properly embedded hypersurface $(\Sigma_{s_1},\partial \Sigma_{s_1}) \subset (M,\partial M)$ which is stable in $\A_{s_1,s_2}(p)$. Suppose $s_1'<s_1<s$, then we have $\Sigma_{s_1'}=\Sigma_{s_1}$ on $\A_{s_1,t}(p)$ by the unique continuation of minimal hypersurfaces. Hence, if
\[ \Sigma:=\bigcup_{0<s_1<s} \Sigma_{s_1},\]
then $(\Sigma,\partial \Sigma) \subset (M,\partial M)$ is a smooth almost properly embedded hypersurface which is stable in $\tBcal^+_{s_2}(p) \setminus \{p\}$.

\vspace{0.3cm}

\textit{Claim 3: $\spt \|V\|=\Sigma$ in the punctured ball $\tBcal^+_s(p) \setminus \{p\}$.}

\vspace{0.3cm}

\textit{Proof of Claim 3:} Consider the set
\[ T^V_p=\left\{ y \in \spt\|V\| : \begin{array}{l}
\VarTan(V, y) \text{ consists of an $n$-plane or}\\
\text{ a half $n$-plane transversal to } \tScal^+_{\tir_p(y)}(p)
\end{array} \right\};\]
using a first variation argument as in \cite[Lemma B.2]{Colding-deLellis03} and the convexity of small Fermi half-balls (Lemma \ref{L:Fermi-convex}), it can be easily shown that $T^V_p$ is a dense subset of $\spt \|V\|\cap \tBcal^+_s(p)$. We now show that $T^V_p \cap (\tBcal^+_s(p) \setminus \{p\}) \subset\Sigma$, and hence $\spt\|V\|\cap (\tBcal^+_s(p) \setminus \{p\})\subset\Sigma$. 

Fix $y \in T^V_p \cap (\tBcal^+_s(p) \setminus \{p\})$, and let $\rho=\tir_p(y)$. Let $V^*$ be the replacement of $V$ in $\A_{s,t}(p)$ and $V^{**}$ be the replacement of $V^*$ in $\A_{\rho,s_2}(p)$ where $s_2 \in (s,t)$ is chosen as in Step 1. Since $V^{**}=V^*=V$ inside $\tBcal^+_\rho(p)$ and $y \in T^V_p$, 
\[ y  \in \Clos(\spt\|V\|\cap \tBcal^+_\rho(p))=\Clos(\spt\|V^{**}\|\cap \tBcal^+_\rho(p)).\]
Since $\spt \|V^{**}\|= \Sigma$ in $\A_{\rho,t}(p)$, by (\ref{E:max-principle}) and above, we have $y \in \Sigma$. 

To show the reverse inclusion $\Si\subset \spt\|V\|$, first observe that at every $y\in \Sigma$ where $T_y\Sigma$ is transversal to $\tScal^+_{\tir_p(y)}(p)$, we must have $y\in \spt\|V\|$ by the same argument above.
By the Constancy Theorem \cite[41.1]{Simon}, $\spt\|V\|\cap (\tBcal^+_s(p) \setminus \{p\})$ is identical to $\Si$ in $M \setminus \partial M$. For $y\in \Sigma\cap \partial M$, $T_y\Sigma$ is either a half $n$-plane perpendicular to $T_y(\partial M)$ or $T_y(\partial M)$. In the first case $y$ is a limit point of $\Si \cap int(M)$ and thus $y\in \spt\|V\|$. In the second case $T_y\Sigma=T_y(\partial M)$ is transversal to $\tScal^+_{\tir_p(y)}(p)$ and hence $y\in \spt\|V\|$ as well. This proves our claim.

\vspace{.3cm}

\noindent \textbf{Step 4:} \textit{Removable singularity at $p$.}

\vspace{.3cm}

From Step 3 we know that $\spt\|V\|$ coincides with a smooth hypersurface $\Sigma$ in the punctured Fermi half-ball $\tBcal^+_t(p) \setminus \{p\}$. Finally we want to prove that $p$ is in fact a removable singularity. We first need to establish the following claim: (note that this claim is non-trivial as we do not have a constancy theorem for stationary varifolds with free boundary due to the presence of false boundary point.)

\vspace{.3cm}

\textit{Claim 4: At each false boundary point $q$ of $(\Sigma,\partial \Sigma) \subset (M,\partial M)$, $V$ is stationary in a neighborhood of $q$ as a varifold in $\tM$.}

\vspace{.3cm}

\textit{Proof of Claim 4:} 
As in Step 3 we know that if we choose different $s_1$ in the construction of the second replacement $V^{**}$, then the replacements actually agree on the overlapping annulus. Hence, we can uniquely continue $V^{**}$ up to the point $p$ as in Step 3. Denote the unique continuation of $V^{**}$ by $\tilde{V}$. By Theorem \ref{T:freebdy-cpt}, it is clear that $\Theta^n(\|\tilde{V}\|,\cdot)$ is constant on each connected component of $\interior(\Sigma)$. By similar argument as in (\ref{E: density of V* and V**1}) we have
\[ \Theta^n(\|V\|,x)=\Theta^n(\|\tilde{V}\|,x) \quad \text{ for all $x \in T^V_p$}.\]
Let $q$ be any false boundary point $\Sigma$, we must have $q \in T^V_p$ since $T_q \Sigma=T_q (\partial M)$. By smoothness of $\Sigma$, there exists a connected neighborhood of $q$ in $\Sigma$ such that all the point $x$ in this neighborhood belongs to $T_p^V$. Therefore, the equation above implies that $\Theta^n(\|V\|,\cdot)$ must also be constant on this neighborhood, hence $V$ is stationary in this neighborhood as a varifold in $\tM$.
This proves \emph{Claim 4}.

Therefore, we can apply the classical constancy theorem outside the set of true boundary points of $(\Sigma,\partial \Sigma) \subset (M,\partial M)$ to obtain
\[ V \lc (\tBcal^+_t(p) \setminus \{p\}) = \sum_{i=1}^N n_i |\Sigma_i|\]
where each $(\Sigma_i,\partial \Sigma_i) \subset (M,\partial M)$ is a smooth, connected, almost properly embedded hypersurface which is stable in $\tBcal^+_{s_2}(p) \setminus \{p\}$. We now prove the following claim, which is a strengthened version of Proposition \ref{P:tangent-cone} at $p$:

\vspace{.3cm}

\textit{Claim 5: One and only one of the following cases could happen:}
\[ \VarTan(V, p)=\{\Theta^n(\|V\|, p) |T_p(\partial M)|\}; \text{ or}\]
\[ \VarTan(V,p) \subset \left\{ 2\Theta^n(\|V\|,p) |S \cap T_pM| \left| \begin{array}{c} S \in \Gr(L,n), S \subset T_p \tM, \\
S \perp T_p(\partial M) \end{array} \right.\right\}. \]

\vspace{.3cm}

\textit{Proof of Claim 5:} Suppose not, then there exist sequences $r_i \to 0$, $r_i'\to 0$ such that
\begin{equation}
\label{E:first-case} \lim_{i \to \infty} (\bleta_{p,r_i})_\sharp V = \Theta^n(\|V\|, p) |T_p(\partial M)|, 
\end{equation}
\[ \lim_{i \to \infty} (\bleta_{p,r_i'})_\sharp V = 2\Theta^n(\|V\|,p) |S \cap T_pM| \]
for some $n$-dimensional subspace $S \subset T_p \tM$ which is perpendicular to $T_p(\partial M)$. By passing to a subsequence if necessary, we can assume that 
\begin{equation}
\label{E:ri}
  2 r'_{i+1} < r_i < \frac{1}{2}  r_i'
 \end{equation}
for all $i$. Recall that $(\Sigma,\partial \Sigma) \subset (M,\partial M)$ is stable in $\tBcal^+_{s_2}(p) \setminus \{p\}$. By Theorem \ref{T:freebdy-cpt} we know that for all $i$ sufficiently large
\begin{itemize}
\item $\partial \Si = \emptyset$ in $A_{r_i, 2r_i}(p)\cap \tM$, and
\item $\partial \Si \neq \emptyset$ in $A_{r^\pr_i, 2r^\pr_i}(p)\cap \tM$.
\end{itemize}
Note that $A_{r_i, 2r_i}(p)$ and $A_{r^\pr_i, 2r^\pr_i}(p)$ are mutually disjoint annuli by (\ref{E:ri}). Therefore if we let 
\[ R_i:=\sup\{R>0 \; : \; \partial \Si=\emptyset \text{ in }  A_{r_i, R}(p)\cap \tM\}.\]
Clearly, $2r_i \leq R_i \leq 2r_{i-1}'$. Furthermore, we have $\lim_{i\rightarrow \infty}R_i/r_i=\infty$. Otherwise, the sequence $(\bleta_{p,r_i})_\sharp \Sigma$ will converge on a compact subset of $\mathbb{R}^L \setminus\{0\}$ to a stable minimal hypersurface with non-empty boundary, contradicting (\ref{E:first-case}). Consider the blow-up limit (after passing to a subsequence)
\[ C:=\lim_{i \to \infty} (\bleta_{p,R_i})_\sharp V,\]
by Theorem \ref{T:freebdy-cpt} and the definition of $R_i$ we have
\begin{itemize}
\item $\partial (\spt \|C\|) =\emptyset$ in $A_{1/2,1}(0)$, and
\item $\partial (\spt \|C\|) \neq \emptyset$ in $A_{1,2}(0)$.
\end{itemize}
This contradicts Proposition \ref{P:tangent-cone} as $C \in \VarTan(V,p)$ must have support either as $T_p (\partial M)$ (which has no boundary) or a half $n$-plane in $T_p \tM$ perpendicular to $T_p(\partial M)$ (which has non-empty boundary). This completes the proof of \emph{Claim 5}.

\vspace{0.3cm}

If the first case of Claim 5 occurs, then by Theorem \ref{T:freebdy-cpt} again, $\partial \Si = \emptyset$ in a punctured neighborhood of $p$, and hence $V$ is stationary in that neighborhood as a varifold \emph{in $\tM$}. By the uniform volume ratio bound at $p$ (Lemma \ref{L:uniform-density-bdd}) and a standard extension argument (c.f. the proof in \cite[Theorem 4.1]{Harvey-Lawson}), we can extend the stationarity across $p$, and show that $V$ is stationary in a neighborhood of $p$ in $\tM$. Then the removable singularity argument reduces to that in \cite[Theorem 7.12]{Pitts}. In fact, our situation is easier as $V$ has a unique tangent cone (an integer multiple of $T_p(\partial M)$) at $p$.

Now suppose the second case of Claim 5 occurs, then we have $p \in \Clos(\partial \Sigma)$. By Proposition \ref{P:tangent-cone} we know that
\[2 \Theta^n (\|V\|,p)=m \]
for some $m \in \mathbb{N}$. Since $\Sigma$ is stable in a punctured ball of $p$, by the proof of Theorem \ref{T:freebdy-cpt} and that $p \in \Clos(\partial \Sigma)$, for any sequence $r_i \to 0$, 
\[ (\bleta_{p,r_i})_\sharp \Sigma \to m|S \cap T_p M| \]
locally smoothly in $\mathbb{R}^L \setminus \{0\}$ for some $n$-plane $S \subset T_p \tM$ which is perpendicular to $T_p(\partial M)$. However, $S$ may depend on the sequence $r_i$. This implies that $\Sigma$ has no false boundary point in a neighborhood of $p$.
Moreover, there exists $\si_0>0$ small enough, such that for any $0<\si\leq \si_0$, $V$ has an $m$-sheeted, pairwise disjoint, graphical decomposition in $A_{\si/2, \si}(p)$:
\begin{equation}
\label{E: annuli-decomposition}
V\lc A_{\si/2, \si}(p)=\sum_{i=1}^{l(\si)}m_i(\si) |\Si_i(\si)|.
\end{equation}
Here $\Si_i(\si)$ is a graph over $A_{\si/2, \si}(p)\cap S$ for some $m$-plane $S\subset T_p\tM$ with $S\perp T_p(\partial M)$; $m_i(\si), l(\si)$ are positive integers such that
\begin{equation}
\label{E: sum of density}
\sum_{i=1}^{l(\si)}m_i(\si)=m.
\end{equation}
Since (\ref{E: annuli-decomposition}) is true for all $\si$, by continuity of $\Si$, $m_i(\si), l(\si)$ are independent of $\si$, so we can continue each $\Si_i(\si_0)$ to $(B_{\si_0}(p)\setminus\{p\})\cap M$ and get $\Si_i$. Since the properly embedded minimal hypersurfaces $(\Si_i,\partial \Si_i) \subset (M,\partial M)$ are pairwise disjoint and stable in $(B_{\si_0}(p)\setminus\{p\})\cap M$ with $\partial \Sigma_i \neq \emptyset$, by the standard extension argument as before, each $\Si_i$ can be extended as a stationary varifold in $B_{\si_0}(p)\cap M$ with free boundary. Since we are in the second case of Claim 5, every $C_i\in \VarTan(\Si_i, p)$ is supported on a half $n$-plane in $T_p \tM$ perpendicular to $T_p(\partial M)$. To see that $C_i$ has multiplicity one, first notice that
\begin{equation}
\label{E:density=1}
2  \Theta^n(\|C_i\|,p) \geq 1, 
\end{equation}
since each $\Si_i$ is stable and thus its rescalings converge with multiplicity to a smooth stable hypersurface by Theorem \ref{T:freebdy-cpt}. If equality does not hold for some $i$ in (\ref{E:density=1}), this will contradict (\ref{E: sum of density}) as 
\[ V\lc B_{\si_0}(p)= \sum_{i=1}^l m_i |\Si_i|.\]
Therefore, each $\Si_i$ is stationary in $B_{\si_0}(p) \cap M$ with free boundary and $\Theta^n(\||\Sigma_i|\|,p)=1/2$; by the Allard type regularity theorem for stationary rectifiable varifolds with free boundary \cite[Theorem 4.13]{Gruter-Jost86}, $\Si_i$ extends as a smooth hypersurface across $p$. Finally, by the maximum principle for free boundary minimal hypersurfaces, we have $l=1$ and this finishes the proof of Theorem \ref{T:main-regularity}.
\end{proof}

\appendix


\section{Fermi coordinates}
\label{A:Fermi}

In this appendix, we give a rather self-contained exposition of \emph{Fermi coordinates} near a boundary point of a Riemannian manifold with boundary that are used in this paper.

\begin{definition}[Fermi coordinates]
\label{D:Fermi}
Let $p \in \partial M$. Suppose $(x_1,x_2,\cdots,x_n)$ is the geodesic normal coordinates of $\partial M$ centered at $p$, which is defined in a small neighborhood of $p$ in $\partial M$. Let $t=\dist_M(\cdot,\partial M)$, which is well-defined and smooth in a small relatively open neighborhood of $p$ in $M$. Then $(x_1,x_2,\cdots,x_n,t)$ is said to be the Fermi coordinate system of $(M,\partial M)$ centered at $p \in \partial M$. Moreover we define the \emph{Fermi distance function from $p$} on a relatively open neighborhood of $p$ in $M$ by
\begin{equation}
\label{E:rtilde}
\tir=\tir_p(q):=|(x,t)|=\sqrt{x_1^2+\cdots+x_n^2+t^2}.
\end{equation}
The \emph{Fermi exponential map at $p$}, $\tE_p$, which gives the Fermi coordinate system is defined on a half-ball of $\R^{n+1}_+ \cong T_pM$ centered at the origin. In particular, $\tE_p(x_1,\cdots,x_n,t)$ is the point on $M$ whose Fermi coordinates as defined above is given by $(x_1,\cdots,x_n,t)$.
\end{definition}

\begin{lemma}[Metric and connection in Fermi coordinates]
\label{L:Fermi}
Let $g_{ab}$ and $\Gamma_{ab}^c$ be the components of the metric and connection in Fermi coordinates $(x,t)=(x_1,\cdots,x_n,t)$ centered at some $p \in \partial M$. Then, we have for $i=1,\cdots,n$,
\begin{equation}
\label{E:g-Gamma} 
g_{tt} \equiv 1, \; g_{it} \equiv 0, \; \Gamma_{it}^t=\Gamma_{tt}^i=\Gamma_{tt}^t \equiv 0.
\end{equation}
Moreover, there exist constants $r_1,C>0$ depending only on the isometric embedding $M \hookrightarrow \mathbb{R}^L$ such that for any $\tir \in [0,r_1)$,
\begin{displaymath}
\begin{split}
 |g_{ij}(x,t)-\de_{ij}| & \leq C \tir, \\
|\Gamma_{ij}^k(x,t)| \leq C \tir,  \quad | \Gamma_{it}^j (x,t)| &+|\Gamma_{ij}^t(x,t)| \leq C.
\end{split}
\end{displaymath}
\end{lemma}

\begin{proof}
Throughout this proof, we will denote $r_1$ and $C$ to be any constants depending only on the embedding $M \hookrightarrow \mathbb{R}^L$. Since $t=\dist_M(\cdot,\partial M)$, we have $g_{tt} \equiv 1$ and $g_{it} \equiv 0$, from which it follows that $\Gamma_{it}^t=\Gamma_{tt}^i=\Gamma_{tt}^t \equiv 0$. Moreover, since $\frac{\partial}{\partial x_i}$ are Jacobi fields along the geodesics $\gamma_x(t)=(x,t)$, it satisfies the Jacobi equation $\nabla_{\frac{\partial}{\partial t}} \nabla_{\frac{\partial}{\partial t}} \frac{\partial}{\partial x_i} = -R\left(\frac{\partial}{\partial t},\frac{\partial}{\partial x_i}\right) \frac{\partial}{\partial t}$, where $R$ is the Riemann curvature tensor of $M$. By standard estimates for geodesic normal coordinates, for $|x|<r_1$ we have
\begin{equation}
\label{E:g-ij-a}
|g_{ij}(x,0)-\de_{ij}| \leq C |x| 
\end{equation}
as $\partial M$ has bounded sectional curvature. On the other hand, $\partial_t g_{ij}(x,0)=h_{ij}(x)$, where $h_{ij}$ is the components of the second fundamental form of $\partial M$ under the coordinates $x=(x_1,\cdots,x_n)$. Using the Jacobi equation, one can derive as in \cite[Lemma 2.2]{Marques05} the Riccati-type equation
\[ \partial_t^2 g_{ij}=-2 R_{titj} +2g^{rs} (\partial_t g_{ir}) (\partial_t g_{sj}).\]
Since the curvature is bounded, we have for all $t < r_1$,
\begin{equation}
\label{E:g-ij-b}
|g_{ij}(x,t) - g_{ij}(x,0)| \leq C t.
\end{equation}
Combining (\ref{E:g-ij-a}) and (\ref{E:g-ij-b}) gives the estimate for $g_{ij}$. A similar argument gives the estimates for the Christoffel symbols using the bounds on the derivatives of the curvatures.
\end{proof}

\begin{lemma}[Gradient and Hessian estimates for $\tir$]
\label{L:r-tilde}
The Fermi distance function $\tir$ is smooth (in where it is defined) and there exist constants $r_1,C>0$ depending only on the isometric embedding $M \hookrightarrow \mathbb{R}^L$ such that for any $\tir \in [0,r_1)$,
\begin{itemize}
\item[(i)] $\|\nabla \tir - \frac{\partial}{\partial \tir}\|_g \leq C \tir$. 
\item[(ii)] $\| \Hess \tir - \frac{1}{\tir} g_{\tir} \|_g \leq C$,
\end{itemize}
where $g_{\tir}$ is the round metric on a Euclidean sphere of radius $\tir$.
\end{lemma}

\begin{proof}
Since 
\begin{equation}
\label{E:grad-rtilde}
\nabla \tir = \sum_{i,j=1}^n g^{ij}(x,t) \frac{x_i}{\tir} \frac{\partial}{\partial x_j} +\frac{t}{\tir} \frac{\partial}{\partial t},
\end{equation}
(i) follows directly from Lemma \ref{L:Fermi}. For (ii), a direct calculation gives
\[ \Hess \tir \left(\frac{\partial}{\partial t}, \frac{\partial}{\partial t}\right) = \frac{\tir^2-t^2}{\tir^3}, \qquad \left| \Hess \tir \left(\frac{\partial}{\partial x_i}, \frac{\partial}{\partial t}\right) + \frac{x_i t}{\tir^3} \right| \leq C, \]
\[ \left| \Hess \tir \left(\frac{\partial}{\partial x_i}, \frac{\partial}{\partial x_j}\right)- \frac{\delta_{ij} \tir^2-x_i x_j}{\tir^3}\right| \leq C, \]
which easily implies (ii).
\end{proof}

\begin{definition}
\label{D:Fermi-balls}
For each $p \in \partial M$, we define the \emph{Fermi half-ball and half-sphere of radius $r$ centered at $p$} respectively by
\[ \tBcal^+_r(p):= \{ q \in M \; : \; \tir_p(q) < r \}, \qquad \tScal^+_r(p):=\{ q \in M \; : \; \tir_p(q) = r \}.\]
\end{definition}

We summarize their geometric properties which are most relevant to our application in the lemma below. 


\begin{lemma}
\label{L:Fermi-convex}
There exists a small constant $r_{\textrm{Fermi}}>0$, depending only on the isometric embedding $M \subset \mathbb{R}^L$, such that for all $0<r<r_{\textrm{Fermi}}$
\begin{itemize}
\item $\tScal^+_r(p)$ is a smooth hypersurface meeting $\partial M$ orthogonally,
\item $\tBcal^+_r(p)$ is a relatively convex domain in $M$,
\item $B_{r/2}(p) \cap M \subset \tBcal^+_r(p) \subset B_{2r}(p) \cap M$.
\end{itemize}
\end{lemma}

\begin{proof}
It follows directly from Lemma \ref{L:Fermi} and \ref{L:r-tilde}.
\end{proof}


\section{Proof of Theorem \ref{T:def-equiv}}
\label{A}

From the definitions, it is easy to see that $(a) \Longrightarrow (b) \Longrightarrow (c)$ using
\[ \sA_k(U; \ep,\de';\nu) \subset \sA_k(U; \ep',\de;\nu) \quad \text{for any $\ep \leq \ep', \de \leq \de', \nu=\F,\M,\mF$,} \]
\[ \sA_k(U; \ep, \de; \F) \subset \sA_k(U; \ep, \de; \mF)  \subset \sA_k(U; \ep, \de/2; \M).\]
Therefore, it remains to show $(c)\Longrightarrow (d)$. Fix any relatively open subset $W \subset \subset U$. If $\Clos(W) \cap \partial M=\emptyset$, then the theorem reduces to the case of \cite[Theorem 3.9]{Pitts}. Therefore, we would assume $\Clos(W) \cap \partial M \neq \emptyset$. We will be following the proof of \cite[Theorem 3.9]{Pitts} rather closely except for some new modifications to adapt to our case with nonempty boundary. In particular, we have to generalize \cite[Lemma 3.5, Lemma 3.7, Lemma 3.8]{Pitts} to the case with boundary.  As many of the key ingredients are rather delicate near the boundary, we will give a relatively self-contained proof.

Assume the contrary that $(c)$ holds but $(d)$ does not, then there exist $W \subset \subset U$ and $\zeta>0$ such that for any $\de>0$ and $\tau\in \Z_k(M, \partial M)$ with
\[\mF(|T|, V)<\zeta,\]
where $T \in \tau$ is the canonical representative, we have $\tau\notin \sA_k(W; \zeta, \de; \F)$. In other words, there exists a sequence $\tau=\tau_0, \tau_1, \tau_2, \cdots, \tau_m\in\Z_k(M, \partial M)$ such that for each $i=1, \cdots, m$,
\begin{itemize}
\item $\spt(\tau_i-\tau)\subset W$,
\item $\F(\tau_i - \tau_{i+1})\leq \de$,
\item $\M(\tau_i)\leq \M(\tau)+\de$, 
\end{itemize}
but $\M(\tau_m)< \M(\tau)-\zeta$. 
We will carry out an interpolation between each pair $\tau_i, \tau_{i+1}$ to get a new sequence $\tau=\ti{\tau}_0, \ti{\tau}_1, \cdots, \ti{\tau}_l=\tau_m\in\Z_k(M, \partial M)$ satisfying all the above properties of $\tau_i$ and in addition, for each $j=1,\cdots,l$,
\begin{itemize}
\item $\spt(\ti{\tau}_j-\tau)\subset U$,
\item $\M(\ti{\tau}_j- \ti{\tau}_{j+1})\leq \de$,
\end{itemize}
which would give a contradiction to $(c)$.
The required interpolation sequence $\ti{\tau}_j$ will be constructed using the Interpolation Lemma below (Lemma \ref{L:interpolation}). Assume the lemma holds for the moment and fix an
\[ 0< \ep < \min \left\{ \ep \left(\M(\tau)+ \frac{\de}{2}, \frac{\de}{2}, W,\tau\right), \frac{\de}{2} \right\}.\]
Since $\tau \notin \sA_k(W;\zeta, \ep; \F)$, we can find a sequence $\tau_i$ as above with $\de$ replaced by $\ep$.  Then we can apply Lemma \ref{L:interpolation} to each pair $\tau_i, \tau_{i+1}$ to get an interpolation sequence between them as $\F(\tau_i- \tau_{i+1})\leq \ep$.  Combining these interpolation sequences we get the desired sequence $\ti{\tau}_j$. This then completes the proof of Theorem \ref{T:def-equiv}.

\begin{lemma}[Interpolation Lemma]
\label{L:interpolation}
Given $L>0$, $\de>0$, a relatively open subset $W\subset \subset U$ with $W \cap \partial M\neq \emptyset$, and $\tau\in\Z_k(M, \partial M)$, there exists 
\[ \ep=\ep(L, \de, W, \tau)>0,\]
such that for any $\si_1, \si_2\in \Z_k(M, \partial M)$ satisfying
\begin{itemize}
\item $\spt(\si_i-\tau)\subset W$, $i=1, 2$,
\item $\M(\si_i)\leq L$, $i=1, 2$,
\item $\F(\si_1-\si_2)\leq \ep$,
\end{itemize}
there exist a sequence $\si_1=\tau_0, \tau_1, \cdots, \tau_m=\si_2\in \Z_k(M, \partial M)$ such that for each $j=0, \cdots, m-1$,
\begin{itemize}
\item[(i)] $\spt(\tau_j-\tau)\subset U$,
\item[(ii)] $\M(\tau_j)\leq L+\de$,
\item[(iii)] $\M(\tau_j-\tau_{j+1})\leq \de$.
\end{itemize}
\end{lemma}

\begin{remark}
The Interpolation Lemma says that if two equivalence classes of relative cycles are close enough in the $\F$-norm, we can find an interpolation sequence between them which is continuous in the $\M$-norm, and at the same time keeping their masses not increased by too much and their supports not changed by too much.
\end{remark}

The rest of this appendix is devoted to the proof of Lemma \ref{L:interpolation}, which is a rather delicate adaptation of the proof of \cite[Lemma 3.7, Lemma 3.8]{Pitts}. The major difference here is that we are using equivalence classes of relative cycles (see Definition \ref{D:relative-cycles}). Moreover, we also need to deal with the extra boundary terms. 

We observe that Lemma \ref{L:interpolation} follows by a straightforward covering argument as in \cite[page 124]{Pitts} from the lemma below, which is Lemma \ref{L:interpolation} with one of the $\si_i$ fixed.

\begin{lemma}[Pre-interpolation Lemma]
\label{L:pre-interpolation-lemma}
Given $L,\de, W$ and $\tau$ as in Lemma \ref{L:interpolation}, for each $\si_0\in\Z_k(M, \partial M)$ satisfying
\[ \spt(\si_0-\tau) \subset W \quad \text{ and } \quad \M (\si_0)\leq L,\]
there exists 
\[ \ep=\ep(L, \de, W, \tau, \si_0)>0,\]
such that if $\si\in \Z_k(M, \partial M)$ satisfies
\begin{itemize}
\item $\spt(\si-\tau) \subset W,$ 
\item $\M(\si)\leq L$,
\item $\F(\si-\si_0)\leq \ep,$
\end{itemize}
then there is a sequence as in Lemma \ref{L:interpolation} with $\si_1, \si_2$ replaced by $\si, \si_0$.
\end{lemma}

\begin{proof}
Suppose the contrary that the lemma is false, then there exist a sequence $\si_j\in\Z_k(M, \partial M)$, $j=1,\cdots$, satisfying
\begin{itemize}
\item $\spt(\si_j-\tau)\subset W,$
\item $\M(\si_j)\leq L,$
\item $\lim_{j\rightarrow\infty} \F(\si_j - \si_0)=0$,
\end{itemize}
but none of the $\si_j$ can be connected to $\si_0$ by a sequence as in Lemma \ref{L:interpolation} satisfying (i)-(iii) with $\si_1, \si_2$ replaced by $\si_j, \si_0$. To prove our lemma, we will explicitly construct such a sequence connecting $\si_j$ to $\si_0$ for $j$ sufficiently large, hence obtain a contradiction.

Let $S_j \in \si_j$ be the canonical representative. Since the sequence $S_j$ have uniformly bounded mass as $\M(S_j)=\M(\si_j) \leq L$, after passing to a subsequence, the associated varifolds $|S_j|$ converges as varifolds to some $V \in \V_k(M)$, i.e.
\[ V=\lim_{j\rightarrow\infty}|S_j|. \]
By the lower semi-continuity of $\M$ (Lemma \ref{L:lower-semicts}), we have
\begin{equation}
\label{E:S_0-V}
\|S_0\| (A) \leq \|V\| (A) \quad \text{ for all Borel $A \subset \R^L$}. 
\end{equation}
Let $\al=\de/5$. We will divide our construction into two cases, depending on whether the varifold $V$ has point mass in $W$ larger than $\al$:
\begin{itemize}
\item {\bf Case 1:} $\|V\|(\{q\})\leq \al$ for all $q\in W$;
\item {\bf Case 2:} $\{q\in W: \|V\|(\{q\})> \al\}\neq \emptyset$.
\end{itemize}

We now handle {\bf Case 1} first. Fix a finite collection of geodesic balls 
$\{\tBcal_{r_i}(p_i)\}_{i=1}^m$ in $\tM$ with centers $p_i \in \spt \|V\| \cap W$ such that $\Clos(\tBcal_{r_i}(p_i)) \cap M \subset U$ are mutually disjoint and satisfy the following: 
\begin{itemize}
\item $\|V\| (\tScal_{r_i}(p_i) \cap M)=0$;
\item $\|V\| (\Clos(\tBcal_{r_i}(p_i)))  <2\al$;
\item $\|V\| (W \setminus \cup_{i=1}^m \Clos(\tBcal_{r_i}(p_i)) ) <2\al$.
\end{itemize}
This is possible because 
all point mass of $\|V\|$ is at most $\alpha$ by the assumption of \textbf{Case 1}. 
Note that by (\ref{E:S_0-V}), $S_0$ also satisfies all the inequalities above.
Since $\si_j$ converges to $\si_0$ in the $\F$-topology, we can apply the $\F$-isoperimetric lemma (Lemma \ref{L:isoperimetric-F}) so that for each $j$ large enough, there exists an integal $(k+1)$-currents $Q_j$ such that
\begin{itemize}
\item $\spt(Q_j) \subset M$;
\item $\spt(S_j-S_0-\partial Q_j)\subset \partial M$;
\item $\lim_{j\rightarrow\infty}\M(Q_j)=0$.
\end{itemize}
Moreover, for $j$ large enough, using the slicing theorem \cite[\S 28]{Simon} and the lower-semicontinuity of measures under weak convergence, we can find sequences of radii $r^j_i\searrow r_i$, $i=1, \cdots, m$, such that: 
\begin{itemize}
\item the geodesic balls $\{\tBcal_{r_i^j}(p_i)\}_{i=1}^m$ have pairwise disjoint closures;
\item $\cup_{i=1}^m \tBcal_{r_i^j}(p_i) \cap M \subset U$;
\item The slice $\lan Q_j, d^{\tM}_{p_i}, r^j_i\ran$ of $Q_j$ by $d^{\tM}_{p_i}:=\dist_{\ti{M}}(p_i, \cdot)$ at $r^j_i$ is an integer rectifiable current \cite[28.4]{Simon} with (see \cite[28.5(1)]{Simon})
\[ \sum_{i=1}^m \M (\lan Q_j, d^{\tM}_{p_i}, r^j_i\ran) <\al; \]
\item $\|S_j\| ( \cup_{i=1}^m \tScal_{r^j_i}(p_i)) = \|Q_j\| ( \cup_{i=1}^m \tScal_{r^j_i}(p_i)) =0$;
\item $\|S_j\| (\Clos(\tBcal_{r_i^j}(p_i)) \leq 2\al$ and $\|S_j\| (W \setminus \cup_{i=1}^m \Clos(\tBcal_{r_i^j}(p_i))) \leq 2\al$;
\item $\|S_0\| (\Clos(\tBcal_{r_i^j}(p_i)) \leq 2\al$ and $\|S_0\| (W \setminus \cup_{i=1}^m \Clos(\tBcal_{r_i^j}(p_i))) \leq 2\al$;
\item $(\|S_0\|-\|S_j\|) (\Clos(\tBcal_{r_i^j}(p_i)))\leq \frac{\al}{m}$;
\item $(\|S_0\|-\|S_j\|) (W \setminus \cup_{i=1}^m \Clos(\tBcal_{r_i^j}(p_i))) \leq \al$.
\end{itemize}
Now we construct a sequence of currents 
\[ S_j=R^j_0, R^j_1,\cdots,R^j_m, R^j_{m+1}=S_0 \in Z_k(M,\partial M) \] 
where for $i=1,2,\cdots,m$,
\begin{equation}
\label{currents connecting S_j to S_0}
R^j_i:=S_j-\sum_{s=1}^i \left( \partial [Q_j\lc \Clos(\tBcal_{r_s^j}(p_s))]-(\partial Q_j) \lc [ \Clos(\tBcal_{r_s^j}(p_s)) \cap \partial M] \right).
\end{equation}
Note that by \cite[28.5(2)]{Simon} and our choice of $Q_j$,
\begin{equation*}
\begin{split}
& \ \ \ \partial [Q_j\lc \Clos(\tBcal_{r_s^j}(p_s))]-(\partial Q_j) \lc [ \Clos(\tBcal_{r_s^j}(p_s)) \cap \partial M] \\
& = \lan Q_j, d^{\tM}_{p_s}, r^j_s\ran + (\partial Q_j) \lc \Clos(\tBcal_{r^j_s}(p_s)) -(\partial Q_j)\lc [\Clos(\tBcal_{r^j_s}(p_s)) \cap \partial M]\\
& = \lan Q_j, d^{\tM}_{p_s}, r^j_s\ran + (\partial Q_j-\partial Q_j \lc \partial M) \lc \Clos(\tBcal_{r^j_s}(p_s))\\
& = \lan Q_j, d^{\tM}_{p_s}, r^j_s\ran + S_j\lc \Clos(\tBcal_{r^j_s}(p_s))- S_0\lc \Clos(\tBcal_{r^j_s}(p_s)).
\end{split}
\end{equation*}
We now show that the sequence $\tau^j_i:=[R^j_i] \in \Z_k(M,\partial M)$ satisfies all the properties (i)-(iii) in Lemma \ref{L:interpolation} with $\si_1, \si_2$ replaced by $\si_j, \si_0$. For (i), let $T \in \tau$ be the canonical representative, by Lemma \ref{L:can-rep} and our assumption on $\si_j$, it is easily seen that $\spt(\tau^j_i-\tau)\subset \spt(R^j_i-T)\subset U$ since each $\Clos(\tBcal_{r_i^j}(p_i)) \cap M \subset U$.  For (ii), it follows from
\begin{eqnarray*}
\M(R^j_i) & \leq & \M(S_j)+\sum_{s=1}^i(\|S_0\|-\|S_j\|)( \Clos(\tBcal_{r^j_s}(p_s))) +\sum_{s=1}^i \M(\lan Q_j, d^{\tM}_{p_s}, r^j_s\ran)\\
                   & \leq &  L+i\cdot\frac{\al}{m}+\al\leq L+2\al<L+\de.
\end{eqnarray*}
Finally, (iii) follows from the inequalities below: for each $i=1,\cdots,m+1$,
\begin{eqnarray*}
\M(R^j_i-R^j_{i-1}) & \leq & \M(\lan Q_j, d^{\tM}_{p_i}, r^j_i\ran) +\M (S_j\lc \Clos(\tBcal_{r_i^j}(p_i)))+ \M (S_0\lc\Clos(\tBcal_{r_i^j}(p_i)))\\
                                     & \leq & 5\al=\de; 
\end{eqnarray*}
and
\begin{eqnarray*}
\M(S_0-R^j_m) &\leq & \M\left(\sum_{i=1}^m\lan Q_j, d^{\tM}_{p_i}, r^j_i\ran \right) +\M\left((S_0-S_j)\lc (W\setminus \cup_{i=1}^m \tBcal_{r_i^j}(p_i))\right)\\
                              & \leq & 5\al=\de.
\end{eqnarray*}
This finishes the argument for \textbf{Case 1}.

For \textbf{Case 2}, first note that $\{q\in W: \|V\|(\{q\})> \al\}$ must be a finite set as $\|V\|(W)\leq L$.  We will assume that 
\[ \{q\in W: \|V\|(\{q\})> \al\}=\{q\},\]
and the general case follows by induction on the number of elements in the subset.  

If $q\in W \setminus \partial M$, then by the same argument as in \cite[Lemma 3.7, page 118-121]{Pitts}, we can construct sequences 
\[ S_j=R^j_0, R^j_1, \cdots, R^j_{m_j} \in Z_k(M,\partial M)\]
such that the following hold for each $s=1,\cdots,m_j$:
\begin{itemize}
\item $\spt(R^j_s-S_j)$ is contained in a neighborhood of $q$ inside $U$;
\item $\M(R^j_s)\leq M(S_j)+\de$;
\item $\M(R^j_s-R^j_{s-1})\leq \de$.
\end{itemize}
Moreover, we also have
\begin{itemize}
\item $\M(R^j_{m_j})\leq \M(S_j)$;
\item $\lim_{j\rightarrow \infty}R^j_{m_j}=\lim_{j\rightarrow \infty}S_j\in \si_0$ as currents;
\item $\lim_{j\rightarrow \infty}|R^j_{m_j}|=\|V\|\lc (M\setminus\{q\})$ as varifolds.
\end{itemize}
Therefore we can use the results in {\bf Case 1} to connect $R^j_{m_j}$ to $S_0$, thus obtain a contradiction.

Suppose now $q\in W\cap\partial M$.  We can also find sequences as above by adapting \cite[Lemma 3.7, page 118-121]{Pitts} to our case using Lemma \ref{L:replace-cone} in place of \cite[Lemma 3.5]{Pitts}. We list the necessary details for the sake of completeness.  

Take $\ep>0$, such that $2\ep\|V\|(U)<k \cdot \frac{\al}{2}$. Now choose a neighborhood $Z$ of $q$ as given in Lemma \ref{L:good-nbd} corresponding to the chosen $\ep$. 
By the basic properties of Radon measures (\cite[2.6(2)(d)]{Allard72}), we have (recall Definition \ref{D:annulus})
\begin{itemize}
\item $\lim_{r\rightarrow 0}\|V\| (\Clos(\tBcal^+_r(q)) \setminus\{q\})=0,$
\item $\lim_{j\rightarrow\infty}\|S_j\|\lc \Clos(\A_{s,r}(p))=\|V\|\lc \Clos(\A_{s,r}(p))$
\end{itemize}
 for any $0<s<r$ such that $\|V\| (\partial \A_{s,r}(p))=0$ (which holds for $\mH^n$-a.e. $p \in U \cap \partial M$ by Lemma \ref{no mass on spheres}).
Similar to the arguments in \cite[Lemma 3.7, page 119]{Pitts}, for $j$ large enough, we can find $p_j \in U \cap \partial M$ converging to $q$ and $s_j,r_j>0$ with 
\[ 0<\frac{s_j}{2}<r_j\leq s_j \quad  \text{ and }  \quad \lim_{j\rightarrow\infty}s_j=0 \]
such that all the following hold (using Lemma \ref{no mass on spheres} in place of \cite[Lemma 3.6]{Pitts} and $\tir_{p_j}$ in place of $u(p_j)$ in \cite[page 119]{Pitts}):
\begin{itemize}
\item $\tBcal^+_{r_j/4}(q) \subset \tBcal^+_{r_j/2}(p_j)\subset \tBcal^+_{2s_j}(p_j)$;
\item $\|S_j\| (\tScal_t(p_j))=0$, for any $0\leq t\leq r_j$;
\item $\lan S_j, \tir_{p_j}, r_j\ran$ is an integer rectifiable current \cite[\S 28]{Simon};
\item $\lim_{j\rightarrow\infty}\|S_j\| (\Clos(\A_{s_j/2,2s_j}(p_j)))=0$;
\item $16\|S_j\| (\Clos(\A_{s_j/2,2s_j}(p_j)) )\geq 3 r_j\M (\lan S_j, \tir_{p_j}, r_j\ran)$;
\item $\ep\|S_j\|(U)\leq k \cdot \frac{\al}{2}$  and $\|S_j\|(\Clos(\tBcal^+_{r_j}(p_j)))\geq \al$;
\item $2\M (\lan S_j, \tir_{p_j}, r_j\ran) \cdot \frac{r_j}{k}+\frac{\al}{2}\leq \|S_j\|(\Clos(\tBcal^+_{r_j}(p_j)))$;
\item $\lim_{j\rightarrow\infty}|S_j|\lc (M\setminus \Clos(\tBcal^+_{r_j}(p_j)))=V\lc (M\setminus\{q\})$ as varifolds.
\end{itemize}
Define
\[ \tS_j:=S_j \lc (M\setminus \Clos(\tBcal^+_{r_j}(p_j)))+(\tE_{p_j})_{\sharp}\big(\de_0\ttimes (\tE_{p_j}^{-1})_{\sharp}\lan S_j, \tir_{p_j}, r_j\ran\big).\]
Note that by \cite[28.5(2)]{Simon} (using notions in Lemma \ref{L:local-scaling}),
\[ \partial_1 (S_j\lc \Clos(\tBcal^+_{r_j}(p_j)))=\lan S_j, \tir_{p_j}, r_j\ran.\]
Applying Lemma \ref{L:replace-cone} for $k, \al/2, r_j, \ep, S_j\lc \Clos(\tBcal^+_{r_j}(p_j))$, we can connect $S_j$ to $\tS_j$ by a finite sequence $R^j_i\in Z_k(M, \partial M)$ with $\be=\de$. Therefore, it is easy to check that the sequence $\tau^j_i:=[R^j_i] \in \Z_k(M,\partial M)$ satisfies all the properties (i)-(iii) in Lemma \ref{L:interpolation} with $\si_1, \si_2$ replaced by $\si_j, [\tS_j]$.
Finally, since $\M(\tS_j)\leq \M(S_j)-\frac{\al}{2}$ and $|\tS_j|$ converges to $V$ in $M\setminus\{q\}$ as varifolds, 
we can carry out an induction argument to finish the proof.
\end{proof}


We will show that the Fermi exponential map $\tE_p$ defined in Definition \ref{D:Fermi} satisfies similar local properties as the usual exponential map listed in \cite[3.4(4)(7)]{Pitts}. These properties will be used in Lemma \ref{L:replace-cone}.


\begin{lemma}
\label{L:good-nbd}
Given $q\in \partial M$, and $\ep\in (0, 1)$, there exists a small neighborhood $Z$ of $q$ in $M$, such that all the properties in \cite[3.4(4)]{Pitts} (see also \cite[\S 7.2(a)$\cdots$(f)]{Zhou17}) are satisfied with the usual exponential map $\exp_p$ replaced by $\tE_p$, and the usual distance function $u(p)$ replaced by $\tir_p$ for any $p\in Z \cap \partial M$. In particular, if $p\in Z\cap\partial M$, $W=\tE_p^{-1}(Z)\subset T_p M\cong\R_+^{n+1}$, and $E=\tE_p|_W$, then the following properties hold:
\begin{itemize}
\item[$(a)$] $E$ is a $C^2$ diffeomorphism onto $Z$;

\item[$(b)$] $Z$ is strictly relatively convex;

\item[$(c)$] $(Lip E)^k(Lip E^{-1})^k\leq 2$;

\item[$(d)$] $Lip(\tir_p|_Z)\leq 2$;

\item[$(e)$] If $x\in Z$ and $0\leq \la\leq 1$, then $E\circ \bmu_\la \circ E^{-1}(x)\in Z$;

\item[$(f)$] if $x\in Z$, $0\leq \la\leq 1$, and $\om\in \La_k T_x \tM$ ($k$-th wedge product of $T_x\tM$ \cite[\S 25]{Simon}), then
\begin{equation}
\label{E:shrinking-estimates}
\|D\big( E\circ \bmu_\la \circ E^{-1}\big)_* \om\|\leq \la^k\big(1+\ep(1-\la)\big)\|\om\|.
\end{equation}
Also $\la^k\big(1+\ep(1-\la)\big)\leq 1$ for all $0\leq \la\leq 1$.
\end{itemize}
\end{lemma}

\begin{proof}
These properties follow easily using Lemma \ref{L:Fermi} and \ref{L:r-tilde} and similar arguments as \cite[3.4(4)]{Pitts}. In particular, we want to point out that the proof of (\ref{E:shrinking-estimates}) in \cite[3.4(4)]{Pitts} only uses the fact that the Lipschitz constant of $E^*g$ is bounded, which is satisfied by Lemma \ref{L:Fermi}.
\end{proof}

\begin{lemma}
\label{L:local-scaling}
Using the same notations as in Lemma \ref{L:good-nbd}, assume that $\tBcal^+_r(p) \subset \subset Z$ and $S \in Z_k \big[ \Clos(\tBcal^+_r(p)), \partial \Clos(\tBcal^+_r(p)) \big]$. Denote $\partial_1 S=\partial S\lc (\tScal^+_r(p))$.  Then all the properties in \cite[3.4(7)]{Pitts} (see also \cite[\S 7.2(j)(k)]{Zhou17}) are satisfied with $T, \partial T$ replaced by $S, \partial_1 S$. In particular,
\begin{itemize}
\item[$(g)$] Given $r>0$, $0\leq \la\leq 1$, then by (f),
\[ \M\big((E\circ \bmu_\la \circ E^{-1})_{\#}S\big)\leq \la^k(1+\ep(1-\la))\M(S)\leq \M(S);\]

\item[$(h)$] Denote $S_{\la}=E_{\#}\big(\de_0\ttimes\big[E^{-1}_{\#}(\partial_1 S)-(\bmu_\la \circ E^{-1})_{\#}(\partial_1 S)\big]\big)$, then,
\[ \partial S_{\la}-\big[\partial_1 S-(E\circ \bmu_\la \circ E^{-1})_{\#}\partial_1 S\big] \in Z_{k-1}(Z\cap\partial M, Z\cap\partial M), \]
\[ spt(S_{\la})\subset \A_{\la r, r}(p), \]
\[ \M(S_{\la})\leq (Lip E)^k(Lip E)^{-k}r k^{-1}(1-\la^k)\M(\partial_1 S)\leq 2 r k^{-1}(1-\la^k)\M(\partial_1 S). \]
\end{itemize}
\end{lemma}

\begin{remark}
Note that Fermi coordinates are essentially needed here, as we use the fact that $E^{-1}[\Clos(\tBcal^+_r(p))]$ is a half-ball in $T_pM=\mathbb{R}^{n+1}_+$, so that the cone construction in $S_\la$ still lies in $E^{-1}[\Clos(\tBcal^+_r(p))]$.
\end{remark}

The following lemma is the version of \cite[Lemma 3.6]{Pitts} for manifolds with boundary.

\begin{lemma}
\label{no mass on spheres}
Given $k\in\N$, $2\leq k\leq n$, and $Z$ an open set in Lemma \ref{L:good-nbd}, and $V \in \V_k(M)$ which is rectifiable in $Z$, then for $\mH^{n}$-a.e. $p\in Z\cap\partial M$, we have $\|V\|(\tScal^+_r(p))=0$ for all $r>0$ with $\tBcal^+_r(p) \subset \subset Z$.
\end{lemma}

\begin{proof}
Given $p\in Z\cap \partial M$ and $\tBcal^+_r(p) \subset \subset Z$, since $V$ is rectifiable in $Z$, the approximate tangent space $T_x V$ (see \cite[15.3]{Simon}) is a $k$-dimensional subspace of $T_x(\tScal^+_r(p))$ for $\|V\|$-a.e. $x\in  \tScal^+_r(p)$. Hence, we have
\[ \|V\|(\tScal^+_r(p))=\|V\|\{x\in \tScal^+_r(p) : T_x V \subset T_x(\tScal^+_r(p))\}.\]
Therefore to prove the lemma it suffices to show that for $\mathcal{H}^n$-a.e. $p \in Z \cap \partial M$,
\[ \|V\|\{x\in Z :  T_xV\subset T_x (\tScal^+_{\tir_p(x)}(p))\}=0.\]
By Fubini's theorem,
\begin{displaymath}
\begin{split}
&\ \ \ \int_{p\in Z\cap\partial M}\|V\|\{x\in Z :  T_xV\subset T_x (\tScal^+_{\tir_p(x)}(p))\} \; d\mH^n(p)\\
&=\int_{x\in Z}\mH^n \{p\in Z\cap\partial M :T_xV\subset T_x (\tScal^+_{\tir_p(x)}(p))\} \; d\|V\|(x)\\
\end{split}
\end{displaymath}
Therefore, we just have to show that for $\|V\|$-a.e. $x\in Z$, 
\begin{equation}
\label{E:measure-zero}
\mH^n\{p\in Z\cap\partial M :  T_xV \perp \nabla \tir_p(x)\}=0.
\end{equation}
Since $T_xV$ is a $k$-dimensional subspace in $T_x \tM$, the condition $T_xV\perp \nabla \tir_p(x)$ implies $\nabla \tir_p(x)$ lies in the $(n+1-k)$-dimensional complement $(T_x V)^{\perp} \subset T_x \tM$. Let $\pi(x)$ be the nearest projection of $x$ to $\partial M$. 

\vspace{0.3cm}
\textit{Claim: $\nabla^{\partial M} \tir_p (\pi(x))$ lies in an affine subspace of dimension at most $n+1-k$.}

\textit{Proof of Claim:} Note that by (\ref{E:rtilde}), 
\[ \tir_p^2(x)=(\tir_p\circ\pi)^2(x)+t^2(x). \] 
So $\tir_p(x)\nabla \tir_p(x)=\tir_p(\pi(x))\nabla (\tir_p\circ\pi)(x)+t(x)\nabla t(x)$. Since $t(x)\nabla t(x)$ is independent of $p$, we know that $\tir_p(\pi(x))\nabla (\tir_p\circ\pi)(x)$ lies in an affine subspace in $T_x\tM$ of dimension at most $n+1-k$. It is easily seen that the linear map $\nabla^{\partial M} \tir_p (\pi(x))\to \nabla (\tir_p\circ\pi)(x)$ is injective when $p$ varies. In fact, $\nabla^{\partial M} \tir_p (\pi(x))=\sum_{i, j}g^{ij}(x, 0)\frac{x_j}{\tir_p(\pi(x))}\frac{\partial}{\partial x_i}|_{\pi(x)}$, and $\nabla (\tir_p\circ\pi)(x)=\sum_{i, j}g^{ij}(x, t)\frac{x_j}{\tir_p(\pi(x))}\frac{\partial}{\partial x_i}|_{x}$, so the map is given by $\sum_{i}v^i\frac{\partial}{\partial x_i}|_{\pi(x)}\to \sum_{i, j, k}g^{ij}(x, t)g_{jk}(x, 0)v^k\frac{\partial}{\partial x_i}|_{x}$, which is linear and injective by (\ref{E:g-ij-b}). The claim is now proved.

\vspace{0.3cm}
Note that $\tir_p(\cdot)|_{Z\cap \partial M}$ is the distance function to $p$ on $\partial M$. Therefore, the set of all such $p$ lie inside a submanifold of $Z \cap \partial M$ of dimension at most $n+1-k\leq n-1$, which clearly implies (\ref{E:measure-zero}).
\end{proof}

\begin{lemma}
\label{L:replace-cone}
Given $k\geq 2$, $\de>0$, $r>0$, $\ep \in (0,1)$, $Z$ and $S$ as in Lemma \ref{L:good-nbd} and \ref{L:local-scaling}, assume furthermore the following:
\begin{itemize}
\item $\ep\M(S)<k\de$ and $\frac{2r}{k} \M (\partial_1 S )+\de\leq \M(S)$,
\item $\|S\|(\tScal^+_s(p))=0$, for all $0\leq s\leq r$.
\end{itemize}
Then for any $\be>0$, we can find a sequence of integer rectifiable currents:
\[ S=R_0, R_1, \cdots, R_m \in Z_k (\Clos(\tBcal^+_r(p)), \partial \Clos(\tBcal^+_r(p)))\]
such that for every $i$, we have
\begin{itemize}
\item $\partial R_i\lc \tScal^+_r(p)=\partial_1 S$;
\item $\M(R_i)\leq \M(S)+\be$;
\item $\M(R_i-R_{i-1})\leq \be$;
\item $R_m=(\tE_p)_{\sharp}\big(\de_0\ttimes (\tE_p^{-1})_{\sharp}\partial_1 S\big)$, 
\item $\M(R_m)\leq \frac{2r}{k}\M(\partial_1 S) \leq \M(S)-\delta$.
\end{itemize}
\end{lemma}
\begin{proof}
The proof can be adapted in a straightforward manner from \cite[Lemma 3.5]{Pitts} using Lemma \ref{L:Fermi}, \ref{L:r-tilde}, \ref{L:good-nbd} and \ref{L:local-scaling} in place of \cite[3.4(4)(7)]{Pitts}.
\end{proof}

\bibliographystyle{amsplain}
\bibliography{references}

\end{document}